\documentclass[reqno,11pt]{amsart}
\usepackage{amsmath,amssymb,mathrsfs,amsthm,amsfonts}
\usepackage[inline]{enumitem} 
\usepackage[usenames,dvipsnames]{xcolor}
\usepackage{hyperref}
\usepackage{comment}
\usepackage{filecontents}
\usepackage{array,longtable,booktabs,caption}
\DeclareMathAlphabet{\mathpzc}{OT1}{pzc}{m}{it}
\usepackage{stmaryrd}

\hypersetup{%
	colorlinks=true, linkcolor=blue,
	citecolor=ForestGreen
}
\usepackage[paper=letterpaper,margin=0.9in, top=1in]{geometry}
\usepackage{graphicx}
\usepackage{mathrsfs}
\usepackage{listings}
\newcommand{\ind}{\mathbf{1}}
\usepackage{bbm}
\usepackage{tikz}
\usetikzlibrary{arrows}

\newtheorem{thm}{Theorem}[section]
\newtheorem{cor}[thm]{Corollary}
\newtheorem{prop}[thm]{Proposition}
\newtheorem{lem}[thm]{Lemma}
\newtheorem{lemma}[thm]{Lemma}

\theoremstyle{definition}
\newtheorem{defn}[thm]{Definition}

\newtheorem{rk}[thm]{Remark}

\numberwithin{equation}{section}

\newcommand{\sd}{\textcolor{black}}
\newcommand{\R}{\mathbb{R}}
\newcommand{\V}{\mathcal{L}}
\renewcommand{\hat}{\widehat}
\newcommand{\til}{\widetilde}
\renewcommand{\bar}{\overline}

\renewcommand{\Pr}{\mathbb{P}}
	
\newcommand{\Ex}{\mathbb{E}}
\newcommand{\E}{\mathbb{E}}
\newcommand{\Con}{\mathrm{C}}
\newcommand{\mE}{\mathbf{E}}

\newcommand{\e}{\varepsilon}
\newcommand{\sig}{\tau}
\newcommand{\stt}{skewed}

\title[Multiplicative SHE limit of RWRE]{Multiplicative SHE limit of random walks in space-time random environments}

\author[S.\ Das]{Sayan Das}
\address{S.\ Das,
	Department of Mathematics, University of Chicago,
	\newline\hphantom{\quad \ \ S. Das}
	5734 S.~University Avenue, Chicago, Illinois 60637 USA
}
\email{sayan.das@columbia.edu}

\author[H.\ Drillick]{Hindy Drillick}
\address{H.\ Drillick,
	Department of Mathematics, Columbia University,
	\newline\hphantom{\quad \ \ H. Drillick}
	2990 Broadway, New York, NY 10027 USA
}
\email{hindy.drillick@columbia.edu}

\author[S.\ Parekh]{Shalin Parekh}
\address{S.\ Parekh,
	Department of Mathematics, University of Maryland,
	\newline\hphantom{\quad \ \ S. Parekh}
	4176 Campus Dr, College Park, MD 20742 USA
}
\email{parekh@umd.edu}



\begin{document}
	
	\begin{abstract} We show that under a certain moderate deviation scaling, the multiplicative-noise stochastic heat equation (SHE) arises as the fluctuations of the quenched density of a 1D random walk whose transition probabilities are iid [0,1]-valued random variables. In contrast to the case of directed polymers in the intermediate disorder regime, the variance of our weights is fixed rather than vanishing under the diffusive rescaling of space-time. Consequently, taking a naive limit of the chaos expansion fails for this model, and a nontrivial noise coefficient is observed in the limit. Rather than using chaos techniques, our proof instead uses the fact that in this regime the quenched density solves a discrete SPDE which resembles the SHE. As a byproduct of our techniques, it is shown that independent noise is generated in the limit, in the sense that the prelimiting noise field does not converge to the driving noise of the limiting SPDE.
	\end{abstract}

\maketitle
{
		\hypersetup{linkcolor=black}
		\setcounter{tocdepth}{1}
		\tableofcontents
	}

\section{Introduction}

The (1+1) dimensional Kardar–Parisi–Zhang (KPZ) equation \cite{kpz} is a stochastic partial differential equation (SPDE) given by 
\begin{align}
    \label{def:kpz} \tag{KPZ}
\partial_t\mathcal{H}=\tfrac12\partial_{xx}\mathcal{H}+\tfrac12(\partial_x\mathcal{H})^2+\gamma \cdot \xi, \qquad \mathcal{H}=\mathcal{H}_t(x),
\end{align}
where $\gamma>0$ and $\xi=\xi(t,x)$ is a space-time white noise. The KPZ equation has been studied intensively over the past few decades in both the mathematics and physics literature.
We refer to \cite{FS10, Qua11, Cor12, QS15, CW17, CS20} for some surveys of the mathematical studies of the KPZ equation. 

As an SPDE, \eqref{def:kpz} is ill-posed due to the presence of the non-linear term $(\partial_x\mathcal{H})^2$. One way to make sense of the equation is to consider $\mathcal{U}:= e^{\mathcal{H}}$ which formally solves the stochastic heat equation (SHE) with multiplicative noise:
\begin{equation}\label{she0} \tag{SHE}\partial_t \mathcal{U} = \tfrac12 \partial_{xx} \mathcal{U} + \gamma \cdot \mathcal{U} \xi, \qquad \mathcal{U}=\mathcal{U}_t(x).\end{equation}
The SHE is known to be well-posed and has a well-developed solution theory based on the It\^o integral and chaos expansions \cite{Wal86, BC95, Qua11, Cor12}. 
In this paper, we will consider the solution to \eqref{she0} started with Dirac delta initial data $\mathcal{U}_0(x) = \delta_0(x)$. For this initial data, \cite{flo} established that $\mathcal{U}_t(x) > 0$
for all $t > 0$ and $x \in \R$ almost surely (see also \cite{mue91}). Thus $\mathcal{H}=\log \mathcal{U}$ is well-defined and is called the Cole-Hopf solution of the KPZ equation. This is the notion of solution that we will work with in this paper, and it coincides with other existing notions of solutions \cite{Hai13, Hai14, GIP15, GP17, GJ14, GP18}, under suitable assumptions.

The subject of the present work will be to show that \eqref{she0} arises as the fluctuations limit of a simple model of random walks in random environments (RWRE), with a nontrivial noise coefficient $\gamma$ that has an interesting story behind it.  Consider a family of iid $[0,1]$-valued random variables $\omega_{r,y}$ with $r \in \mathbb{Z}_{\ge 0}$ and $y\in \mathbb Z$, drawn from a common law $\nu$. A random walk $(R(r))_{r\ge 0}$ in the environment $\omega$ starts from $(r,y)=(0,0)$ and goes from $(r,R(r))\to (r+1,R(r)+1)$ with probability $\omega_{r,R(r)}$ and from $(r,R(r))\to (r+1,R(r)-1)$ with probability $1-\omega_{r,R(r)}$.

We will consider the (quenched) random transition probabilities 
\begin{align}
    \label{pwtx}
  \mathsf  P^\omega(r,y) := \mathsf P^\omega(R(r)=y).
\end{align}
{We will show that in a window of size $r^{1 / 2}$ around the spatial location $y\sim r^{3 / 4}$, which we call the moderate deviation regime in this paper,} these transition kernels have fluctuations that are described by the KPZ equation. 
We define for $t\ge 0$ and $x\in \mathbb R$ and $N\in \mathbb N$ the constant 
\begin{align}
\label{cntx}
  C_{N,t,x}:= e^{N^{1/4}x +(N^{1/2}-N\log\cosh(N^{-1/4}))t}  
\end{align}
and then {we define for each $N\ge 1$ and $t\in  N^{-1}\mathbb Z_{\ge 0}$ a random measure $\mathscr U_N(t,\cdot)$ on $\mathbb R$, given by a superposition of Dirac masses as follows }\begin{equation}\label{field}\mathscr U_N(t,\cdot):=  \sum_{x\in N^{-1/2}\mathbb Z-N^{1/4}t} C_{N,t,x} \cdot \mathsf P^{\omega}(Nt, N^{3/4}t + N^{1/2}x)\cdot  \delta_x.
\end{equation}
Here $\delta_x$ denotes a Dirac mass at spatial position $x$, and $a\mathbb Z+b := \{ax+b:x\in\mathbb Z\}$. {The definition of $\mathscr U_N(t,\cdot)$ is extended to $t\in \mathbb R_+$ by linearly interpolating, i.e., taking an appropriate convex combination of the two measures at the two nearest points of $N^{-1}\mathbb Z_{\ge 0}$. We call $\mathscr U_N$ the \textbf{quenched density field}.} In all of our results below, we fix a terminal time $T>0$ which is arbitrary. Our main result is as follows. 

\begin{thm}\label{main}
Let $\omega$ denote an environment as above with $\mathbb E[\omega_{t,x}] = \frac12$ and $ \operatorname{Var}(\omega_{t,x})=\sigma^2\in [0,1/4).$ There is an explicit Banach space $X$ of distributions on $\mathbb R$, which is continuously embedded in $\mathcal S'(\mathbb R)$, such that the collection $\{\mathscr U_N\}_{N \ge 1}$ is tight with respect to the topology of $C([0,T], X)$. Furthermore any limit point as $N\to \infty$ lies in $C((0,T],C(\mathbb R))$ and coincides with the law of the It\^o solution of the SPDE \begin{equation}\label{she}\partial_t \mathcal U(t,x) = \frac12\partial_x^2 \mathcal U(t,x) + \sqrt{\frac{8\sigma^2}{1-4\sigma^2}} \cdot \mathcal U(t,x)\xi(t,x),\;\;\;\;\;\;\;\;\;\;\;\;\;\;\;\;t\ge 0 ,x\in\mathbb R
\end{equation}started with $\delta_0$ initial condition at $t=0$. Here $\xi$ is a standard (Gaussian) space-time white noise on $\mathbb R_+\times \mathbb R.$ 
\end{thm}

{Note that tightness and uniqueness of the limit point means that $\mathscr U_N$ \textit{converges in distribution} to $\mathcal U$ in the aforementioned topology, as $N\to \infty$.}
{The explicit Banach space $X$ is given by a weighted H\"older space $C^{\alpha,\tau}(\mathbb{R})$ with negative exponent $\alpha<-3$ and polynomial weight of degree $\tau>1$, see Definition \ref{ehs}. These exponents are non-optimal but the result is strong enough to imply e.g. convergence in law of $(\mathscr U_N(t,\phi))_{t\in [0,T]}$ to $\big( \int_\mathbb R \mathcal U_t(x)\phi(x)dx\big)_{t\in [0,T]}$ in the topology of $C[0,T]$ for any fixed $\phi\in C_c^\infty(\mathbb R)$, where $\mathscr U_N(t,\phi):=\int_\mathbb R \phi(x)\mathscr U_N(t,dx)$.}

{While our main result above is formulated in terms of the quenched density field, there is another (somewhat weaker) formulation in terms of the \textbf{quenched tail field} which is concerned with the behavior of the tail probability $\mathsf{P}^\omega\big(R(r) \geq y\big)$. We obtain the following:
\begin{thm}[Multi-point convergence of the quenched tail field to the SHE]\label{unifconv} Suppose that the deterministic sequence of vectors $(x_{N,1},\ldots,x_{N,m}) \in \mathbb R^m$ converges as $N\to \infty$ to $(x_1,\ldots,x_m)\in \R^m$. Fix $t_1,\ldots,t_m>0$ and suppose we have sequences $t_{N,i}\in N^{-1}\mathbb Z_{\ge 0}$ such that $t_{N,i}\to t_i$ as $N\to\infty$. Then we have the joint convergence $$\big( N^{1/4} C_{N,t,x} \mathsf{P}^\omega\big(R(Nt) \geq N^{3/4}t_{N,i}+N^{1/2}x_{N,i}\big)\big)_{i=1}^m \stackrel{d}{\to} \big( \mathcal U_{t_i}(x_i)\big)_{i=1}^m,$$ with $\mathcal U$ as in \eqref{she}.
\end{thm}
This will be proved in Section 7.}

\begin{rk}\label{mainrk}
Besides the fact that the limit is given by the multiplicative-noise stochastic heat equation \eqref{she}, two other aspects of the above result are striking.
\begin{enumerate}[leftmargin=18pt]
    \item As \eqref{field} suggests, these KPZ fluctuations are obtained by probing the tail of the probability distribution $\mathsf P^\omega(r,y)$ near spatial location $y= N^{3/4} t+N^{1/2}x$  at time $r = Nt$, where $N$ is very large. It is not so surprising that the spatial variable $x$ needs a factor $N^{1/2}$ at time of order $N$, since this scaling respects the parabolic structure of the limiting equation. However the factor $N^{3/4}$ in front of $t$ is more mysterious and eluded a rigorous understanding until recently. It is now known that $3/4$ is the \textit{unique} exponent at which one expects the KPZ equation to arise. For exponents strictly smaller than $3/4$ one instead obtains a Gaussian field as the limit, as evidenced by works \cite{timo2, yu, timo}. {On the other hand, when the exponent is equal to $1$ (i.e., at the large deviation scale) it was shown in \cite{bc} that under a Beta-distributed random environment, one obtains the Tracy-Widom GUE distribution as the one-point fluctuations. The Tracy-Widom GUE distribution is the one-point marginal of the directed landscape \cite{MQR,dov}, a recently constructed universal scaling limit of models in the KPZ universality class. We expect that for all exponents strictly larger than $3/4$, we have the directed landscape as the scaling limit.} In this sense $3/4$ is the \textbf{crossover exponent} of this model, that is, the unique exponent where one observes the crossover from Gaussian to non-Gaussian behavior. 
    
    \item There is a nontrivial coefficient $\sqrt\frac{8\sigma^2}{1-4\sigma^2}$ appearing in front of the noise, which depends on the weights only through their variance $\sigma^2$. {Note that this coefficient can be factored as $\sqrt{2}\sqrt\frac{4\sigma^2}{1-4\sigma^2}$ where $\sqrt2$ can be viewed as a consequence of the periodicity of the random walk.} We call this coefficient the \textbf{environmental variance coefficient}. 
    {We shall see later in our proofs that this coefficient arises in a nontrivial way through intersection processes of the two-point motion for the RWRE model.} 
     To give some idea for why this is the correct coefficient, note that $1-4\sigma^2$ is the long-term \textit{quenched} diffusion rate of the random walk given the realization of the environment $\omega$. Consequently, if one samples two independent walks in a fixed realization of the environment, then their pairwise intersection process converges in law (upon diffusive scaling) to the local time of Brownian motion \textit{divided} by a factor of $1-4\sigma^2$. This is precisely why we repeatedly see this factor arise in our analysis and in the limiting SPDE.
\end{enumerate}
\end{rk}

One may ask why we have used Dirac masses to define the field in \eqref{field}, instead of linear interpolation. Indeed since the limiting field in \eqref{she} is a continuous function of space and time, one could ask about convergence in a space of continuous functions, rather than tempered distributions. While Dirac masses are computationally easier to work with, this is \textit{not} the only reason for our choice. 
Heuristically, the $\mathscr{U}_N$ density field has a tendency to concentrate more heavily on certain favored sites microscopically, such that macroscopically these favored sites have a Lebesgue density equal to $1-4\sigma^2>0$ (see the discussions at the end of Section \ref{iden} for a more detailed explanation of this fact).
Therefore, it is hopeless to ask for convergence in a topology of continuous functions. 
Indeed, we can rigorously show that it is impossible for convergence to occur in any space of continuous functions, as it would contradict one of the key estimates of the paper. See the discussion just after Proposition \ref{4.1} below.

Given these discontinuity phenomena, a natural question to ask is what happens when one keeps track of the rescaled field $\mathscr U_N$ \textit{together} with the noise variables $\omega_{t,x}$. This leads us to define the prelimiting noise field, {given by a superposition of Dirac masses on space-time:}
\begin{align}
    \label{xi_field}
    \Xi_N(\cdot,\cdot):= (2N^{3/2}\sigma^2)^{-1/2}\sum_{t\in N^{-1}\mathbb Z\cap [0,T]} \; \sum_{\substack{x\in N^{-1/2}\mathbb Z - N^{1/4}t\\\mathrm{same\;parity}}} \big(\omega_{Nt,N^{1/2}x+N^{3/4}t} -\tfrac12\big)\cdot \delta_{(t,x)}.
\end{align}
where the sum is understood as being over those points $(t,x)$ such that $Nt$ and $N^{1/2}x+N^{3/4}t$ are both integers of the same parity. We now address the question of taking a \textbf{joint} limit of $(\Xi_N,\mathscr U_N)$. In principle $\Xi_N$ just tracks the driving noise of the discrete system, so one might be led to believe that it will converge to the driving noise $\xi$ of $\mathcal U$. This is indeed what happens for other systems such as directed polymers when one takes the limit to the stochastic heat equation. {This can be deduced from the methods used in \cite{akq,poly} (those papers prove term-by-term convergence of the respective chaos series, and in particular the first term of the series which encodes the full information of the noise, will converge jointly with the rest of the series)}. In fact, this is not the case in our model, and we have the following refinement of Theorem \ref{main}.

\begin{thm}[Creation of independent noise in the limit] \label{xi2} Let $X\subset \mathcal S'(\mathbb R)$ be as in Theorem \ref{main}. There exists an explicit Banach space $Y\subset \mathcal S'(\mathbb R^2)$ of distributions supported on $[0,T]\times \mathbb R$, such that $(\Xi_N,\mathscr U_N)$ are tight with respect to the topology of $Y\times C([0,T],X)$. Any limit point is of the form $(\xi_1\;,\; \mathcal U),$ where $\xi_1$ is a standard space-time white noise, and $$\partial_t\mathcal U(t,x) = \frac12 \partial_x^2 \mathcal U(t,x) + \bigg(\sqrt{8\sigma^2}\;\xi_1(t,x) + \sqrt\frac{32\sigma^4}{1-4\sigma^2} \;\xi_2(t,x)\bigg)\mathcal U(t,x),\;\;\;\;\;\;\;\;\;\;\; t\ge 0, x\in \mathbb R,$$
with $\mathcal U(0,x) = \delta_0(x)$. Here $\xi_2$ is another space-time white noise,\textbf{ independent} of $\xi_1.$
\end{thm}
This will be proved as Theorem \ref{create} in the main body of the paper. {The proof of the theorem will be non-constructive, and it remains an open problem to identify at the microscopic scale precisely where the independent part $\xi_2$ comes from}. Note that this recovers the earlier result since the sum of squares of the two noise coefficients is exactly the squared noise coefficient from Theorem \ref{main}. As far as we know, this phenomenon of creating independent noise in the limit is quite new. In Section \ref{sec:fail}, we give an intuitive interpretation of the independent part $\xi_2$ in terms of the failure of a naive chaos expansion.
 
Our main results confirm and generalize a physics prediction of \cite{bld}. They conjecture the convergence of the RWRE to the KPZ equation in the moderate deviations regime, but only for the specific choice of Beta-distributed weights $\omega_{t,x}$ which makes the model integrable. Their field is defined slightly differently, but it agrees with the ``quenched tail field" that we study in Section \ref{sec7}. Note that when $\omega \sim $ Beta$(\alpha,\alpha)$ one has $8\sigma^2/(1-4\sigma^2) = 1/\alpha$ which agrees precisely with the prediction of \cite[Section 4.3]{bld}. {Their predictions were based on moment computations using contour integral formulas that are only available when the weights are Beta-distributed. As moments of \eqref{she0} grow too fast to determine the distribution uniquely, the convergence of moments in this case does not imply weak convergence as in Theorem \ref{main}.} Our proof of Theorem \ref{main} will instead be based purely on stochastic analysis and will not involve any exact solvability. The denominator $1-4\sigma^2$ of the noise coefficient will instead be seen to arise naturally from local time considerations coming from pairwise interactions in the 2-point motions, see Theorem \ref{converge}.

Note that the extreme case $\sigma^2=0$, which occurs when all $\omega_{t,x}$ are deterministically $1/2$, corresponds to deterministic $\mathsf P^\omega$ satisfying a discrete heat equation. The other extreme case $\sigma^2=1/4$ occurs {when the $\omega_{t,x}$ are Bernoulli$(1/2)$ which gives coalescing random walks. Consequently, the limit of the $\mathscr U_N$ will be zero \textit{almost surely}, since the probability of a simple random walk path being at location of order $N^{3/4}$ at time of order $N$ is superpolynomially small.} Both of these extreme cases make sense at the level of the limiting SPDE. In the $\sigma^2=0$ case, \eqref{she} becomes the heat equation. The $\sigma^2\to 1/4$ case is slightly more difficult to see: in this case the renormalization effect inherent in the Ito equation \eqref{she} becomes so strong that $\mathcal{U}$ becomes zero everywhere at positive times, and this is indeed the correct limit for the coalescing walks under the above scaling regime.

{Let us now give an overview of the literature related to the fluctuations in RWRE and related models. The model of random walks in dynamical random media gained attention in the context of stochastic flows and their discrete counterparts, see e.g. \cite{lejan, timo2, hw09, hw09b, sss0, sss,sss2,timo}, which studied existence, uniqueness, continuity, and construction of invariant measures for these types of stochastic processes. The question arose naturally to study the fluctuations of the quenched density and related quantities in such models. Gaussian fluctuations in the bulk region were proved in works such as \cite{timo2, yu, timo}. Tracy-Widom fluctuations for the large deviation regime were proved in \cite{bc,mark} for certain exactly solvable models (the Beta-distributed RWRE and a continuous counterpart),} though it still remains wide-open for general weights. The physics papers \cite{ldt2, ldt,bld} conjectured that an explicit crossover from Gaussian to Tracy-Widom fluctuations occurs at spatial location $N^{3/4}$ in these types of stochastic flows and that the KPZ equation should, in fact, appear at this crossover location. In \cite{DDP23} and the present work, we develop general methods to tackle the problem of studying the quenched density at this crossover exponent, for the continuous and discrete models of random walks in random environments respectively. Forthcoming work will explore how to rigorously find the crossover exponent and environmental variance coefficient for more general models of stochastic flows.

{In Theorem \ref{main} we did not aim for the greatest possible generality, but let us briefly describe the sort of generalizations that can be easily proved by adapting our methods, while also commenting on results that are out of reach with our methods. First, rather than assuming that the weights have mean $1/2$, we can assume they have mean $\mu \in (0,1)$. In this case, one needs to change the definition of the constant $C_{N,t,x}$, but one still has weak convergence to \eqref{she0} in the moderate deviation regime. However in the scalings, one will observe an additional shearing of space-time $(t,x) \mapsto (t,x+Vt)$ for appropriate $V\in \mathbb R$, and the stochastic PDE limit will now be given by the solution of $$\partial_t \mathcal U(t,x) = \sqrt{\mu(1-\mu)} \partial_x^2 \mathcal U(t,x)+ \sqrt{\tfrac{2\sigma^2}{\mu(1-\mu) -\sigma^2}}\;\mathcal U(t,x) \xi(t,x),\;\;\;\;\;\;\;\; t\ge 0, x\in \mathbb R.$$
A further generalization is if the weight distribution is not fixed with $N$, but varies in such a way that the mean stays within a window of size $N^{-1/4}$ of some fixed value $\mu$ and the variance converges to some $\sigma^2 \in [0,\mu(1-\mu))$. Then our methods can still show weak convergence to \eqref{she0} with a different shearing coefficient $V$ of space-time. }
   
   {Still another generalization is to look at $\mathsf{P}^\omega(Nt, \beta N^{3/4} t +N^{1/2}x)$ where $\beta >0$, and then look at how the limit depends on $\beta.$ In this case our methods still work, and in fact this is actually equivalent to our result by substituting $(N,t,x) \to (\beta^{-4}N,\beta^4t,\beta^2x)$. The limit will thus be
   $$\partial_t \mathcal U_\beta (t,x) = \frac{1}2 \partial_x^2 \mathcal U_\beta(t,x) + \sqrt\frac{8\beta^2 \sigma^2}{1-4\sigma^2} \;\mathcal U_\beta(t,x) \xi(t,x),\;\;\;\;\;\;\;\; t\ge 0, x\in \mathbb R.$$
   This indicates that one should observe a ``smooth crossover" from Gaussian to Tracy-Widom fluctuations as $\beta$ goes from 0 to $+\infty$, see Remark \ref{mainrk} just above.} 
   
 { Finally, we can also allow jumps of size larger than one. This seems like a more interesting problem because the correct noise coefficient would now depend on more than just the mean and variance of the weights. The latter is \textit{not} something we can easily do by adapting the methods of our paper. In particular, when the jumps have unbounded range, our methods break down completely as one no longer has the power of Azuma's inequality which we repeatedly use. This is something we plan to explore in some depth in a future work.}

\subsection{A failure of the chaos expansion} \label{sec:fail} In this section we discuss interpretations of Theorem \ref{xi2}. One interesting feature of Theorem \ref{main} is the noise coefficient $\sqrt\frac{8\sigma^2}{1-4\sigma^2}$. This coefficient is quite different from the situation of directed polymers, in which one also sees the KPZ equation appear in the fluctuations of a physically complex system driven by iid weights on a discrete space-time lattice. The present result is different for two reasons. Firstly, our weights are \textbf{fixed} with $N$, they are not being tuned so that their variance approaches zero as in the directed polymer setting. Secondly, in the polymer setting if the iid weights have variance $\sigma^2$ before rescaling, then the noise coefficient in the limiting KPZ equation is simply $\sigma$, as opposed to something more complicated as we see in the present work.

In the polymer setting, the method used to prove convergence to \eqref{she0} is the method of \textit{chaos expansion} \cite{akq, poly}. One expands the partition function as a multilinear polynomial of the weights $\omega_{t,x}$ and then proves term-by-term joint convergence of the resultant expression to the term-by-term expansion of the solution of \eqref{she0} given by the Gaussian Fock space structure generated by the noise $\xi$. The natural question to ask for our model is if a chaos expansion exists and if it can be used to prove the convergence stated in Theorem \ref{main}. It turns out that a polynomial chaos expansion is indeed available for the model considered herein, but it is \textbf{{not}} useful in proving the above result if applied naively. We will now illustrate its failure.

Let $p_{\mathsf{RW}}(n,y)$ denote the transition density of a simple symmetric random walk on $\mathbb Z$. Let $\hat \omega_{t,x} = 2\omega_{t,x}-1$. By Lemma 2.1 in \cite{gu} for $n,y$ of the same parity we have

	\begin{align*}
		\mathsf P^{\omega}(n,y)=p_{\mathsf{RW}}(n,y)+\sum_{k=1}^n \sum_{(z_1,z_2,\ldots,z_k)} \phi_k^{(n,y)}(z_1,z_2,\ldots,z_k)\hat \omega_{z_1}\hat \omega_{z_2}\cdots\hat \omega_{z_k}
	\end{align*}
where the summation $\sum_{(z_1,z_2,\ldots,z_k)}$ is over all possible $z_\ell = (i_\ell,j_\ell)$ with $0\le i_1< i_2 < \cdots < i_k \le n-1$ and $j_1,j_2,\ldots,j_k \in \mathbb Z$, and the expansion coefficients are given by
\begin{align*}
	\phi_k^{(n,y)}(z_1,\ldots,z_k) & =  \tfrac1{2^k} p_{\mathsf{RW}}(i_1,j_1) \times \left[p_{\mathsf{RW}}(i_{2}-i_{1}-1,j_{2}-j_{1}-1)-p_{\mathsf{RW}}(i_{2}-i_{1}-1,j_{2}-j_{1}+1)\right]\\ & \hspace{1.5cm} \times \left[p_{\mathsf{RW}}(i_{3}-i_{2}-1,j_{3}-j_{2}-1)-p_{\mathsf{RW}}(i_{3}-i_{2}-1,j_{3}-j_{2}+1)\right] \\ & \hspace{1.5cm} \times \cdots 
	\\ & \hspace{1.5cm} \times \left[p_{\mathsf{RW}}(n-i_{k}-1, y-j_{k}-1)-p_{\mathsf{RW}}(n-i_{k}-1, y-j_{k}+1)\right].
\end{align*}

Now fix $t>0, x\in \R$, and $k \in \mathbb N$. Set $t_0=0, t_{k+1}=t, x_0=0,$ and  $x_{k+1}=x$. Then for \textbf{fixed} $k\in\mathbb N$ one has the convergence in law as $N\to \infty$,

\begin{align*}
	& \frac{C_{N,t,x}}{N^{k/2}}\sum_{(z_1,z_2,\ldots,z_k)} \phi_k^{(Nt,N^{1/2}x+N^{3/4}t)}(z_1,\ldots,z_k)\hat \omega_{z_1}\cdots\hat \omega_{z_k}
 \stackrel{d}{\longrightarrow}(8\sigma^2)^{k/2} \int_{\mathbb R^k \times \Delta_k(t)} \prod_{\ell=1}^{k+1} p_{}(\bar{t}_{\ell},\bar{x}_{\ell}) \xi^{\otimes k}(d\mathbf t, d\mathbf x)
\end{align*}	
where $\bar{t}_{\ell}:=t_{\ell}-t_{\ell-1}$, $\bar{x}_{\ell}:=x_{\ell}-x_{\ell-1}$, and $p(t,x)$ is the continuum heat kernel. The right-hand side is a $k$-fold iterated Wiener-Ito integral where $\xi$ is a standard Gaussian space-time white noise and $\Delta_k(0,t)$ is the simplex of ordered times $0\le t_1\leq \cdots \leq t_k \leq t$. 

The above convergence is fairly straightforward to verify since the weights $\hat \omega$ have variance $4\sigma^2$ and one picks up an extra factor of $2$ from the parity considerations of the discrete lattice on which the $\phi_k$ are supported. In particular the above convergence would suggest that if the limit (in distribution as $N\to\infty$) of the individual terms commuted with the infinite sum (over $k\in\mathbb N$) of those terms, then the coefficient $\sqrt{8\sigma^2/(1-4\sigma^2)}$ in Theorem \ref{main} is incorrect, and that the more naive guess of $\sqrt{8\sigma^2}$ should be the correct one. 

Since Theorem \ref{main} states that this is \textit{not} the case, it suggests that taking the limit of the above chaos expansion does not work. More precisely, it formally converges to the correct equation but with an incorrect noise coefficient that is strictly smaller than the correct one. The reason that this happens is that a nonzero proportion of the $L^2$-mass in the terms of the above chaos expansion escapes into the tails of the series (i.e., terms in the expansion indexed by large values of $k$ that are of unbounded order as $N\to\infty$), which causes a failure of the infinite sum over $k$ to commute with the limit in distribution of each term of the chaos series as $N\to\infty$. In the context of the generalized polynomial chaos convergence result of \cite{poly}, this exactly means that condition (iii) fails when trying to apply their main result Theorem 2.3. Note that the naive limit of the expansion is exactly the limiting object in Theorem \ref{xi2} \textit{without} the $\xi_2$ term. Thus a heuristic interpretation of that theorem is that the amount of extra noise that has escaped into the tails of the chaos expansion is exactly the $\xi_2$ part in Theorem \ref{xi2}.

As a reality check for the above calculation, note that in the extreme case $\sigma^2=1/4$, the weights are Bernoulli distributed. This means that independent random walkers in the environment will coalesce forever upon collision. In particular, the quenched transition densities $\mathsf P^\omega(n,\cdot)$ have all of their mass concentrated at a single point for all $n$ almost surely. Consequently, there will \textit{not} be any nontrivial fluctuation behavior. Nonetheless, the above discrete chaos expansion is valid and formally converges to the stochastic heat equation with noise coefficient $\sqrt{2}$ if one believes that the sum over $k$ commutes with the limit $N\to\infty$. This is already a clear contradiction.

The reason why the above chaos expansion fails is that the variance of the weights is fixed and positive, rather than tending to zero. In other words, we are not in the ``weak noise regime" of \cite{akq, poly}. Note however that when $\sigma$ is very close to $0$, the correct coefficient of $\sqrt{8\sigma^2/(1-4\sigma^2)}$ agrees to first order with the more naive guess of $\sqrt{8\sigma^2}$, which makes sense at a non-rigorous level since $\sigma\to 0$ is a ``weak noise limit."

There are results in higher spatial dimensions where we suspect that independent noise is also created in the limit, see for example \cite{marginal,u1,u3,u2,u4,u5,u6,car,csz_2d}. However, Theorem \ref{xi2} seems to be the first precise result along these lines. To the best of our knowledge, this also seems to be the first result where this phenomenon was observed for a model with just one spatial dimension. A related question is that given the failure of Theorem \ref{main} to hold in a space of continuous functions, it is natural to ask about probing additional \textit{pointwise} fluctuations of the field $\mathscr U_N(t,x)$ as $N\to \infty$, similar to the quenched local limit theorem proved in \cite{ew6}. These pointwise fluctuations would give some clue as to where exactly the additional noise is generated.  We leave this as a future work.

 The fact that the distribution of weights is fixed as a function of $N$ is in stark contrast with the results of \cite{gu} where the authors use a weak noise scaling and consider the large deviation regime. In contrast to the present work, their rescaled field does not concentrate on certain favorite sites. In fact, that regime is more similar to the polymers, because the weights converge as $N\to\infty$ to $\delta_{1/2}$ in a certain precise way, and the variance therefore tends to zero under the diffusive scaling.  Consequently, the chaos expansion method works to prove convergence, and the convergence occurs in the space of continuous functions. Our method of proof is completely different, as the regime considered here seems not to be conducive to such chaos expansion methods as shown above. The work of \cite{dom} studies the weak-noise regime of a related continuum model of transport SPDEs, and Figure 1 therein illustrates an interesting phase transition. This transition clearly distinguishes the weak and strong noise regimes. Our results fall into the upper right quadrant of the figure, which corresponds to strong noise, whereas the results of \cite{dom,gu} fall into the upper left quadrant, which corresponds to weak noise. Independent noise does not appear in the latter regime when taking the limit to \eqref{she0}.

\subsection{Extrema of sticky random walks}\label{extrema}

 Results about the behavior of a probability distribution at its extreme ends can often be translated into information about the extremal behavior of a large number of independent particles sampled from that distribution. Theorem \ref{main} is a result about the tails of the probability distribution $\mathsf P^\omega(Nt,\cdot)$ at a specific location in the tail near $N^{3/4}t$. Theorem \ref{main} combined with some additional calculations thus leads to the following result.

 \begin{thm}[Extremal particle limit theorem] \label{t:max0}
Fix $c,T>0$ and $d\in \mathbb R$. Let $(R^1(r),\ldots,R^k(r))_{r\ge 0}$ be sampled according to the $\omega$-averaged law of $k$ independent random walk particles in the environment $\omega$. Set the number of particles $k=k(N):= \lfloor \exp(\frac12cN^{1/2} +dN^{1/4}+r_N)\rfloor$ where $r_N$ can be any sequence satisfying $r_N=o(N^{1/4})$. Consider any sequence $t_N\in N^{-1}\mathbb Z_{\ge 0}$ such that $t_N\to t>0$ as $N\to \infty$. Then we have convergence in distribution as $N\to \infty$
\begin{align*}
    \max_{1\le i\le k(N)} \big\{N^{-\frac14} R^i(Nt_N)\big\} -a_N(c,d,t_N) \stackrel{d}{\to} \sqrt{\tfrac{t}{c}} \big( G+\log \mathcal{U}_{c}(d)\big),
\end{align*}
where
\begin{align*}
   a_N(c,d,s):= \sqrt{csN} - dN^{\frac14}\sqrt{\tfrac{s}{c}} -\sqrt{\tfrac{c}{s}} \big(r_N-\tfrac14 \log N\big)-\frac{c^{3/2}}{6\sqrt{s}}.
\end{align*}
Here $G$ is a standard Gumbel random variable which is independent of $\mathcal U$, and $(t,x)\mapsto \mathcal U_t(x)$ solves \eqref{she} but with noise coefficient replaced by $\sqrt{\frac{8\sigma^2c}{(1-4\sigma^2)t}}$.
 \end{thm}
This will be proved as Theorem \ref{t:max}. The Gumbel distribution appearing in the maximum is not that surprising, since it often arises in extreme value theory. But the question of how the solution $\mathcal U$ of \eqref{she} appears is more mysterious. Unlike in the case of \textit{independent} random walkers, the $k$ separate particles in the random environment model may see a nontrivial\textit{ interaction} upon collision, and this nontrivial interaction is precisely what gives rise to the non-Gumbel term in the above limit theorem. The nontrivial interaction is sometimes referred to as ``stickiness," since particles that are together tend to stay together for a while before separating.
 
Taking $d=r_N=0$ and $t=1$, we see that the above statement is a result of the maximum of $e^{\frac12c\sqrt{N}}$ many sticky random walk particles at time $N$, {where $N$ is very large}. This is the same as understanding the maximum of $M$ sticky particles when time is of the order $(\log M)^2,$ which then leads to the question of what happens when one looks at the maximum of \sd{$M$ particles at timescales different from $(\log M)^2$. At timescales smaller than $(\log M)^2$, one expects Tracy-Widom fluctuations, with no Gumbel term at all as $M\to \infty$. However this has only been proved when the timescale is $\log M$ in certain exactly solvable cases, see \cite{bc,mark}. At timescales larger than $(\log M)^2$, one expects \textit{purely} Gumbel fluctuations in the $M\to\infty$ limit, as the particles begin to behave independently.} This is evidenced by the Gaussian fluctuation behavior below the crossover exponent, which was explained above. There are numerous open questions related to obtaining a more complete picture in this area, see the physics papers \cite{bld,hass23,ldb,hass2023b,lda}.

From a physical standpoint, the sticky model of particles can be used to describe a certain turbulent flow of particles caused by a nontrivial viscosity of the fluid in which they evolve, see e.g. \cite{gaw}. The physics works \cite{hass23,hass2023b} demonstrate through numerics that the edge of the cloud of particles will exhibit a certain behavior that is different from that of the classical Einstein's model of diffusion. Our main results (Theorems \ref{main} and \ref{t:max0}) thus \sd{serve to make some rigorous progress in this area by showing the correct crossover exponent, but there are still a number of open conjectures.} 

\subsection{Main ideas of the proof}\label{pfidea}

The main technique used in our analysis will be a discrete Girsanov's formula. We illustrate this idea here by computing the first moment of $\mathscr{U}_N$ as defined in \eqref{field} and showing that it converges to the first moment of the SHE.  This computation contains the core idea that will be used in many of the later proofs. For a simple symmetric random walk $R(r)$ on $\mathbb Z$ starting from the origin, {the process} $$e^{\lambda R(r) - r\log\cosh(\lambda)} = \frac{e^{\lambda R(r)}}{\mathbb E[e^{\lambda R(r)}]}$$ is a martingale in $r$, simply because it is a product of $r$ iid strictly positive mean-one random variables $e^{\lambda \xi_i}/\mathbb E[e^{\lambda \xi_i}]$ where the $\xi_i$ are the increments of the process $(R(r))_{r\ge 0}$. Given this fact, let us now fix some Schwartz function $\phi$ and compute the first moment of the pairing $(\mathscr U_N(t,\cdot),\phi)_{L^2(\mathbb R)}=: \mathscr U_N(t,\phi).$ Note that for $t\in \mathbb Z/N$
\begin{align}\notag \mathscr U_N(t,\phi) &= \sum_{x\in N^{-1/2}\mathbb Z-N^{1/4}t} C_{N,t,x} \cdot \mathsf P^{\omega}(Nt, N^{3/4}t + N^{1/2}x)\cdot  \phi(x) \\ &= \notag \sum_{x \in \mathbb Z} C_{N,t,N^{-1/2}(x- N^{3/4} t)}\mathsf P^\omega(Nt,x) \phi\big( N^{-1/2}(x- N^{3/4} t)\big) \\ \label{1stmoment}&= \mathsf E^\omega[ C_{N,t,N^{-1/2}(R(Nt)- N^{3/4} t)}\phi\big( N^{-1/2}(R(Nt)- N^{3/4} t)\big)],
\end{align}where $\mathsf E^\omega$ denotes a quenched expectation operator given the realization of the environment $\omega$. Thus after taking the \textit{annealed} expectation (that is, averaging over all possible environments $\omega$) we have $$\mathbb E[\mathscr U_N(t,\phi)] = \mathbf E[C_{N,t,N^{-1/2}(R(Nt)- N^{3/4} t)}\phi\big( N^{-1/2}(R(Nt)-N^{3/4} t)\big)],$$ where now the expectation on the right is with respect to a simple symmetric random walk path on $\mathbb Z$ {(recall $\mathbb{E}[\omega]=\frac{1}{2}$)}.

Notice that $$C_{N,t,N^{-1/2}(R(Nt)- N^{3/4} t)} = e^{N^{-1/4}R(Nt) - Nt\log\cosh(N^{-1/4})}.$$
We already know that $e^{N^{-1/4}R(Nt) - Nt\log\cosh(N^{-1/4})}$ is a martingale in the $Nt$ variable and thus has mean 1, so it can be interpreted as a change of measure where the annealed law of the increments of $R(r)$ has changed from the usual symmetric law $\frac12(\delta_1+\delta_{-1})$ to the new law given by $\frac1{2\cosh(N^{-1/4})}(e^{N^{-1/4}}\delta_1+ e^{-N^{-1/4}}\delta_{-1})$. 

This new law has mean asymptotically given by $N^{-1/4}+O(N^{-3/4})$. Consequently, under the new law the process $N^{-1/2}(R(Nt)-N^{3/4}t)$ is centered up to a vanishing error term, and by Donsker's principle will converge in law to a standard Brownian motion. In summary, denoting $\widetilde{\mathbf E}$ as \sd{the tilted expectation on the path space of the random walk}, we have that
$$\mathbb E[\mathscr U_N(t,\phi)] = \widetilde{\mathbf E}[ \phi(N^{-1/2}(R(Nt)- N^{3/4}t))]\stackrel{N\to\infty}{\longrightarrow} \mathbf E_{BM}[\phi(B_t)]= \int_{\mathbb R} p(t,x)\phi(x)dx$$ for a standard Brownian motion $B$. Here $p(t,x) = (2\pi t)^{-1/2}e^{-x^2/2t}$ is the standard heat kernel. The right-hand side is indeed equal to $\mathbb E \big[ \int_\mathbb R \mathcal U_t(x)\phi(x)dx \big],$
where $(t,x)\mapsto \mathcal U_t(x)$ solves \eqref{she} with initial condition $\delta_0$. This shows that the first moment of $\mathscr U_N(t,\cdot)$ converges to the first moment of the SHE as desired. While the calculation for higher moments is more complicated, the core idea of using a ``discrete Girsanov transform" is rather similar, and these higher moments will be computed in Section \ref{mom}.

Once the moments are computed using this method, in Section \ref{hopf} we derive a discrete Hopf-Cole transform for the field in \eqref{field}. The martingale observables arising from the Hopf-Cole transform can also be analyzed using the above discrete Girsanov transforms. In particular, this Girsanov trick yields moments and regularity estimates for our observables which eventually leads to tightness estimates for our field. Sufficiently strong estimates using this method will eventually lead us to the proof of Theorem \ref{main}, by showing that any limit point must satisfy the martingale problem for \eqref{she}. \sd{This martingale-based approach circumvents all of the issues with the chaos expansion, as the quadratic variations arising in the martingales automatically take into account the singular behavior of the model on small scales.} Remarkably, this model has the property that the error terms appearing in the discrete martingale equation behave very nicely in relation to the original object itself, which is very rare among KPZ-related models where martingale characterizations have been used, see e.g. \cite{BG97, dembo, GJ14, yang23} where extremely careful analysis was needed to show vanishing of error terms for exclusion-type models.

\begin{rk}[Comparison to sticky Brownian motion] Sticky Brownian motion is obtained by a diffusive scaling of sticky random walks where the law of the weights converges as $N\to\infty$ to $\frac12(\delta_0+\delta_1)$ in a certain precise way \cite[Theorem 2.10]{sss}, which is in a sense the opposite of the weak-noise regime considered in \cite{gu}. Since in the present work the weights are not being scaled, the discrete kernel $\mathsf P^\omega$ is not converging to a sticky Brownian motion kernel. Therefore, our result is \textit{not} simply a ``commutation of limits" with our other work \cite{DDP23} on sticky Brownian motion (just as it is not a corollary of \cite{gu} as explained earlier in Section \ref{sec:fail}). Indeed, the noise amplitude of $\frac{8\sigma^2}{1-4\sigma^2}$ does not correspond in any natural way to the noise amplitude of $2\mu([0,1])^{-1} $ appearing in that paper. However, the overall approach we use here is similar to the one that we developed in the sticky Brownian motion case \cite{DDP23}. Nonetheless, a number of new challenges arise, and several ideas need to be introduced that were not present in that paper. For instance, in Section \ref{sec:girt} the Radon-Nikodym derivatives on the path space are substantially more complicated than that of the continuous case, involving the stochastic exponential of entropy-type observables of the $k$-point motion rather than just simple linear functions of the process. In Section \ref{hopf} we find a discrete Hopf-Cole transform, which is more difficult than the continuous case where the usual heat operator applied to the field immediately yielded a martingale. In Section \ref{iden}, we need to develop several estimates that bound the difference between discrete and continuum heat operators, which was not needed in the previous work.
\end{rk}

\subsection*{Organization} The rest of the paper is organized as follows. In Section \ref{sec:girtt} we describe the discrete Girsanov transform and prove estimates related to the $k$-point motion of sticky random walks. In Section \ref{sec:tilt}, we prove a general weak convergence theorem under the tilted measure. In Section \ref{hopf}, we show that our prelimiting field satisfies a discrete version of the SHE. Section \ref{sec5} is devoted to the analysis of the quadratic variation of the martingale observables in the prelimiting field. Theorems \ref{main} and  \ref{xi2} are proven in Section \ref{iden} using the estimates from the previous sections. Finally, in Section \ref{sec7}, we prove the results related to the extrema of sticky random walks discussed in Section \ref{extrema}.
There is a detailed glossary at the end of this paper that recalls and points to the definitions
of much of the notation introduced elsewhere.

\subsection*{Notations and conventions} Throughout this paper we use $\Con = \Con(a, b, c, \ldots) > 0$ to denote a generic deterministic positive finite
constant depending on $a, b, c, \ldots$ that may change from line to line. We respectively write $\Ex$ and $\mathsf{E}^{\omega}$ for annealed and quenched expectations in the context of random motions in random environments. We use $\mathbf{E}$ to denote expectation under path measures such as Brownian motion. 

\subsection*{Acknowledgements} \sd{We thank Ivan Corwin for suggesting the problem, for his feedback on the paper, and for many insightful discussions about the project. We thank Yu Gu and Li-Cheng Tsai for illuminating discussions that eventually led us to derive Theorem \ref{xi2}. We thank Guillaume Barraquand, Jeremy Quastel, Rongfeng Sun, and B\'alint Vir\'ag for useful discussions about the context and possible refinements of our results.} {We are grateful to the two anonymous referees for their many valuable suggestions.} The project was initiated during the authors’ participation in the “Universality and Integrability in Random Matrix Theory and Interacting Particle Systems” semester program at MSRI in the fall of 2021. The authors thank the program organizers for their hospitality and acknowledge the support from NSF DMS-1928930. SD and HD’s research was partially supported by Ivan Corwin’s NSF grant DMS-1811143, the Fernholz Foundation’s “Summer Minerva Fellows” program, and also the W.M. Keck Foundation Science and Engineering Grant on “Extreme diffusion.” HD was also supported by the NSF Graduate Research Fellowship under Grant No. DGE-2036197.

\section{Transformations and estimates for the $k$-point motion} \label{sec:girtt}

In this section, we develop the basic framework of our proof. As mentioned in the introduction, our proof relies on a certain Girsanov-type transform for random walks. In Section \ref{sec:girt} we describe this transform that we will use repeatedly in our later analysis. In Section \ref{sec:sbmest}, we collect several estimates related to the random walks.

	\subsection{Girsanov transforms} \label{sec:girt} We begin with some necessary notation and definitions. Throughout this paper, we assume $\nu$ is a probability measure on $[0,1]$. We define $\mu:=\int_{[0,1]} x\nu(dx)$. We set $\sigma^2 = \int_{[0,1]}(x-\mu)^2\nu(dx),$ and sometimes we may write $\mu_\nu,\sigma_\nu^2$ to emphasize the dependence on $\nu$.  While we take $\mu_{\nu} = 1/2$ in the setting of Theorem \ref{main}, we will need this more general framework for our proofs. In our analysis, we shall often consider a special sequence of measures obtained by small perturbation of $\nu$. We introduce this below.

\begin{defn}[Skewed sequence of measures] \label{def:stt}
    Fix any probability measure $\nu$ on $[0,1]$ of mean 1/2. Assume $\omega$ has law $\nu$. Let $\nu_N^*$ be the law of $\min\{1, \omega +d_N\}$ where $d_N= \frac12 N^{-1/4}+o(N^{-1/4})$ are constants chosen so that the expectation $\mu_{\nu_N^*}$ is exactly equal to $\rho_N$ where 
    \begin{align}\label{def:rhon} \rho_N:= \frac{e^{N^{-1/4}}}{2\cosh(N^{-1/4})}=\frac12 + \frac12 N^{-1/4} +O(N^{-3/4}).\end{align}
     The choice of $d_N$ implies that $\lim_{N\to \infty} \sigma_{\nu_N^*}^2 =\sigma_\nu^2$. We shall call  $\{\nu_N^*\}_{N\ge 1}$ the \stt\ sequence of measures corresponding to $\nu$.
\end{defn}

 \begin{defn}\label{def:prw}
      For each $k\in \mathbb N$ we denote $\mathbf P_{RW_\nu^{(k)}}$ to be the annealed law on the canonical space $(\mathbb Z^k)^{\mathbb Z_{\ge 0}}$ of $k$ independent walks sampled from the environment $\omega$ whose weights are distributed according to $\nu$, all started from 0. 
 \end{defn}

 	\begin{itemize}[leftmargin=18pt]
		\setlength\itemsep{0.5em}

            \item We denote by $\mathbf R = (R^1,\ldots,R^k)$ the canonical process on $(\mathbb Z^k)^{\mathbb Z_{\ge 0}}$. We will usually index such discrete-time processes by a time-variable $r\in \mathbb N$, e.g., $\mathbf R = (\mathbf R(r))_{r\ge 0}.$    We further define
            \begin{align}\label{def:overlap}
    \Gamma^{(v)}:=\big\{r\in \mathbb{Z}_{\ge0}  \mid \#\{R^1(r),\ldots,R^k(r)\}=v\big\}, \qquad \Gamma^{(\le v)}:=\bigcup_{u=1}^{v} \Gamma^{(u)}
\end{align}
to be the sets of times that we have exactly $v$ and no more than $v$ distinct particles, respectively.
\item For each $k \in \mathbb N$ and $1\leq i\neq j \leq k$ we define the functional $V^{ij}:(\mathbb Z^k)^{\mathbb Z_{\ge 0}} \to \mathbb Z^{\mathbb Z_{\ge 0}}$ by 
  \begin{align}
      \label{vijai}
      V^{ij}(r):= \sum_{s=0}^{r-1} \ind_{\{R^i(s)=R^j(s)\}}.
  \end{align}
 We remark that the processes $V^{ij}$ are predictable with respect to~the canonical filtration on $(\mathbb Z^k)^{\mathbb Z_{\ge 0}}$. 
            \item For a martingale $(M(r))_{r\ge 0}$ adapted to the canonical filtration of $(\mathbb Z^k)^{\mathbb Z_{\ge 0}},$ we will denote by $[M](r):= \sum_{s=1}^{r} (M(s)-M(s-1))^2$ its optional quadratic variation.
            \item All objects in this section should implicitly be understood as being adapted path functionals of $\mathbf R$, though we will not write this out explicitly, e.g. $V^{ij}(r) = V^{ij}(\mathbf R;r), \mathpzc M^\lambda(r) = \mathpzc M^\lambda(\mathbf R; r).$
	\end{itemize}

  \begin{figure}[ht]
    \centering
    
    			\begin{tikzpicture}[scale = 2.5]
            \begin{scope}
			\draw[thick, ->](0.1,0.1) --(.5,.5);
			\draw[thick, ->](0.1, -0.1) -- 
			(0.5,- 0.5);
            \draw[thick, ->](-0.1, -0.1) -- 
			(-0.5,- 0.5);
            \draw[thick, ->](-0.1, 0.1) -- 
			(-0.5, 0.5);

			\node at (0, 0) {$(y,y)$};
            \node at (0.6, 0.8) {$(y+1,y+1)$};
            \node at (-0.6, -0.8) {$(y-1,y-1)$};
            \node at (-0.6, 0.8) {$(y-1,y+1)$};
            \node at (0.6, -0.8) {$(y+1,y-1)$};

            \node at (1.5, 0.2) {$\E[\omega^2] = \mu^2 + \sigma^2$};
            \node at (-1.5, -0.2) {$\E[(1-\omega)^2] =(1- \mu)^2 + \sigma^2$};
            \node at (-1.5, 0.2) {$\E[\omega(1-\omega)] = \mu(1- \mu) - \sigma^2$};
             \node at (1.5, -0.2) {$\E[\omega(1-\omega)] = \mu(1- \mu) - \sigma^2$};
        \end{scope}

         	\end{tikzpicture}
			
    \caption{Markov chain transition diagram of the two-point motion $(R^i, R^j)$ at points along the diagonal of $\mathbb Z^2$, where $\omega$ denotes a random variable sampled from $\nu$.} 
\label{fig:markov_transitions}
\end{figure}
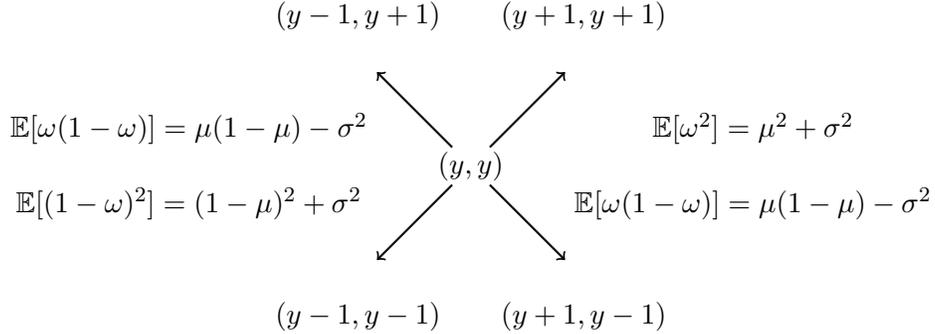

Note that when $\nu=\delta_{a}$ for some $a\in [0,1]$, the path measure $\mathbf P_{RW_\nu^{(k)}}$ is just the law of $k$ independent random walks on $\mathbb Z$ where the increments have mean $2a-1$. However, for $\nu\neq \delta_a$, the law is more complicated: while the $k$ particles still behave independently when they are apart, there can be a nontrivial interaction between them when they collide. As an extreme case, consider $\nu = p\delta_0 + (1-p) \delta_1$ with $p\in [0,1]$. In this setting, the particles coalesce forever upon collision.

If we only look at a single particle, then its trajectory is just that of a nearest-neighbor random walk with right jump rate given by $\E[\omega] = \mu_{\nu}$ and left jump rate given by $\E[1- \omega] = 1- \mu_{\nu}$. So for all $1 \leq i \leq k$, the one-point motion $R^i(r)$ is a random walk with mean $(2\mu_\nu - 1)$ under $\mathbf P_{RW^{(k)}_{\nu}}$. It follows that the process $R^i(r) - (2\mu_\nu - 1)r$ is a mean-zero random walk. 

Next, we examine the interactions between two particles. For $1 \leq i  <  j \leq k $, the two-point motion $(R^i, R^j)$ is a random walk whose transition probabilities at time $t$ differ depending on whether $R^i(t) = R^j(t)$ or not. If  $R^i(t) \neq R^j(t)$, then $R^i$ and $R^j$ behave as two independent random walks. If $R^i(t) = R^j(t) = y$, then the two walks are no longer independent since they are sampled using the same environment variable $\omega_{t, y}$. See Figure \ref{fig:markov_transitions} for the transition probabilities in this case. 

Using these transition probabilities, one can check that the following process \begin{equation}\label{m} \mathcal M(r):= 4\big[\mu_\nu(1-\mu_\nu)-\sigma_\nu^2\big]V^{ij}(r) - |R^i(r)-R^j(r)|
     \end{equation}
is a $\mathbf P_{RW_\nu^{(k)}}$-martingale. Rewriting this as 

\begin{equation} |R^i(r)-R^j(r)| =  4\big[\mu_\nu(1-\mu_\nu)-\sigma_\nu^2\big]V^{ij}(r) - \mathcal M(r),
\end{equation}
we can view this as a discrete version of Tanaka's formula, where the term $4\big[\mu_\nu(1-\mu_\nu)-\sigma_\nu^2\big]V^{ij}(r)$ describes the intersection times of the two walks and corresponds to the local time term in the continuous version of Tanaka's formula. If we take $\mu_{\nu} = 1/2$, as we will in the setting of Theorem \ref{main}, the coefficient in front of $V^{ij}(r)$ becomes $1- 4 \sigma^2$, which as discussed in the introduction will show up in the denominator of the noise coefficient in \eqref{she}.

 Note that if $(X(r))_{r\ge 0}$ is any sequence of random variables with finite exponential moments defined on some probability space $(\Omega,\mathcal F,\mathbf P)$ and adapted to a filtration $(\mathcal F_r)_{r\ge 0}$ on this space, then the process \begin{equation}\label{aft}\exp\bigg( X(r)- \sum_{s=1}^{r} \log \mathbf E[e^{X(s)-X(s-1)}|\mathcal F_{s-1}]\bigg)\end{equation} is a martingale in $r\ge 0$, with respect to~the same filtration. Henceforth, our probability space will always be $\Omega =(\mathbb Z^k)^{\mathbb Z_{\ge 0}},$ equipped with its canonical filtration. For $\lambda>0$, if we apply \eqref{aft} to the process $X:=\lambda(R^1+\cdots+R^k)$, we see that for each $\lambda>0$ 
 \begin{align}\label{m_n} \mathpzc M^{\lambda}(r)&:=\exp\bigg( \lambda\sum_{j=1}^k R^j(r) - \sum_{\ell =0}^{r-1} f_{\ell}^{\lambda,k}\bigg)
\end{align}
 is a $\mathbf P_{RW_\nu^{(k)}}$-martingale for the $k$-point motion where 
 \begin{align}
     \label{def:flk}
     f_\ell^{\lambda,k} := f^{\lambda,k,\nu}\big(R^1(\ell),\ldots,R^k(\ell)\big), \text{ and } f^{\lambda,k,\nu}(x^1,\ldots, x^k) := \log\mathbf E_{RW_\nu^{(k)}}^{(x^1,\ldots,x^k)}[e^{\lambda\sum_{j=1}^k(R^j(1)-x^j)}].
 \end{align}
Here $\mathbf E_{RW_\nu^{(k)}}^{(x^1,\ldots, x^k)}$ denotes expectation of the Markov chain when started from $(x^1,\ldots,x^k).$  $f_{\ell}^{\lambda,k}$ admits a simple expression for $\ell\in \Gamma^{(k)}, \Gamma^{(k-1)}$ where $\Gamma^{(v)}$ is defined in \eqref{def:overlap}. Indeed,  for $\ell\in \Gamma^{(k)}$, $f_{\ell}^{\lambda ,k}=k\log \cosh(\lambda)$, whereas for $\ell\in \Gamma^{(k-1)}$ we have
$$f_{\ell}^{\lambda,k}=k\log\cosh(\lambda)-g(\lambda,\sigma^2).$$
where $g(\lambda,\sigma^2):=\log (\frac{1+4\sigma^2}{2} \cosh(2\lambda)+\frac{1-4\sigma^2}{2})-2\log\cosh (\lambda)$, as may be verified by an inspection of Figure \ref{fig:markov_transitions}. This fact will be very useful in many calculations below.

{Before discussing the importance of the martingale defined in \eqref{m_n}, we record a useful estimate for the $f^{\lambda,\kappa,\nu}$.

\begin{lem}\label{prop:bound}
Assume $\mu_\nu=1/2$. For any $k\in\mathbb N$, we have the bound
\begin{equation*}\sup_{(x^1,\ldots,x^k)\in \mathbb Z^k}\;\;\;\sup_{|\lambda|<1}\;\;\;\;
\lambda^{-2}|f^{\lambda,k,\nu}(x^1,\ldots,x^k)|<\infty.
\end{equation*}  
\end{lem}
\begin{proof}
Note that the function $f^{\lambda,k,\nu}:\mathbb Z^k\to \mathbb R$ is defined using only the first step of the Markov chain, so that $f^{\lambda,k,\nu}(x^1,\ldots,x^k)$ depends only on the first $k$ moments of the measure $\nu$ and the number of particle clusters in the set $\{x^1,\ldots,x^k\}$, as well as the exact number of particles present in each cluster. This means that the supremum over $\mathbb Z^k$ is really a supremum over some \textbf{finite} set of partitions of $k$. Thus it suffices to show for each \textit{individual} $(x^1,\ldots,x^k)\in\mathbb Z^k$ that the sup over $\lambda$ is finite. Since the $R^j-x^j$ are centered if $\mu_\nu=1/2$, we have that $\lambda^{-2}\log\mathbf E_{RW_\nu^{(k)}}^{(x^1,\ldots,x^k)}[e^{\lambda\sum_{j=1}^k(R^j_1-x^j)}]$ actually converges as $\lambda\to 0$ to the quantity $\frac12\mathbf E_{RW_\nu^{(k)}}^{(r^1,\ldots,r^k)}\big[\big(\sum_{j=1}^k(R^j(1)-r^j)\big)^2\big]$. 
\end{proof}

We use the martingale $\mathpzc M^{\lambda}(r)$  in \eqref{m_n} to define the following measure.

\begin{defn}\label{q}
    Define the measure $\mathbf Q^\lambda_{RW^{(k)}_\nu}$ on the canonical space $(\mathbb Z^k)^{\mathbb Z_{\ge 0}}$ to be given by $$\frac{d\mathbf Q^\lambda_{RW^{(k)}_\nu}}{d\mathbf P_{RW^{(k)}_\nu}}\bigg|_{\mathcal F_r} =  \mathpzc M^{\lambda}(r),$$ where $\mathcal F_r$ is the $\sigma$-algebra on $(\mathbb Z^k)^{\mathbb Z_{\ge 0}}$ generated by the first $r$ coordinate maps, and $ \mathpzc M^\lambda(r)$ is the martingale given by \eqref{m_n}.
\end{defn}

It turns out that proving KPZ convergence for the quenched kernels will amount to studying the canonical process under the measures $\mathbf Q^\lambda_{RW^{(k)}_\nu},$ with $\lambda=N^{-1/4}$. The following proposition is the main result of this section. It essentially tells us that $\mathbf Q^{N^{-1/4}}_{RW^{(k)}_\nu}$ is similar to the measure $\mathbf P_{RW^{(k)}_{ \nu_N}}$, albeit with a different underlying measure $\nu_N$ that has a different mean and variance than the original measure $\nu$, {but which is close enough for our purposes in studying the diffusive limits and the intersection times.}

\begin{prop}\label{g.exists}
    Consider any probability measure $\nu$ on $[0,1]$ with $\mu_\nu=1/2$ and $\sigma_\nu^2 \in [0,1/4)$. Fix $k\in \mathbb N$, and let $\mathbf Q^\lambda_{RW^{(k)}_\nu}$ be as in Definition \ref{q}. Let $\{\nu_N^*\}_{N\ge 1}$ be the \stt\ sequence of measures corresponding to $\nu$ defined in Definition \ref{def:stt}. For all $r\ge 0$, one has \begin{equation}\label{g}\frac{d\mathbf Q^{N^{-1/4}}_{RW^{(k)}_\nu}}{d\mathbf P_{RW^{(k)}_{\nu_N^*}}}\bigg|_{\mathcal F_r}(\mathbf R) = \exp\left(\widetilde{\mathcal{G}_N}(r)\right):=\exp \bigg(\mathcal G_N(r) -\sum_{s\leq r} \log \mathbf E_{RW^{(k)}_{\nu_N^*}}[e^{\mathcal G_N(s)-\mathcal G_N(s-1)}|\mathcal F_{s-1}] \bigg),\end{equation} where $\mathcal G_N$ is a family of $\mathbf P_{RW^{(k)}_{\nu_N^*}}$-martingales satisfying the optional variation bound
    \begin{align}
        \label{g1}
        [\mathcal G_N](r+1) -[\mathcal G_N](r) \leq CN^{-1/2} \sum_{i<j} \big(V^{ij}(r+1)-V^{ij}(r)\big),
    \end{align}
		for all $r \ge 0$, where $V^{ij}$ are defined in \eqref{vijai}, and $C$ is a constant independent of $N,r$.
\end{prop}

\begin{proof} For clarity the proof is divided into two steps. 

\medskip

\textbf{Step 1.} In this step, we will construct a family of $\mathbf{P}_{RW^{(k)}_{\nu_N}}$ martingales $\mathcal{G}_N$ satisfying \eqref{g}. Let $\{\nu_N\}_N$ be any arbitrary sequence of probability measures on $[0,1]$. Write the Radon-Nikodym derivative in \eqref{g} as a product of two terms: \begin{equation}\label{2terms}\frac{d\mathbf Q^{N^{-1/4}}_{RW^{(k)}_\nu}}{d\mathbf P_{RW^{(k)}_\nu}}\bigg|_{\mathcal F_r}(\mathbf R) \cdot \frac{d\mathbf P_{RW^{(k)}_{\nu}}}{d\mathbf P_{RW^{(k)}_{\nu_N}}}\bigg|_{\mathcal F_r}(\mathbf R).\end{equation}
    
    The second term in \eqref{2terms} is the quotient of two Markov chain transition probabilities. This can be written in the form $\exp \big(\mathcal D_N(r)\big)$ where $\mathcal D_N(0)=0$ and where $\mathcal D_N(r+1)-\mathcal D_N(r)$ can be described deterministically from the path $\mathbf R$ as follows: 
    
    Assume that at time $r$, the $k$ particles of $\mathbf R(r)\in \mathbb Z^k$ form $v=v(r)$ disjoint groups, with all particles in each distinct group at the same site in $\mathbb Z$. Assume that the $v$ respective groups contain $n_1,\ldots,n_v$ respective particles, where $n_1+\cdots+n_v=k$. Suppose that from time $r$ to $r+1$, $b_j$ of the $n_j$ particles in the $j^{th}$ group go up one step and $n_j-b_j$ particles go down one step for each $1\le j \le v$.
    Then we have the entropy formula: \begin{align}
        \label{def:diff}
        \mathcal D_N(r+1)-\mathcal D_N(r) = \sum_{j=1}^v \log\left[m_{b_j,n_j-b_j}\right], \quad \mbox{where }m_{b,n-b}:=\frac{\int_{[0,1]} x^{b} (1-x)^{n-b}\nu(dx)}{\int_{[0,1]} x^{b} (1-x)^{n-b}\nu_N(dx)}.
    \end{align}
    The exponential of the r.h.s.~of the above equation is precisely the quotient of the two relevant Markov chain transition probabilities from $\mathbf R(r) \to \mathbf R(r+1)$. Although $\mathcal D_N$ is not a $\mathbf P_{RW_{\nu_N}^{(k)}}$-martingale, we can define another process $\widetilde{\mathcal D}_N$ with $\widetilde{\mathcal D}_N(0)=0,$ and 
    \begin{align}
        \label{def:diff2}
        \widetilde{\mathcal D}_N(r+1)-\widetilde{\mathcal D}_N(r):= \mathcal D_N(r+1)-\mathcal D_N(r)- \mathbf E_{RW_{\nu_N}^{(k)}} [\mathcal D_N(r+1)-\mathcal D_N(r)|\mathcal F_r].
    \end{align}
     By construction, $\widetilde{\mathcal D}_N$ is a $\mathbf P_{RW_{\nu_N}^{(k)}}$-martingale. Utilizing the above relation, we observe that
    \begin{align*}
        \log \mathbf E_{RW^{(k)}_{\nu_N}}\left[e^{\widetilde{\mathcal D}_N(r+1)-\widetilde{\mathcal D}_N(r)}|\mathcal F_{r}\right]  & =\log \mathbf E_{RW^{(k)}_{\nu_N}}\left[e^{{\mathcal D}_N(r+1)-{\mathcal D}_N(r)-\mathbf E_{RW_{\nu_N}^{(k)}} [\mathcal D_N(r+1)-\mathcal D_N(r)|\mathcal F_r]}|\mathcal F_{r}\right] \\ & = \log \mathbf E_{RW^{(k)}_{\nu_N}}\left[e^{{\mathcal D}_N(r+1)-{\mathcal D}_N(r)}|\mathcal F_{r}\right]- \mathbf E_{RW_{\nu_N}^{(k)}} [\mathcal D_N(r+1)-\mathcal D_N(r)|\mathcal F_r] \\ & = - \mathbf E_{RW_{\nu_N}^{(k)}} [\mathcal D_N(r+1)-\mathcal D_N(r)|\mathcal F_r],
    \end{align*}
    where in the last line we used that $e^{\mathcal{D}_N}$ is a $\mathbf P_{RW_{\nu_N}^{(k)}}$-martingale. This implies 
    \begin{align}
        \label{terma}
        \frac{d\mathbf P_{RW^{(k)}_{\nu}}}{d\mathbf P_{RW^{(k)}_{\nu_N}}}\bigg|_{\mathcal F_r}(\mathbf R) =\exp \big(\mathcal D_N(r)\big) = \exp\bigg( \widetilde{\mathcal D}_N(r) - \sum_{s\leq r} \log \mathbf E_{RW^{(k)}_{\nu_N}}[e^{\widetilde{\mathcal D}_N(s)-\widetilde{\mathcal D}_N(s-1)}|\mathcal F_{s-1}] \bigg).
    \end{align}
 On the other hand, by Definition \ref{q}, we have
 \begin{align}
     \label{termb}
     \frac{d\mathbf Q^{N^{-1/4}}_{RW^{(k)}_\nu}}{d\mathbf P_{RW^{(k)}_\nu}}\bigg|_{\mathcal F_r}(\mathbf R)=\exp \bigg(\mathcal H_N(r) -\sum_{s\leq r} \log \mathbf E_{RW^{(k)}_{\nu}}[e^{\mathcal H_N(s)-\mathcal H_N(s-1)}|\mathcal F_{s-1}] \bigg)
 \end{align}
 where $\mathcal H_N(r) = N^{-1/4}(R^1_r +\cdots+ R^k_r).$ Note that the process $\widetilde{\mathcal H}_N(r):= \mathcal H_N(r) - rk N^{-1/4} (2\mu_{\nu_N}-1)$ is a $\mathbf P_{RW^{(k)}_{\nu_N}}$-martingale since each of the processes $R^i(r) - r(2\mu_{\nu_N}-1)$ for $1\le i \le k$ are mean-zero random walks under $\mathbf P_{RW^{(k)}_{\nu_N}}$ and replacing $\mathcal H$ by $\widetilde{\mathcal H}$ leaves the above expression unchanged. Let us set $\mathcal G_N := \widetilde{\mathcal H}_N + \widetilde{\mathcal D}_N.$  Clearly 
 $\mathcal{G}_N$ is a $\mathbf P_{RW^{(k)}_{\nu_N}}$-martingale. 
Multiplying both sides of \eqref{terma} and \eqref{termb} we get
\begin{align*}
  & \frac{d\mathbf Q^{N^{-1/4}}_{RW^{(k)}_\nu}}{d\mathbf P_{RW^{(k)}_{\nu_N}}}\bigg|_{\mathcal F_r}(\mathbf R) \\ & = \exp\left(\mathcal{G}_N(r)-\sum_{s\le r} \log \bigg[\mathbf{E}_{RW^{(k)}_\nu}\big[e^{\widetilde{\mathcal{H}}_N(s)-\widetilde{\mathcal{H}}_{N}(s-1)}\mid \mathcal{F}_{s-1}\big]\mathbf{E}_{RW^{(k)}_{\nu_N}}\big[e^{\widetilde{\mathcal{D}}_N(s)-\widetilde{\mathcal{D}}_N({s-1})}\mid \mathcal{F}_{s-1}\big]\bigg)\right) \\ & = \exp \bigg(\mathcal G_N(r) -\sum_{s\leq r} \log \mathbf E_{RW^{(k)}_{\nu_N}}[e^{\mathcal G_N(s)-\mathcal G_N(s-1)}|\mathcal F_{s-1}] \bigg),
\end{align*}
    where the last equality follows by observing that for each $s\le r$, by tilting we have
\begin{align*}
& \mathbf{E}_{RW^{(k)}_\nu}\left[e^{\widetilde{\mathcal{H}}_N(s)-\widetilde{\mathcal{H}}_{N}(s-1)}\mid \mathcal{F}_{s-1}\right] \\ & = \mathbf{E}_{RW^{(k)}_{\nu_N}}\left[e^{\widetilde{\mathcal D}_N(s) - \widetilde{\mathcal D}_N(s-1)- \log \mathbf E_{RW^{(k)}_{\nu_N}}[e^{\widetilde{\mathcal D}_N(s)-\widetilde{\mathcal D}_N(s-1)}|\mathcal F_{s-1}]+\widetilde{\mathcal{H}}_N(s)-\widetilde{\mathcal{H}}_{N}(s-1)}\mid \mathcal{F}_{s-1}\right] \\ & = \frac{\mathbf{E}_{RW^{(k)}_{\nu_N}}\left[e^{{\mathcal{G}}_N(s)-{\mathcal{G}}_N(s-1)}\mid \mathcal{F}_{s-1}\right]}{\mathbf E_{RW^{(k)}_{\nu_N}}[e^{\widetilde{\mathcal D}_N(s)-\widetilde{\mathcal D}_N(s-1)}|\mathcal F_{s-1}]}.
\end{align*}
  This verifies \eqref{g}. 
  
  \medskip

  \textbf{Step 2.} In this step, we will show \eqref{g1}. This is where we set $\nu_N=\nu_N^*$, the \stt\ sequence of probability measures corresponding to $\nu$ (see Definition \ref{def:stt}). With this choice, we claim that for all $r$
\begin{align}\label{claim1}
   & [\mathcal G_N](r+1) -[\mathcal G_N](r) = \ind_{r \in \Gamma^{(\le k-1)}} \big(\mathcal G_N(r+1)-\mathcal G_N(r)\big)^2, \\ & \big(\mathcal G_N(r+1)-\mathcal G_N(r)\big)^2 \le CN^{-1/2}, \label{claim2}
\end{align}
  where $\Gamma^{(\le k-1)}$ is defined in \eqref{def:overlap}. From the definition of the $V^{ij}$ functionals from  \eqref{vijai}, we have $\ind_{r \in \Gamma^{(\le k-1)}} \leq \sum_{i<j} \big(V^{ij}(r+1)- V^{ij}(r)\big)$. Inserting this bound and the bound in \eqref{claim2} in the right-hand side ~of \eqref{claim1} leads to \eqref{g1}.

    Recall the sets $\Gamma^{(v)}$ from \eqref{def:overlap}.
  We claim that 
  \begin{align}
      \label{dheq}
      \widetilde{\mathcal D}_N(r+1)-\widetilde{\mathcal D}_N(r) 
  =-\frac{N^{1/4}}2\log \frac{\mu_{\nu_N^*}}{1-\mu_{\nu_N^*}}(\widetilde{\mathcal H}_N(r+1)-\widetilde{\mathcal H}_N(r)),
  \end{align}
  for all $r\in \Gamma^{(k)}$. We will prove \eqref{dheq} below, but we first complete the rest of the proof, assuming \eqref{dheq} to be true.
  Since the $\nu_N^*$ have mean exactly equal to $\mu_{\nu_N^*} = \rho_N=e^{N^{-1/4}}(2\cosh(N^{-1/4}))^{-1}$, we get that $\frac{N^{1/4}}2\log \frac{\mu_{\nu_N^*}}{1-\mu_{\nu_N^*}}=1$ {which in turn forces
  $$\widetilde{\mathcal D}_N(r+1)-\widetilde{\mathcal D}_N(r) 
  =-(\widetilde{\mathcal H}_N(r+1)-\widetilde{\mathcal H}_N(r)).$$ Since $\mathcal{G}_N=\til{\mathcal{H}}_N+\til{\mathcal{D}}_N$}, this implies that $\mathcal G_N(r+1)-\mathcal G_N(r)=0$ for all $r\in \Gamma^{(k)}$, in turn implying \eqref{claim1}. 
  
  We now turn to the proof of \eqref{claim2}. By the  Cauchy-Schwarz inequality, we have
  \begin{align}
      \label{rn0}
      \big(\mathcal G_N(r+1)-\mathcal G_N(r)\big)^2\leq 2\big(\widetilde{\mathcal H}_N(r+1)-\widetilde{\mathcal H}_N(r)\big)^2+2\big(\widetilde{\mathcal D}_N(r+1)-\widetilde{\mathcal D}_N(r)\big)^2.
  \end{align}
    By definition 
    \begin{equation}
    \big(\widetilde{\mathcal H}_N(r+1)-\widetilde{\mathcal H}_N(r)\big)^2 = N^{-1/2} \bigg( \sum_{j=1}^k (R^j(r+1) - R^j(r) - (2\mu_{\nu_N^*}-1))\bigg)^2\le 4k^2N^{-1/2} \label{rn1}
    \end{equation}
     as $|R^i(r+1) - R^i(r) - (2\mu_{\nu_N^*}-1)|\leq 2.$  Recall the definition of $\widetilde{\mathcal D}_N({r+1})-\widetilde{\mathcal D}_N(r)$ from \eqref{def:diff} and \eqref{def:diff2}. We have
\begin{align}\nonumber\big(\widetilde{\mathcal D}_N(r+1)-\widetilde{\mathcal D}_N(r)\big)^2 &= \bigg( \sum_{j=1}^v\big( \log(m_{b_j,n_j-b_j}) - \mathbf E_{RW_{\nu_N^*}^{(k)}} [ \log(m_{b_j,n_j-b_j})|\mathcal F_r]\big)\bigg)^2 \\& \leq 2k \sum_{j=1}^v\left[\big( \log(m_{b_j,n_j-b_j})\big)^2 + \mathbf E_{RW_{\nu_N^*}^{(k)}} \big[\big( \log(m_{b_j,n_j-b_j})\big)^2\big|\mathcal F_r\big]\right], \label{rn2}
    \end{align}
    where we used Cauchy-Schwarz and $v\le k$. Our particular choice of $\nu_N$ implies that
    $$\left|\int_{[0,1]} x^j\nu_{N}^*(dx)-\int_{[0,1]}x^j\nu(dx)\right| \leq C N^{-1/4} \text{ for all } 1\le j \le k.$$  
    This implies
    $$\left|\int_{[0,1]} x^b(1-x)^{n-b}\nu_{N}^*(dx)-\int_{[0,1]}x^b(1-x)^{n-b}\nu(dx)\right| \leq C N^{-1/4} \text{ for all } 0\le b\le n \le k.$$
    Since $\sigma_\nu^2 <\frac14$, we have that $\int_{[0,1]}x^b(1-x)^{n-b}\nu(dx)>0$ for all $0\le b\le n\le k$ and hence $|m_{b,n-b}^{-1}-1|\leq CN^{-1/4}$ for $0\le b\le n \le k$, where $C=C(\nu,k)$ is independent of $N,b,n$. In turn, this is enough to imply that $|\log(m_{b,n-b})|\leq CN^{-1/4},$ for a deterministic constant $C$, so that $(\log(m_{b,n-b}))^2\leq CN^{-1/2}$. Inserting this bound back in \eqref{rn2} and combining it with the bound in \eqref{rn1}, in view of \eqref{rn0}, we have the desired bound of \eqref{claim2}.

        \medskip
        
    \noindent\textbf{Proof of \eqref{dheq}.} Recall the sets $\Gamma^{(v)}$ from \eqref{def:overlap} and the notation in \eqref{def:diff}.   Note that for $r\in \Gamma^{(k)}$, we have $v=k$, $n_j=1$, and $b_j=\frac12(R_{r+1}^j-R^j(r)+1)$ for $1\le j\le k$. Note that $b_j$ is either one or zero depending on whether the particle jumps right or left.  We can therefore compute 
   \begin{align*}
        m_{b_j, 1-b_j} &= \frac{\int_{[0,1]} x^{b_j} (1-x)^{1-b_j}\nu(dx)}{\int_{[0,1]} x^{b_j} (1-x)^{1-b_j}\nu_N(dx)} \\
       &= \frac{\ind_{R_{r+1}^j - R^j(r) = 1}\mu_{\nu} + \ind_{R_{r+1}^j - R^j(r) = -1}(1-\mu_{\nu})}{\ind_{R_{r+1}^j - R^j(r) = 1}\mu_{\nu_N} + \ind_{R_{r+1}^j - R^j(r) = -1}(1-\mu_{\nu_N})}=\frac{1/2}{1/2 + (\mu_{\nu_N} - 1/2)(R_{r+1}^j - R^j(r))}.
    \end{align*}
We observe that for $x \in \{-1,1\}$ the following relation holds true:
$$\log (1+Ax)= \frac12\left[x\log \left(\frac{1+A}{1-A}\right)+\log(1-A^2)\right].$$
Since $R_{r+1}^j-R^j(r)\in \{-1,1\}$, utilizing the above identity we have
\begin{align*}
    \log [m_{b_j,1-b_j}] & =-\log \left[1+(2\mu_{\nu_N}-1)(R_{r+1}^j-R^j(r))\right]\\ & =-\frac12\left[(R_{r+1}^j-R^j(r))\log \left(\frac{\mu_{\nu_N}}{1-\mu_{\nu_N}}\right)+\log(1-(2\mu_{\nu_N}-1)^2)\right].
\end{align*}
    
       Thus for $r\in \Gamma^{(k)}$, \eqref{def:diff} boils down to
   \begin{align*}
       \mathcal D_N(r+1)-\mathcal D_N(r) 
       =  -\frac12\log \left(\frac{\mu_{\nu_N}}{1-\mu_{\nu_N}}\right)\sum_{j=1}^k(R_{r+1}^j-R^j(r)) -\frac{k}2\log(1-(2\mu_{\nu_N}-1)^2).
   \end{align*}
   This implies for $r\in \Gamma^{(k)}$
\begin{align*}
    \mathbf E_{RW_{\nu_N}^{(k)}} [\mathcal D_N(r+1)-\mathcal D_N(r)|\mathcal F_r]=  -\frac12\log \left(\frac{\mu_{\nu_N}}{1-\mu_{\nu_N}}\right)\sum_{j=1}^k(2\mu_{\nu_N}-1) -\frac{k}2\log(1-(2\mu_{\nu_N}-1)^2).
\end{align*}
 Plugging these identities back in \eqref{def:diff2} we get   
  \begin{align*}
      \widetilde{\mathcal D}_N(r+1)-\widetilde{\mathcal D}_N(r) & = -\frac12\log \left( \frac{\mu_{\nu_N}}{1-\mu_{\nu_N}}\right)\sum_{j=1}^k (R_{r+1}^j-R^j(r) - (2\mu_{\nu_N}-1))
 \\ & =-\frac{N^{1/4}}2\log \left(\frac{\mu_{\nu_N}}{1-\mu_{\nu_N}}\right)(\widetilde{\mathcal H}_N(r+1)-\widetilde{\mathcal H}_N(r)),
  \end{align*}
  where the second equality follows from the definition of $\widetilde{\mathcal{H}}_N(r)$. This verifies \eqref{dheq}.
\end{proof}

\subsection{$k$-point motion moment estimates}\label{sec:sbmest} In this subsection, we collect various moment estimates that will be used in our later analysis. Recall the functionals $V^{ij}$ from \eqref{vijai}. 
	\begin{lem}\label{traps}
  There exists an absolute constant $\Con>0$ such that uniformly over all $k,N >1$, $\theta>0$ and all probability measures $\nu$ on $[0,1]$, we have 
  \begin{align}
      \label{e.traps}
      \mathbf E_{RW_\nu^{(k)}}\bigg[\exp{\bigg(\theta \sum_{i=1}^k \big[\big(\max_{r\leq N}|R^i(r)-(2\mu_\nu - 1)r|\big) +(\mu_\nu(1-\mu_\nu)-\sigma_\nu^2)\sum_{j\neq i}V^{ij}(N) }\big]\bigg)\bigg] \le \Con \cdot e^{\Con \cdot k^4\theta^2 N}.
  \end{align}
	\end{lem}
We remark that this bound will be most useful after replacing $N\to Nt$ and $\theta \to \theta N^{-1/2}$
. The bound is optimal in the sense that the constant $\mu_\nu(1-\mu_\nu)-\sigma_\nu^2$ cannot be improved in a manner that is uniform over all probability laws on $[0,1]$. Indeed this factor vanishes precisely if $\nu$ is the law of Bernoulli$(p)$ for some $p\in [0,1],$ which corresponds to coalescing walks $(R^1,\ldots,R^k)$.
 \begin{proof}
    By the Cauchy-Schwarz inequality, the expectation on the left-hand side ~of \eqref{e.traps} is bounded above by $\sqrt{E_1\cdot E_2}$ where 
     \begin{align*}
         E_1 &: = \mathbf E_{RW_\nu^{(k)}}\bigg[\exp \bigg(2\theta \sum_{i=1}^k \max_{r\le N} |R^i(r)-(2\mu_\nu - 1)r| \bigg)\bigg]\\& \le \prod_{i=1}^k\mathbf E_{RW_\nu^{(k)}} \bigg[\exp\bigg(2k\theta \max_{r\le N} |R^i(r)-(2\mu_\nu - 1)r| \bigg)\bigg]^{1/k}\\ & = \mathbf E_{RW_\nu^{(1)}} \bigg[\exp\bigg(2k\theta \max_{r\le N} |R(r)-(2\mu_\nu - 1)r| \bigg)\bigg],
         \end{align*}
         \begin{align*}
         E_2 &:= \mathbf E_{RW_\nu^{(k)}}\bigg[\exp \bigg(4\theta \big[\mu_\nu(1-\mu_\nu)-\sigma_\nu^2\big] \sum_{1\le i<j\le k} V^{ij}(N)\bigg)\bigg] \\& \leq \prod_{1\le i<j\le k} \mathbf E_{RW_\nu^{(k)}} \bigg[ \exp\bigg( 2\theta k(k-1) \big[\mu_\nu(1-\mu_\nu)-\sigma_\nu^2\big] V^{ij}(N)\bigg)\bigg]^{2/k(k-1)} \\&= \mathbf E_{RW_\nu^{(2)}} \bigg[ \exp\bigg( 2\theta k(k-1) \big[\mu_\nu(1-\mu_\nu)-\sigma_\nu^2\big] V^{12}(N)\bigg)\bigg].
     \end{align*}
   Above, in obtaining the subsequent bounds, we used H\"older's inequality and the fact that the laws $\mathbf P_{RW_\nu^{(k)}} $ form a projective family of $k$-point motions. Now the law of $(R(r)-(2\mu_\nu - 1)r)_{r\ge 0}$ under $\mathbf P_{RW_\nu^{(1)}} $ is simply that of a mean-zero random walk. Therefore $X(r):=\exp\big(k\theta |R(r)-(2\mu_\nu - 1)r| \big)$ is a positive $\mathbf P_{RW_\nu^{(1)}}$-submartingale, thus by Doob's $L^2$ inequality we find that 
     \begin{align*}\mathbf E_{RW_\nu^{(1)}} \bigg[\exp\bigg(2k\theta \max_{r\le N} |R(r)-(2\mu_\nu - 1)r| \bigg)\bigg]&=\mathbf E_{RW_\nu^{(1)}}\big[\max_{r\le N} X(r)^2\big] \\&\leq 4 \mathbf E_{RW_\nu^{(1)}} [X(N)^2] \\ & = 4\mathbf E_{RW_\nu^{(1)}}\bigg[\exp\bigg(2k\theta |R(N)-(2\mu_\nu - 1)N| \bigg)\bigg].
     \end{align*}
     Then by Azuma's martingale inequality, we obtain that the last expression is bounded above by $Ce^{Ck^2\theta^2 N},$ where $C$ is independent of $k,\theta, N,\nu$. 
     
     Now we need to bound the term with the exponential of the intersection times. To do this, we recall $\mathcal{M}(r)$ from \eqref{m}. Note that $\mathcal{M}(r)$ is a $\mathbf P_{RW_\nu^{(2)}}$-martingale with increments bounded above by $3$, which may be verified by a direct calculation using the Markov chain transition diagram of $(R^1,R^2)$ given in Figure \ref{fig:markov_transitions}. Therefore by the Cauchy-Schwarz inequality followed by Azuma's martingale inequality, we see that 
     \begin{align*}
         & \mathbf E_{RW_\nu^{(2)}} \bigg[ \exp\bigg( {8\theta k(k-1) [\mu_\nu(1-\mu_\nu)-\sigma_\nu^2]} V^{ij}(N)\bigg)\bigg] \\ & = \mathbf E_{RW_\nu^{(2)}} \bigg[ \exp\bigg( 2\theta k(k-1) \big[\mathcal M(N) + |R^1(N)-R^2(N)|\big]\bigg] \\ &\leq \sqrt{\mathbf E_{RW_\nu^{(2)}} \bigg[ e^{4\theta k^2 \mathcal M(N)}\bigg] \mathbf E_{RW_\nu^{(2)}} \bigg[ \exp\bigg( 4\theta k^2 \big|(R^1(N)-N(2\mu_\nu-1))-(R^2(N)-N(2\mu_\nu-1))\big|\bigg)\bigg]}  \\ &\leq Ce^{C \theta^2 k^4 N}\sqrt{\mathbf E_{RW_\nu^{(2)}} \bigg[ \exp\bigg( 4\theta k^2 (|R^1(N)-N(2\mu_\nu-1)|+|R^2(N)-N(2\mu_\nu-1)|)\bigg)\bigg]}.
     \end{align*}
     To bound the last term, another application of Cauchy-Schwarz, the projective property of the random walk measures, and then Azuma's inequality gives 
     \begin{align}
        & \mathbf E_{RW_\nu^{(2)}} \bigg[ \exp\bigg( 4\theta k^2 (|R^1(N)-(2\mu_\nu - 1)N|+|R^2(N)-(2\mu_\nu - 1)N|)\bigg] \\ & \hspace{2.5cm}\leq \mathbf E_{RW_\nu^{(1)}} \bigg[ \exp\bigg( 8\theta k^2 |R(N)-(2\mu_\nu - 1)N|\bigg)\bigg]  \leq Ce^{C\theta^2k^4 N}, 
     \end{align}
      where we again used that $R$ minus its mean drift is a mean-zero random walk (with increments bounded absolutely by 2) under $\mathbf P_{RW_\nu^{(1)}}$.
 \end{proof}

 \begin{lem}\label{lte}
        Fix $k\in\mathbb N$ and $p\ge 1$. There exists a constant $C=C(p,k)>0$ such that uniformly over all $N\ge M\ge 0$ and all probability measures $\nu$ on $[0,1]$, we have
		$$ \mathbf E_{RW_\nu^{(k)}}\bigg[\bigg(\big[\mu_\nu(1-\mu_\nu)-\sigma_\nu^2\big]\sum_{i<j}(V^{ij}(N)-V^{ij}(M))\bigg)^p \bigg]^{1/p} \leq C|M-N|^{1/2}.$$
	\end{lem}

 \begin{proof}
 By Minkowski's inequality, the left side is bounded above by $$\sum_{1\le i<j\le k} \mathbf E_{RW_\nu^{(k)}}\bigg[\bigg((\mu_\nu(1-\mu_\nu)-\sigma_\nu^2)(V^{ij}(N)-V^{ij}(M))\bigg)^p \bigg]^{1/p}$$ which by the projective property of the Markovian laws is the same as $$ \frac12 k(k-1)\mathbf E_{RW_\nu^{(2)}}\bigg[\bigg((\mu_\nu(1-\mu_\nu)-\sigma_\nu^2)(V^{12}(N)-V^{12}(M))\bigg)^p \bigg]^{1/p}.$$
     Using \eqref{m}, we may write $4(\mu_\nu(1-\mu_\nu)-\sigma^2)V^{12}(N) = \mathcal M(N) + |R^1(N)-R^2(N)|$, where $\mathcal M(N)$ is a martingale with increments bounded absolutely by 3. Thus using $||x|-|y||\leq |x-y|$ and Minkowski's inequality again, we find that the last expression is bounded above by $\frac18 k(k-1)$ times $$ \mathbf E_{RW_\nu^{(2)}}\big[\big|\mathcal M(N)-\mathcal M(M)\big|^p \big]^{1/p}+2\mathbf E_{RW_\nu^{(1)}}\big[\big|R(N)-R(M)-(N-M)(2\mu_\nu-1)\big|^p \big]^{1/p}.$$ From here, an application of Azuma's inequality allows us to bound both terms by $C|M-N|^{1/2}$, where we again use that $R$ minus its mean drift is a mean-zero random walk (with increments bounded absolutely by 2) under $\mathbf P_{RW_\nu^{(1)}}$.
 \end{proof}

 \begin{prop}\label{exp} 
 Fix $k\in\mathbb N$. Suppose $\{\nu_N\}_N$ is a family of probability measures on $[0,1]$ satisfying \begin{align}
     \label{condv}
     \liminf_{N\to\infty}\int_{[0,1]} x(1-x)\nu_N(dx) >0.
 \end{align} 
 Let $\mathcal G_N = \big(\mathcal G_N(r)\big)_{r\ge 0}$ be any family of $\mathbf P_{RW^{(k)}_{\nu_N}}$-martingales, adapted to the canonical filtration $(\mathcal F_r)_{r\ge 0}$ on $(\mathbb Z^k)^\mathbb N,$ satisfying
		\begin{align}\label{e:gcond}
			[ \mathcal G_N](r)- [ \mathcal G_N](s)  \leq C N^{-1/2} \sum_{i<j}(V^{ij}(r)-V^{ij}(s))
		\end{align}
		for some deterministic constant $C>0$ independent of $N$. For all $p,T>0$ we have 
		\begin{align}
			\label{e:exp}
			\sup_{N\ge 1}\sup_{1\le r \le NT}\mathbf E_{RW^{(k)}_{\nu_N}} \big[\exp\big(p\mathcal G_N(r)\big)\big]<\infty.
		\end{align}
	\end{prop}

\begin{proof} 
For simplicity of notation, we write $\mathbf E_N$ for $\mathbf E_{RW^{(k)}_{\nu_N}}$. We further assume $T=1$ for simplicity of notation. The general case is analogous. Since $\exp\big(p\mathcal G_N(r)\big)$ for $r\ge 0$ is a positive submartingale, by Doob's $L^p$ inequality it suffices to show $\sup_{N\ge 1} \mathbf E_N\big[\exp\big(p\mathcal G_N(N) \big)\big]<\infty,$ i.e., we may ignore the inner supremum after replacing $r$ by the terminal time $N$. Define 
$$W_N(r):= \frac12 \sum_{s\leq r} \log \mathbf E_N [ e^{2p(\mathcal G_N(s)-\mathcal G_N(s-1))}|\mathcal F_{s-1}].$$ By \eqref{aft}, the process $\exp\big( 2p\mathcal G_N(r) - 2W_N(r)\big)$ for $r\ge 0$ is a martingale, therefore 

\begin{align}
\mathbf E_N [ e^{p\mathcal G_N(N)}] & = \mathbf E_N[e^{p\mathcal G_N(N) - W_N(N)} e^{W_N(N)}]\\ &\leq \mathbf E_N[ e^{2p\mathcal G_N(N) - 2W_N(N)}]^{1/2} \mathbf E_N[e^{2W_N(N)}]^{1/2} = \mathbf E_N[e^{2W_N(N)}]^{1/2}. 
\end{align}

By the assumption \eqref{e:gcond}, we have $|\mathcal G_N(r+1)-\mathcal G_N(r)|\leq C^{1/2}N^{-1/4}\leq 1$ deterministically for large enough $N$. For $x$ in the interval $[-2p,2p]$ we have $e^x \leq 1+x+Dx^2$ for a large constant $D=D(p)>0,$ therefore we find that $$e^{2p(\mathcal G_N(s)-\mathcal G_N(s-1))}\leq 1+ 2p (\mathcal G_N(s)-\mathcal G_N(s-1)) + 4p^2D\cdot (\mathcal G_N(s)-\mathcal G_N(s-1))^2,$$ 
so that by martingality of $\mathcal G_N$ and $\log(1+u)\leq u$ we have $$\log \mathbf E_N [e^{2p(\mathcal G_N(s)-\mathcal G_N(s-1))}|\mathcal F_{s-1}] \leq 4p^2D\cdot \mathbf E_N[(\mathcal G_N(s)-\mathcal G_N(s-1))^2|\mathcal F_{s-1}].$$ Hence we find that $$\mathbf E_N[e^{2W_N(N)}] \leq \mathbf E_N[ e^{4p^2D\sum_{s\leq N} \mathbf E_N[(\mathcal G_N(s)-\mathcal G_N(s-1))^2|\mathcal F_{s-1}]}]  \leq \mathbf E_N [ e^{C'N^{-1/2} \sum_{i<j} V^{ij}(N)}],$$ 
where $C'$ is a larger constant that has now absorbed $D$, and we used \eqref{e:gcond} in the last bound. The bound is clearly finite for each fixed $N$. We thus have to justify that the last expression {remains bounded as $N \rightarrow \infty$}. Now, \eqref{condv} guarantees that $\mu_{\nu_N}(1-\mu_{\nu_N}) - \sigma_{\nu_N}^2$ is strictly positive for large enough $N$. A direct application of Lemma \ref{traps} with $\theta=C'N^{-1/2}/(\mu_{\nu_N}(1-\mu_{\nu_N}) - \sigma_{\nu_N}^2)$ shows that the last expression {remains bounded as $N \rightarrow \infty$}.
 \end{proof}

\section{A general theorem for calculating limits via tilting}\label{sec:tilt}
The goal of this section is to establish a general weak convergence theorem that takes the Girsanov tilting into account. As a consequence of this general theorem, we will prove moment convergence in Section \ref{mom}. Recall $V^{ij}$ from \eqref{vijai} and $\Gamma^{(\le v)}$ from \eqref{def:overlap}. Then we have the following.

\begin{thm}\label{converge}
    Fix any $k\in\mathbb N$ and constants $C, T>0$. Assume that $\nu_N$ is a family of probability measures on $[0,1]$ satisfying $\mu_{\nu_N} \to \frac12$ and $\sigma_{\nu_N}^2 \to \sigma^2 \in [0,\frac14)$. Let $\mathbf{R}=(\mathbf{R}(r))_{r\ge 0}$ be the canonical process on $(\mathbb Z^k)^{\mathbb Z_{\ge 0}}$, and define the rescaled processes 
    \begin{align}
        \label{def:resc}
        {\mathbf X}_N(t):= \frac{\big[\mathbf R(Nt)-(2\mu_{\nu_N}-1)Nt\big]}{\sqrt{N}},\;\;\;\;\;\;\;\;\mathscr V^{ij}_N(t):=\frac{V^{ij}(Nt)}{\sqrt{N}}, \;\;\;\;\;\;\;\; \mathscr T_N(t):= \frac{\#\{[Nt]\cap \Gamma^{\le (k-2)}\}}{\sqrt{N}}.
    \end{align}
    for $t\in N^{-1}\mathbb Z_{\ge 0}$, and linearly interpolated for $t \notin N^{-1}\mathbb Z_{\ge 0}$. 
    We will use $\mathbf E_N$ to abbreviate path measures $\mathbf E_{RW^{(k)}_{\nu_N}} $, as defined in Section \ref{sec:girt}. Let $\mathcal G_N = \big(\mathcal G_N(r)\big)_{r\ge 0}$ be any family of $\mathbf P_{N}$-martingales, adapted to the canonical filtration $(\mathcal F_r)_{r\ge 0}$ on $(\mathbb Z^k)^{\mathbb Z_{\ge 0}},$ satisfying\begin{equation}\label{assn1}[\mathcal G_N](r) -[\mathcal G_N](s) \leq CN^{-1/2} \sum_{i<j} (V^{ij}(r)-V^{ij}(s)\big),
		\end{equation}
		for all $r\ge s \ge 0$.  Define:
  \begin{align}
      W_N(r):= \sum_{s=1}^r \log \mathbf E_N [ e^{\mathcal G_N(s)-\mathcal G_N(s-1)}|\mathcal F_{s-1}].
  \end{align}
  Set $\widetilde{\mathcal{G}}_N(t)=\mathcal{G}_N(Nt)$ and $\widetilde{W}_N(t)=W_N(Nt)$ for $t\in N^{-1}\mathbb{Z}$ and linearly interpolated for $t\notin N^{-1}\mathbb{Z}$. Under the $\mathbf P_N$ law, $$\left( \mathbf X_N, ((1-4\sigma^2)\mathscr V^{ij}_N)_{1\le i<j\le k}, \mathscr T_N, \widetilde{\mathcal{G}}_N,\widetilde{\mathcal{W}}_N\right)$$
is a tight family of random variables in the canonical space of 5-tuples $C([0,T],  \mathbb{R}^k \times \mathbb{R}^{k(k-1)/2}\times \mathbb{R}\times \mathbb{R}\times \mathbb R)$. 
Consider any (joint) limit point $\mathbf P_{\infty}$ of that 5-tuple of processes, which is a probability measure on the canonical space of 5-tuples. Let $\left( \mathbf U, (L^{ij})_{1\le i<j\le k}, \mathbf{0}, \mathcal{G},{\mathcal{W}}\right)$ denote the coordinate process on the canonical space of 5-tuples. Then
\begin{enumerate}[label=(\roman*),leftmargin=18pt]
\item \label{convi} $\mathbf{U}$ is a standard $k$-dimensional Brownian motion on $[0,T]$ under $\mathbf P_{\infty}$.
\item \label{convii} $L^{ij}(t):=L^{U^i-U^j}_0(t)$ is the local time at zero of $U^i-U^j$ accrued by time $t$. {Our convention for} the local time at $a\in\mathbb R$ of a continuous semimartingale $X$ is 
\begin{equation}\label{eq:localTimeDef}
    L^X_a(t):= \lim_{\varepsilon \to 0^+}\frac{1}{2 \varepsilon}\int_0^t \mathbbm{1}_{\{a- \varepsilon < X_s < a + \varepsilon\}}d \langle X, X \rangle_s.
\end{equation}
\item \label{conviii} $\mathbf{0}$ is $\mathbf P_{\infty}$-almost surely equal to the function from $[0,T]\to \mathbb R$ which is identically zero. 

\item \label{conviv} $\mathcal{G}$ is a continuous $\mathbf P_{\infty}$-martingale satisfying $\mathbf E_{\infty}[e^{p {\mathcal{G}}(T)-\frac{p}2\langle {\mathcal{G}}\rangle (T)}]<\infty$ for all $p>0$ as well as \begin{align}\label{e:fil2}
	\mathbf E_{\infty}[\exp(\mathcal G(T)-\tfrac12\langle \mathcal G\rangle (T)) \mid \mathcal{F}_T(\mathbf{U})]=1,
		\end{align} where $\mathcal F_T(\mathbf{U})$ denotes the $\sigma$-algebra generated by $\mathbf U$. 
  \item \label{convv}  $\mathcal{W}=\frac12\langle \mathcal{G}\rangle$, where the latter denotes quadratic variation of the continuous martingale $\mathcal G.$
  \end{enumerate}
  	\end{thm}
We remark that the above five conditions do not uniquely characterize the limit points $\mathbf P_{\infty}$, however, they do uniquely characterize expectations with respect to the particular types of observables that we will be interested in, thus the non-uniqueness will not be an issue. We will abuse notation and use $\mathbf P_{\mathrm{lim}}$ to mean different marginal measures of $\mathbf P_\infty$ throughout the proof, clarifying when needed.

\begin{proof}

Write out the vector components as $\mathbf X_N = (\mathbf X^1_N,\ldots,\mathbf X^k_N)$. Note that each of the processes $(\mathbf X^i_N(t))_{t\in N^{-1}\mathbb Z_{\ge 0}}$ is a $\mathbf P_N$-martingale indexed by $N^{-1}\mathbb Z_{\ge 0}$, with increments bounded above by $2N^{-1/2}$. Therefore Azuma's inequality implies the bound $\sup_N \mathbf E_N[|\mathbf X^i_N(t)-\mathbf X^i_N(s)|^p]^{1/p}\leq C|t-s|^{1/2},$ for the linearly interpolated processes $\mathbf X^i_N$. Thus each coordinate of $\mathbf X^i_N$ is tight in $C[0,T]$.
Furthermore, 
the process $\mathscr V_N$ is tight via Lemmas \ref{traps} and \ref{lte}. 

Let us thus consider any \textit{joint} limit point of $(\mathbf{X}_N,\mathscr V_N)$, say $(\mathbf U, (\mathscr L^{ij})_{i<j})$, whose law we view as a measure $\mathbf P_{\text{lim}}$ on the canonical space $C([0,T], \mathbb R^k \times \mathbb R^{k(k-1)/2})$ equipped with its canonical filtration. Without loss of generality, we will assume $(\mathbf U, (\mathscr L^{ij})_{i<j})$ is simply the canonical process on this space. We first show that $\mathbf U$ is a standard Brownian motion in $\mathbb R^k$ and that the processes $(1-4\sigma^2)\mathscr L^{ij}$ necessarily agree with the local time at zero of $U^i-U^j$. We will use the Levy characterization and Tanaka's formula after identifying useful martingales in the prelimit.

For $1\le i \le k$, define the process $S^i(r):= R^i(r)-(2\mu_{\nu_N}-1)r$ which are $\mathbf P_N$-martingales. In the prelimit, the processes for $1\le i \le k$ given by 
    \begin{align*}Y_N^{i}(r)= S^i(r)^2 - \sum_{s=1}^r \mathbf E_N\big[ \big(S^i(s)-S^i(s-1)\big)^2\big|\mathcal F_{s-1}\big] =S^i(r)^2 - 4\mu_{\nu_N}(1-\mu_{\nu_N})r
\end{align*}
are $\mathbf P_N$-martingales. Additionally, the processes for $1\le i<j\le k$ given by 
\begin{align*}K_N^{ij}(r)&= S^i(r)S^j(r) - \sum_{s=1}^r \mathbf E_N\big[ \big(S^i(s)-S^i(s-1)\big)\big(S^j(s)-S^j(s-1)\big)\big|\mathcal F_{s-1}\big] \\&= S^i(r)S^j(r) - 4\sigma_{\nu_N}^2 V_r^{ij} 
\end{align*}
are $\mathbf P_N$-martingales. 
Rescaling, we find that the processes 
\begin{align}N^{-1}Y^i_N(Nt) &= \mathbf X^i_N(t)^2 - 4\mu_{\nu_N}(1-\mu_{\nu_N}) \label{prot} t\\N^{-1}K^{ij}_N(Nt)&=\label{prod}\mathbf X_N^i(t)\mathbf X_N^j(t) - N^{-1/2}4\sigma_{\nu_N}^2 \mathscr V^{ij}_N(t) 
\end{align}
are $\mathbf P_N$ martingales indexed by $t\in N^{-1}\mathbb Z_{\ge 0}.$  Using Lemma \ref{traps} we find that these martingales are uniformly integrable as well.

\medskip

\noindent\textbf{Proof of \ref{convi}.} Now we pass this information to the limit points. Martingality is preserved by limit points as long as one has uniform integrability, thus it is clear that $\mathbf U$ is necessarily a $\mathbf P_{\text{lim}}$-martingale (with respect to the canonical filtration on $C([0,T],\mathbb R^k\times \mathbb R^{k(k-1)/2})$). Next, note that $\lim_{N\to \infty} 4\mu_{\nu_N}(1-\mu_{\nu_N}) = 1$ by our assumption on $\mu_{\nu_N}$. Again using that martingality is preserved by limit points, we can thus apply \eqref{prot} to conclude that $\mathbf U^i(t)^2-t$ is a $\mathbf P_{\text{lim}}$-martingale, i.e., $\mathbf U^i$ necessarily has quadratic variation $t$. 
Yet again using the fact that martingality is preserved by limit points, we can use \eqref{prod} and tightness in $C[0,T]$ of the processes $\mathscr V^{ij}_N$ to conclude that for all $i<j$, the process $\mathbf U^i \mathbf U^j$ is a $\mathbf P_{\text{lim}}$-martingale, i.e., $\langle \mathbf U^i ,\mathbf U^j\rangle = 0$ for $i\ne j$. By Levy's characterization, we may conclude that $\mathbf U$ is therefore a standard Brownian motion in $\mathbb R^k$. This proves \ref{convi}.

\medskip

\noindent\textbf{Proof of \ref{convii}.} We will again use a martingale problem. 
Recall the martingales $\mathcal{M}(r)$ from \eqref{m}. Rescaling, we find that 
$$|\mathbf X_N^i(t)-\mathbf X_N^j(t)| - 4\big[( \mu_{\nu_N}(1-\mu_{\nu_N})-\sigma_{\nu_N}^2)\big]\mathscr V^{ij}_N(t) 
$$
are $\mathbf P_N$-martingales indexed by $t\in N^{-1}\mathbb Z_{\ge 0}.$ All quantities are uniformly integrable and tight in $C[0,T]$ as $N\to \infty$ (as explained earlier). By our hypothesis on $\nu_N$, we have $\mu_{\nu_N}(1-\mu_{\nu_N})\to 1/4$. Thus for the limit points, we conclude that the processes  $|\mathbf U^i-\mathbf U^j| - (1-4\sigma^2)\mathscr L^{ij}$ are $\mathbf P_{\text{lim}}$-martingales. As $\mathbf U$ is a standard $k$-dimensional Brownian motion, and the processes $\mathscr L^{ij}$ are $\mathbf P_{\text{lim}}$-a.s. nondecreasing, we conclude by Tanaka's formula that $(1-4\sigma^2)\mathscr L^{ij}$ is the local time at zero of $U^i-U^j$. This uniquely identifies the law of the joint limit point $(\mathbf U, (\mathscr L^{ij})_{i<j})$.

\medskip

\noindent\textbf{Proof of \ref{conviii}} We now show that as $N\to \infty$ the processes $\mathscr T_N$ converge in probability to the zero process with respect to the topology of $C[0,T]$. Indeed these processes are tight because for $\delta>0$, $$\sup_{|t-s|<\delta}|\mathscr T_N(t)-\mathscr T_N(s)| \leq  \sup_{|t-s|<\delta}N^{-1/2} \sum_{1\leq i<j\leq k} V^{ij}(Nt) - V^{ij}(Ns) $$
and we already know that the intersection time process on the right side is tight. Now if we take any \textit{joint} limit point $(F_t, \mathbf U_t)_{t\in [0,T]}$ of the pair of processes $(\mathscr T_N,\mathbf X_N)$, then $F_t$ is a nondecreasing process with $F_0=0$. Furthermore by definition of the set $\Gamma^{(\le k-2)}$ (see \eqref{def:overlap}), this limit point $F$ is constant on every connected component of the complement of the random closed set $K\cup L$ where \begin{align*}&\text{$K:=\{t\in [0,T]:\mathbf U^i_t=\mathbf U^j_t=\mathbf U^p_t$ for some $1\le i<j<p\leq k\}$},\\ &\text{$L:=\{t\in [0,T]:\mathbf U^i_t=\mathbf U^j_t$ and $\mathbf U^p_t=\mathbf U^q_t$  for some $i<j$ and $p<q$ with $\{i,j\}\cap\{p,q\}=\emptyset\}.$}
\end{align*}
We already know that $\mathbf U$ is a standard Brownian motion in $\mathbb R^k$. But 3d Brownian motion never hits the diagonal $\{x=y=z\}$ of $\mathbb R^3$ after time $0$, and likewise, a 4d Brownian motion never hits the proper 2d linear subspace of $\mathbb R^4$ given by $\{x=y\}\cap\{z=w\}$ after time $0$. Consequently $K\cup L = \{0\},$ and so $F_t=0$ identically. 

\medskip

\noindent\textbf{Proof of \ref{conviv}} Let us now take any $\mathcal G_N$ satisfying the assumption of the theorem. We first prove that the linearly interpolated processes $\big(\mathcal G_N(Nt)\big)_{t\in [0,T]\cap N^{-1}\mathbb Z_{\ge 0}}$ are tight in $C[0,T]$. Indeed, by Burkholder-Davis-Gundy inequality, we have for all $s,t \in N^{-1}\mathbb Z_{\ge 0}$ that $$\mathbf E \big[ |\mathcal G_N(Nt)-\mathcal G_N(Ns)|^p \big] \leq \mathbf E \bigg[ \bigg([\mathcal G_N](Nt)-[\mathcal G_N](Ns)\bigg)^{\frac{p}2}\bigg] \leq C\cdot\mathbf E \bigg[ \bigg(N^{-\frac12}\sum_{i<j}  \big(V^{ij}(Nt)-V^{ij}(Ns)\big)\bigg)^{\frac{p}2}\bigg]\!.\!$$ 
		The second inequality above is due to the assumption \eqref{assn1}. Lemma \ref{lte} implies that the term on the right-hand side of the above equation is bounded by $C|t-s|^{p/4}$ for some constant $C>0$ free of $N$. This verifies the tightness of $t\mapsto \mathcal G_N(Nt)$. 

  Consider any limit point of the triple $\big(\widetilde{\mathcal{G}}_N,\; \mathbf X_N,\; (1-4\sigma^2)\big(\mathscr V^{ij}_N\big)_{1\leq i<j\leq k}\big)$, say $(\mathcal{G}, \mathbf{U}, \big(L^{ij}\big)_{1\leq i<j\leq k})$ whose law $\mathbf P_{\text{lim}}$ (now a different object from the previous parts of the proof) we view as a measure on the canonical space $C([0,T],\mathbb R\times \mathbb R^k \times \mathbb R^{k(k-1)/2})$. By the tightness and $L^p$ bounds above, the process $\mathcal G$ is still a $\mathbf P_{\text{lim}}$-martingale with respect to the canonical filtration on this space. We claim that the quadratic covariations satisfy $\langle G, \mathbf U^i\rangle_t=0$ for all $t\in [0,T]$. To prove this, note that in the prelimit, $\mathbf X^i_N(t)\mathcal G_N(Nt) - \mathscr P_N(t)$ for $t\in N^{-1}\mathbb Z_{\ge 0}$ is a $\mathbf P_N$-martingale where $$\mathscr P_N(t):= N^{-1/2} \sum_{s=1}^{Nt} \mathbf E_N[ (\mathcal G_N(s)-\mathcal G_N(s-1))(R^i(s)-R^i(s-1)-(2\mu_{\nu_N}-1))|\mathcal F_{s-1}].$$
Note that

  \begin{align*} 
      |\mathscr P_N(t)| & =  N^{-\frac12}\sum_{s\in [Nt]\cap\Gamma^{(\le k-1)}} \big|\mathbf E_N[ (\mathcal G_N(s)-\mathcal G_N(s-1))(R^i(s)-R^i(s-1)-(2\mu_{\nu_N}-1)|\mathcal F_{s-1}]\big| \\ &\leq N^{-\frac12}\sum_{s\in [Nt]\cap\Gamma^{(\le k-1)}} \bigg[\sqrt{\mathbf E_N[ (\mathcal G_N(s)-\mathcal G_N(s-1))^2|\mathcal F_{s-1}]}\\ & \hspace{5cm} \cdot\sqrt{\mathbf E_N[(R^i(s)-R^i(s-1)-(2\mu_{\nu_N}-1))^2|\mathcal F_{s-1}]}\bigg] \\ & \leq \sum_{s\in [Nt]\cap\Gamma^{(\le k-1)}} C^{1/2}N^{-3/4} \leq C^{1/2} N^{-1/4} \sum_{i<j} \mathscr{V}^{ij}_N.
  \end{align*}
 Let us briefly explain the above deductions. The first equality is due to the fact that $\mathcal G_N(s)-\mathcal G_N(s-1)=0$ for $s\in \Gamma^{(k)}$ via assumption \eqref{assn1}. The inequality in the second line follows from Cauchy-Schwarz inequality. Note that $R^i(s)-R^i(s-1)-(2\mu_{\nu_N}-1)$ is uniformly bounded by $2$. Thus, employing the bound in \eqref{assn1}, we get the inequality in the third line. The final inequality follows from the definition of $\Gamma^{(k-1)}$ and $\mathscr{V}^{ij}$. 
 
 Now, since we already know $\mathscr V_N$ are tight, we conclude that $|\mathscr P_N(t)|\to  0$. Invoking Lemmas \ref{traps} and \ref{exp}, we see that $\mathbf X_N^i\mathcal{G}_N-\mathscr P_N$ is uniformly integrable. Hence $\mathbf U^i \mathcal G$, being a limit point of $\mathbf X_N^i\mathcal{G}_N-\mathscr P_N$, is a $\mathbf P_{\mathrm{lim}}$-martingale for all $1\le i \le k$, so that $\langle \mathcal G, \mathbf U^i\rangle_t=0$ for all $t\in [0,T]$, as claimed. In fact, Proposition \ref{exp} guarantees that the process $\exp(\mathcal G - \frac12 \langle \mathcal G\rangle)$ is a $\mathbf P_{\text{lim}}$-martingale. Now a standard argument that we used in our previous paper \cite[Proof of Theorem 3.2]{DDP23} ensures that $\mathbf E_{\text{lim}}[\exp(\mathcal G(T)-\tfrac12\langle \mathcal G\rangle (T)) \mid \mathcal{F}_T(\mathbf{U})]=1.$ For completeness, we reproduce the argument in the next paragraph.

  Take any bounded $\mathcal F_T(\mathbf{U})$-measurable functional $H:C([0,T],\mathbb R^k)\to \mathbb R$.  Note that the process $t\mapsto \mathbf E_{\text{lim}}[H(\mathbf{U})|\mathcal F_t(\mathbf U)]$ is a martingale in the filtration of $\mathbf{U}$. Under $\mathbf P_{\text{lim}},$ the marginal law of $\mathbf{U}$ is a standard $k$-dimensional Brownian motion by the result of part \ref{convi}, thus by the martingale representation theorem we have
		$$H(\mathbf{U})=\mathbf E_{\text{lim}} [H(\mathbf{U})]+\sum_{i=1}^k\int_0^T h_s^i dU_s^i$$
		for some adapted $\mathbb R$-valued processes $h^1,\ldots,h^k$. For stochastic exponential we have $$\exp\left(\mathcal{G}(T) - \frac12 \langle \mathcal{G}\rangle (T)\right) = 1+\int_0^T \exp\left(\mathcal{G}(s) - \frac12 \langle \mathcal{G}\rangle (s)\right) d\mathcal{G}(s).$$ Using this we find that 
		\begin{align*} &\mathbf E_{\text{lim}}\left[\big(e^{\mathcal{G}(T) - \frac12 \langle \mathcal{G}\rangle (T)} -1\big)\big(H(\mathbf{U})-\mathbf E_{\text{lim}}[H(\mathbf{U})]\big)\right] \\ &= \sum_{i=1}^k\mathbf E_{\text{lim}}\bigg[\bigg(\int_0^T e^{\mathcal{G}(s) - \frac12 \langle \mathcal{G}\rangle(s)}d\mathcal{G}(s)\bigg)\bigg(\int_0^T h_s^i  dU^i_s\bigg)\bigg] \\&= \sum_{i=1}^k \mathbf E_{\text{lim}}\bigg[\int_0^T h_s^i e^{\mathcal{G}(s) - \frac12 \langle \mathcal{G}\rangle(s)} d\langle \mathcal{G},U^i\rangle(s)\bigg] = 0,
		\end{align*}
		where the last equality is due to the fact $\langle U^i,\mathcal{G}\rangle \equiv 0$ almost surely. This proves that for all bounded measurable $H$ we have $\mathbf E_{\text{lim}}[e^{\mathcal{G}(T) - \frac12 \langle \mathcal{G}\rangle(T)}H(\mathbf{U})]=\mathbf E_{\text{lim}}[H(\mathbf{U})]$, establishing \eqref{e:fil2}.

  \medskip

\noindent\textbf{Proof of \ref{convv}} 
  By assumption \eqref{assn1}, we have $|\mathcal G_N(r+1)-\mathcal G_N(r)|\leq C^{1/2}N^{-1/4}\leq 1$ deterministically for large enough $N$. For $x\in [-1,1]$ we have $e^x \leq 1+x+x^2$, therefore we find that $$e^{\mathcal G_N(s)-\mathcal G_N(s-1)}\leq 1+ (\mathcal G_N(s)-\mathcal G_N(s-1)) + (\mathcal G_N(s)-\mathcal G_N(s-1))^2,$$ consequently using martingality of $\mathcal G_N$ and $\log(1+u)\leq u$ we find that 
  \begin{align*}W_N(r)-W_N(r-1)& \leq \log \mathbf E_N [e^{\mathcal G_N(r)-\mathcal G_N(r-1)}|\mathcal F_{r-1}] \\&\leq  \mathbf E_N[(\mathcal G_N(r)-\mathcal G_N(r-1))^2|\mathcal F_{r-1}] \leq CN^{-1/2} \sum_{i<j} (V^{ij}(r) - V^{ij}(r-1)),
  \end{align*}
  where we used \eqref{assn1}. Using this bound together with the nondecreasing property of $W_N$, and Lemmas \ref{traps} and \ref{lte}, it is clear that for all $p\ge 0$  $$\sup_{N\ge 1} \mathbf E_N[|W_N(Nt)-W_N(Ns)|^p]^{1/p} \leq C|t-s|^{p/2}, \;\;\;\;\;\;\;\;\;\;\sup_{N\ge 1} \mathbf E_N[e^{pW_N(Nt)}] < \infty. $$
  Thus the linearly interpolated processes $t\mapsto W_N(Nt)$ for $t\in N^{-1}\mathbb Z_{\ge 0}$ are tight in $C[0,T]$, so we may consider any limit point $\mathcal W$ which is taken \textit{jointly} with the earlier triple of processes $(\mathcal{G}, \mathbf{U}, \big(L^{ij}\big)_{1\leq i<j\leq k})$. Note that since $\exp(\mathcal G_N - W_N)$ are martingales in the prelimit, and the above exponential bound on $W_N$ guarantees uniform integrability of its exponential, it follows that $\exp(\mathcal G-\mathcal W)$ is a martingale for any limit point of the joint 4-tuple. Since $\mathcal G$ is a continuous martingale, this then forces that $\mathcal W = \frac12 \langle \mathcal G\rangle$. 
\end{proof}

We record an extension of the above theorem which also takes certain types of measures into account. The following corollary will be useful in our later analysis.

\begin{cor}\label{convc} Assume the same notations and assumptions of Theorem \ref{converge}.
    Suppose $k=2m$ is even. For $t>s>0$ define
		\begin{align}\label{deltadef}
			\Delta_m(s,t):=\{(s_1,s_2,\ldots,s_m)\in [s,t]^m \mid s\le s_1\le s_2\le\cdots\le s_m\le t\}
		\end{align}
		to be the set of all $m$ ordered points in $[s,t]$.    We define $\mathcal M(\Delta_m(0,T))$ to be the space of finite and non-negative Borel measures on that simplex, equipped with the topology of weak convergence. Consider the following sequence of $\mathcal M(\Delta_m(0,T))$-valued random variables
\begin{align}
\label{gamman}
	\gamma_N:= N^{-m/2}(1-4\sigma^2)^m\sum_{u_1\le ...\le u_m \in (N^{-1}\mathbb Z_{\ge 0})\cap [0,T]}\;\;\;\prod_{j=1}^m \ind_{ \{\mathbf X^{2j-1}_N(u_j)=\mathbf X^{2j}_N(u_j)\} }\delta_{(u_1,\ldots,u_m)}
\end{align}
where $\mathbf X_N = (\mathbf X^1_N,\ldots,\mathbf X^k_N)$. The random variables $\{\gamma_N\}_{N\ge 1}$ are tight in $\mathcal M(\Delta_m(0,T))$. Moreover, any limit point of the triple $(\mathbf X_N,\; (1-4\sigma^2)\big(\mathscr V^{ij}_N\big)_{1\leq i<j\leq k},\; \gamma_N)$ is of the form
\begin{align*}
    \left(\mathbf{U},\; \big(L^{ij}\big)_{1\leq i<j\leq k}, \;\prod_{j=1}^m dL^{{2j-1,2j}}(u_j)\right) 
\end{align*}
where $\mathbf U, L^{ij}$ are as in Theorem \ref{converge}, and $dL(t)$ denotes the Lebesgue-Stieltjes measure induced by the increasing function $t\mapsto L(t)$.
\end{cor}

\begin{proof}
Note that
\begin{align*}
	\gamma_N(\Delta_m(0,T)) \le \prod_{j=1}^m\left[N^{-\frac12}\sum_{u\in [0,T]\cap (N^{-1}\mathbb Z_{\ge 0})} \ind_{\{\mathbf X_N^{2j-1}(u)=\mathbf X_N^{2j}(u)\}} \right].
\end{align*}
From Lemma \ref{traps} we know the exponential moments for each of the terms in the product are uniformly bounded under the $\mathbf E_N$ measure. Thus for all $p\ge 1$,
\begin{align}\label{gmass1}
	\sup_{N\ge 1} \mathbf E_N[\gamma_N(\Delta_m(0,T))^p]<\infty.
\end{align}
Hence the laws of $\{\gamma_N\}_{N\ge 1}$ are tight, because the total mass of $\gamma_N$ is a tight family of random variables and because $\Delta_m(0,T)$ as defined in \eqref{deltadef} is a compact space. Let $\left(\mathbf{U}, \big(K^{ij}\big)_{1\leq i<j\leq k}, \gamma\right)$ be any limit point of the sequence $(\mathbf X_N,(1-4\sigma^2)\big(\mathscr V^{ij}_N\big)_{1\leq i<j\leq k}, \gamma_N)$.  Note that the joint cumulative distribution function of the measure $\gamma$ is necessarily given by a product of the $K^{ij}$. Since we know that $K^{ij}=L_0^{U^{i}-U^{j}}$ by Theorem \ref{converge}\ref{convii}, this implies that $\gamma=\prod_{j=1}^m dL_0^{U^{2j-1}-U^{2j}}(u_j)$. 
\end{proof}

\subsection{Convergence of moments}\label{mom}

We will now prove that moments of the field \eqref{field} converge to the moments of \eqref{she}. It will be a consequence of Theorem \ref{converge}, illustrating the applicability of the convergence theorem. While the moment convergence is not enough to prove the weak convergence of Theorem \ref{main}, it gives a strong indication of it, and more importantly, it implies some useful estimates that will be used later.

\begin{prop}[Moment convergence] Fix $\phi \in C_c^\infty(\mathbb R)$, $t>0$, and $k\in \mathbb N$. With $\mathscr U_N$ as defined in \eqref{field}, we have that $$\lim_{N\to \infty} \mathbb E[\mathscr U_N(t,\phi)^k] = \mathbb E \bigg[ \bigg( \int_\mathbb R \mathcal U_t(x)\phi(x)dx\bigg)^k\bigg],$$
    where $(t,x)\mapsto \mathcal U_t(x)$ solves \eqref{she} with initial condition $\delta_0(x).$
\end{prop}

 \begin{proof} Fix any  $\phi\in \mathcal{S}(\R)$. With $\mathscr{U}_N$ defined in \eqref{field}, we wish to compute the annealed expectation $\mathbb E[\mathscr U_N(t,\phi)^k]$. Just like \eqref{1stmoment} we can write 
\begin{align*}\mathscr U_N(t,\phi)^k &= \mathsf E^\omega_{(k)} \bigg[\prod_{j=1}^kC_{N,t,N^{-1/2}(R^j(Nt)- N^{3/4} t)}\phi\big(N^{-1/2}(R^j(Nt)- N^{3/4}t)\big)\bigg],
\end{align*}
where $\mathsf E^\omega_{(k)}$ is a \textit{quenched} expectation of $k$ independent particles sampled from a fixed realization of the environment $ (\omega_{t,x})_{t\ge 0, x\in \mathbb Z}$. Taking the annealed expectation over the quenched one we get
\begin{align}\label{e.utphi}\Ex[\mathscr U_N(t,\phi)^k] &= \mathbf E_{RW_{\nu}^{(k)}} \bigg[\prod_{j=1}^kC_{N,t,N^{-1/2}(R^j({Nt})- N^{3/4} t)}\phi\big(N^{-1/2}(R^j({Nt})- N^{3/4}t)\big)\bigg]. 
\end{align}
Let $\{\nu_N^*\}_{N\ge 1}$ be the \stt\ sequence of measures corresponding to $\nu$. Recall the measure $\mathbf Q_{RW_{\nu}^{(k)}}^{\lambda}$, $\mathpzc M^{\lambda}$ and $\widetilde{\mathcal G}_N$ from Definition \ref{q}, \eqref{m_n} and \eqref{g} respectively. The key idea from here is to get rid of divergent terms appearing in the above expectation through the following change of variable:
\begin{align}\label{tilt}
     \mathbf{E}_{RW_{\nu}^{(k)}}[\mathbf{F}] =\mathbf{Q}_{RW_{\nu}^{(k)}}^{N^{-1/4}}\left[\frac1{\mathpzc M^{N^{-\frac14}}(Nt)}\mathbf{F}\right] 
=\mathbf{E}_{RW_{\nu_N^*}^{(k)}}\left[\frac1{\mathpzc M^{N^{-\frac14}}(Nt)}e^{\widetilde{\mathcal{G}_N}(Nt)}\mathbf{F}\right].  
 \end{align}
where  $\mathbf{F}=\mathbf{F}(\mathbf{R})$ is any measurable functional on the path space. 
Then we appeal to Theorem \ref{converge} to deduce convergence under the measure $\mathbf{P}_{RW_{\nu_N^*}^{(k)}}$. 
 
Indeed, applying the tilts in \eqref{tilt} we get
\begin{equation}
    \label{e.utphi2}
    \begin{aligned}
    \mbox{r.h.s.~of \eqref{e.utphi}} & =\mathbf E_{RW_{\nu_N^*}^{(k)}}\bigg[\exp\bigg({\mathcal{G}}_N(Nt)-\sum_{s\le Nt}\mE_{RW_{\nu_N}^{(k)}}[e^{\mathcal{G}_N(s)-\mathcal{G}_N(s-1)}\mid \mathcal{F}_{s-1}]\bigg) \\ & \hspace{1cm}\cdot \frac1{\mathpzc M^{N^{-\frac14}}(NT)}\prod_{j=1}^kC_{N,t,N^{-\frac12}(R^j({Nt})- N^{\frac34} t)}   \phi\big(N^{-\frac12}(R^j({Nt})- N^{\frac34}t)\big)\bigg]\!.\!
\end{aligned}
\end{equation}
Here $\mathcal{G}_N$ is a family of $\mathbf P_{RW_{\nu}^{(k)}}$-martingales satisfying \eqref{g1}. We now use the notations of rescaled processes: $\mathbf{X}_N, \mathscr{V}_N^{ij},$ and $\mathscr{T}_N$, considered in \eqref{def:resc}. We claim that 
\begin{align}
\label{twort}
    \frac1{\mathpzc M^{N^{-1/4}}(NT)}\prod_{j=1}^kC_{N,t,N^{-1/2}(R^j(Nt)- N^{3/4} t)} =\exp\left(4\sigma^2\sum_{1\le i<j\le k}\mathscr V^{ij}_N(t)-\mathsf{Err}_N(t)\right)
\end{align}
where
\begin{align}\label{errbd}
|\mathsf{Err}_N(t)|\le \Con\bigg(\mathscr{T}_N(t)+N^{-1/2}\sum_{1\le i<j\le k}\mathscr V^{ij}_N(t)\bigg).
\end{align}
The proofs of \eqref{twort} and \eqref{errbd} follow by some algebraic calculations and definition chasing. We shall prove them in a moment. 
Note that $\nu_N$ and $\mathcal{G}_N$ satisfy the hypothesis of Theorem \ref{converge}. 
Thus applying Theorem \ref{converge}, in view of the bound in \eqref{errbd}, we see that any limit point of the random variable inside the expectation on the r.h.s.~of \eqref{e.utphi2} is of the form
\begin{align*}
    \exp\left(\mathcal{G}(t)-\tfrac12\langle \mathcal{G}\rangle (t)\right)\exp\left(\frac{4\sigma^2}{1-4\sigma^2}\sum_{1\le i<j\le k} L^{ij}(t)\right)\prod_{j=1}^k \phi({U}^i(t)), 
\end{align*}
where $\mathbf{U}$ is a standard $k$-dimensional Brownian motion on $[0,T]$, defined on a standard probability space $(\Omega,\mathcal F,\mathbf P)$,  
 $L^{ij}(t):=L^{U^i-U^j}_0(t)$ is the local time at zero of $U_i-U_j$ accrued by time $t$ and $\mathcal{G}$ is a $C[0,T]$-valued random variable satisfying $\mathbf E[e^{p {\mathcal{G}}(T)-\frac{p}2\langle {\mathcal{G}}\rangle (T)}]<\infty$ for all $p>0$ as well as \eqref{e:fil2}.

Thanks to the estimates in Lemma \ref{traps} and Proposition \ref{exp}, we see that the random variables in the prelimit are uniformly integrable. Thus every subsequence of the expectation on r.h.s.~of \eqref{e.utphi2} has a further subsequence which as $N\to\infty$ converges to
 \begin{align*}
    & \mE\left[\exp\left(\mathcal{G}(t)-\tfrac12\langle \mathcal{G}\rangle (t)\right)\exp\left(\frac{4\sigma^2}{1-4\sigma^2}\sum_{1\le i<j\le k} L^{ij}(t)\right)\prod_{j=1}^k \phi(\mathbf{U}^i(t)) \right] \\ & = \mE\left[\exp\left(\frac{4\sigma^2}{1-4\sigma^2}\sum_{1\le i<j\le k} L^{ij}(t)\right)\prod_{j=1}^k \phi(\mathbf{U}^i(t)) \right].
\end{align*}
The last equality is due to \eqref{e:fil2}. Since this expression is free $\mathcal{G}$ (which may depend on the subsequence), we thus get that
\begin{equation} \label{eq:moments}
    \lim_{N\to\infty} \mathbb \E[\mathscr U_N(t,\phi)^k] = \mathbf E\left[\exp\left(\frac{4\sigma^2}{1-4\sigma^2}\sum_{1\le i<j\le k} L^{ij}(t)\right)\prod_{j=1}^k \phi(\mathbf{U}^i(t)) \right].
\end{equation}
			
		
		By the Feynman-Kac formula, the solution $\mathcal U$ of the stochastic heat equation \eqref{she} admits well-known moment formulas in terms of local times of Brownian bridges. In particular, from \cite[Lemma 2.3]{BC95} we have 
		\begin{align*}
			\mathbf E\left[\prod_{i=1}^k \mathcal{U}_t(x_i)\right] = p(t,x_1)\cdots p(t,x_k)\mathbf E_{\vec{x}}\bigg[  e^{\frac{4\sigma^2}{1-4\sigma^2} \sum_{i<j} L^{B^i-B^j}_0(t)} \bigg],
		\end{align*}
		where $p$ is the standard heat kernel, and the $B^i$ are independent Brownian bridges on $[0,t]$ from $0$ to $x_i$ respectively. Note that in \cite{BC95}, local time is defined as an integral against Lebesgue measure: $L^X_a(t):= \lim_{\varepsilon \to 0^+}\frac{1}{2 \varepsilon}\int_0^t \mathbbm{1}_{\{a- \varepsilon < X_s < a + \varepsilon\}}ds$, in contrast with our definition of local time in \eqref{eq:localTimeDef} which is an integral against $d\langle X, X \rangle_s$. Hence, when taking the local time of $B^i - B^j$, our definition of local time differs from theirs by a factor of $2$.   
  
  We can integrate this formula to arrive at the moment formula for  $\mathcal{U}_t(\phi):=\int_{\R} \mathcal{U}_t(x)\phi(x)dx$, with $(t,x) \mapsto \mathcal U_t(x)$ solving \eqref{she}, and we see that it matches the right-hand side of \eqref{eq:moments},  thus completing the proof of moment convergence for $\mathscr U_N$ modulo \eqref{twort} and \eqref{errbd}.

Let us now prove \eqref{twort}-\eqref{errbd}. Using the definition of martingale $\mathpzc M^{\lambda}$ from \eqref{m_n} we see that
\begin{align}
    \label{erty}
    \frac1{\mathpzc M^{N^{-\frac14}}(Nt)}\prod_{j=1}^kC_{N,t,N^{-\frac12}(R^j(Nt)- N^{3/4} t)} = \exp\bigg(\sum_{\ell=0}^{Nt-1} f_{\ell}^{N^{-\frac14},k}  -kNt\log\cosh(N^{-\frac14})\bigg),
\end{align}
where $f_{\ell}^{N^{-1/4},k}$ is defined in \eqref{def:flk}.   Recall $\Gamma^{(v)}$ from \eqref{def:overlap}. Using the expressions for $f_{\ell}^{\lambda ,k}$ for $\ell\in \Gamma^{(k)}, \Gamma^{(k-1)}$ (obtained after \eqref{def:flk}) we have
\begin{equation}
    \label{erty2}
    \begin{aligned} \sum_{\ell=0}^{Nt-1}f_\ell^{N^{-\frac14},k} -kNt\log \cosh (N^{-\frac14})  
& =  g(N^{-1/4},\sigma^2)\cdot \#\{[Nt]\cap \Gamma^{(k-1)}\}+\sum_{\ell \in[Nt]\cap \Gamma^{(\le k-2)}} f_{\ell}^{N^{-1/4},k}\\ & \hspace{2cm}-k\log \cosh (N^{-1/4}) \cdot \#\{[Nt]\cap \Gamma^{(\le k-2)}\}.
\end{aligned}
\end{equation}
We note that 
\begin{align}
\label{relation}
   \#\{[Nt]\cap \Gamma^{(k-1)}\} = N^{1/2}\sum_{1\le i<j\le k} \mathscr{V}_N^{ij}(t)- \sum_{\ell\in [Nt]\cap \Gamma^{(\le k-2)}} \sum_{1\le i<j\le k}\ind_{\{R^i(\ell)=R^j(\ell)\}},
\end{align}
where $\mathscr{V}_N^{ij}$ is defined in \eqref{def:resc}. Let us now define
\begin{align*}
    \mathsf{Err}_N^{(1)}(t)& :=g(N^{-1/4},\sigma^2)\sum_{\ell\in [Nt]\cap \Gamma^{(\le k-2)}} \sum_{1\le i<j\le k}\ind_{\{R^i(\ell)=R^j(\ell)\}} \\ & \hspace{2cm}-\sum_{\ell \in[Nt]\cap \Gamma^{(\le k-2)}} f_{\ell}^{N^{-1/4},k}+k\log \cosh (N^{-1/4}) \cdot \#\{[Nt]\cap \Gamma^{(\le k-2)}\}, \\
    \mathsf{Err}_N^{(2)}(t) & :=\big(g(N^{-1/4},\sigma^2)-4\sigma^2 N^{-1/2}\big)N^{1/2}\sum_{1\le i<j\le k} \mathscr{V}_N^{ij}(t).
\end{align*}
In view of the relation in \eqref{relation} and above definitions, we obtain from \eqref{erty2} that
\begin{align*}
    \sum_{\ell=0}^{Nt-1}f_\ell^{N^{-\frac14},k} -kNt\log \cosh (N^{-\frac14}) = 4\sigma^2 \sum_{1\le i<j\le k} \mathscr{V}_N^{ij}(t)+\mathsf{Err}_N^{(2)}(t)-\mathsf{Err}_N^{(1)}(t).
\end{align*}
Setting $\mathsf{Err}_N(t):=\mathsf{Err}_N^{(2)}(t)-\mathsf{Err}_N^{(1)}(t)$ leads to \eqref{twort} from \eqref{erty} and the above relation. To prove the bound in \eqref{errbd}, note that $g(N^{-1/4},\sigma^2)=4\sigma^2N^{-1/2}+O(N^{-1})$. Thus we have
\begin{align*}
    |\mathsf{Err}_N^{(2)}(t)|\le \Con N^{-1/2}\sum_{1\le i<j\le k} \mathscr{V}_N^{ij}(t).
\end{align*}
Furthermore $\log\cosh(N^{-1/4})=O(N^{-1/2})$ and $f_{\ell}^{N^{-1/4},k}$ is uniformly bounded by $\Con N^{-1/2}$ by Lemma \ref{prop:bound}. Using these bounds and recalling the definition of $\mathscr{T}_N$ from \eqref{def:resc} we see that
\begin{align*}
    |\mathsf{Err}_N^{(1)}(t)|\le \Con \cdot \mathscr{T}_N(t).
\end{align*}
Combining the bounds in the last two displays verifies \eqref{errbd}.
	\end{proof}

\section{A discrete Hopf-Cole transform: Proof overview of Theorem \ref{main}}\label{hopf}
Having established the moment convergence, in this section, we present a technical roadmap to the proof of the main result Theorem \ref{main}. We will rely on the martingale characterization of the solution of the multiplicative noise stochastic heat equation. Specifically, consider a measure $\mu$  on $C([0,T],C(\mathbb R))$, and let $(\mathcal U(t))_{t\in [0,T]}$ denote the canonical process on that space. The canonical filtration $\mathcal F_t$ on $C([0,T],C(\mathbb R))$ is the one generated by $\{\mathcal U(s):s\leq t\}$. Define the heat operator $\mathcal{L}:= \partial_t-\frac12 \partial_x^2$, which acts on tempered distributions $f\in \mathcal S'(\mathbb R^2)$ by the formula $((\partial_t-\tfrac12\partial_x^2)f,\varphi) = (f, (-\partial_t-\tfrac12\partial_x^2)\varphi)$.

{\renewcommand{\arraystretch}{1.2}
	\begin{longtable}[t]{c|c|c|c}
 \captionsetup{width=.9\linewidth}

		\toprule
		Objects & \multicolumn{2}{c|}{Discrete Version}  & Continuous Version \\
		\midrule
  Heat Operator & \multicolumn{2}{c|}{$\mathcal{L}_N$ \eqref{heat1}} & $\mathcal{L}=\partial_t-\frac12 \partial_x^2$  \\
  \midrule
 & Microscopic version & Rescaled field & Continuum limit \\
 \midrule
   Process of interest& ${Z}_N^{\omega}$ \eqref{z_n} & $\mathscr{U}_N$ \eqref{field} & $\mathcal U$ \eqref{she} \\
Relevant martingale & ${v}_N$ \eqref{mart'} & $M_N$ \eqref{m_field} & $\mathcal{Y}$ \eqref{me1} \\
Quadratic Martingale field & & $Q_N$ \eqref{qfield} & Integrated $\mathcal{U}^2$ \\
\bottomrule

\caption{Objects that will appear in the proof of weak convergence} \\

	\end{longtable}
}

We recall a result of \cite[Proposition 4.11]{BG97} inspired by the work of \cite{konno}. If for all $\phi \in C_c^\infty(\mathbb R)$ the processes 
\begin{align}
\label{me1}
    \mathcal{Y}(t,\phi):= \mathcal U(t,\phi) - \frac12 \int_0^t ( \mathcal U(s) , \phi'') ds
\end{align}
 (which is formally the same as $ \int_0^t\int_\mathbb R(\mathcal{L} \mathcal U)(s,y)\phi(y)dyds$) 
  are $(\mathcal F_t,\mu)$-martingales with quadratic variation 
 \begin{align}\label{me2}
     \langle \mathcal{Y}(\cdot,\phi)\rangle(t) = {\frac{8\sigma^2}{1-4\sigma^2}}\int_0^t (\mathcal U(s)^2, \phi^2)_{L^2(\mathbb R)}ds,
 \end{align}
 then (under reasonable assumptions on the spatial growth of $\mathcal U(t)$ at infinity) the measure $\mu$ necessarily coincides with the law of \eqref{she} started from an initial condition that is distributed as $\mathcal U(0)$ under $\mu$.

We shall eventually show our prelimiting field from \eqref{field} is tight and any limit point satisfies the above martingale characterization.  The key observation that drives our proof is that the prelimiting field itself satisfies a lattice stochastic heat equation. This will then allow us to define some observables that will be crucial in proving tightness and identifying the limit points in later sections. We now end the heuristic discussions and rigorously derive the lattice SPDE, then write out the quadratic variations of the resultant martingales. 

\begin{defn}Fix $N\in\mathbb N$. Define a discrete lattice 
\begin{equation}\label{lattice}\Lambda_N:= \{(t,x) \in \mathbb Z_{\ge 0} \times \mathbb R: x+tN^{-1/4} \in \mathbb Z\},
\end{equation} and a \textbf{discrete heat operator} $\mathcal L_N$, which acts on real-valued functions $f:\Lambda_N\to\mathbb R$ by 
\begin{equation}\label{heat1}(\mathcal L_Nf)(t,x) := f(t+1,x-N^{-1/4}) - \big(\rho_Nf(t,x-1) + (1-\rho_N) f(t,x+1) \big),\end{equation} with $\rho_N$ given by \eqref{def:rhon}.\end{defn}
Recall $\mathsf P^\omega(t,x)$ from \eqref{pwtx}. Assume $\mathbb{E}[\omega_{ij}]=\tfrac12$. For each realization of the environment $\omega$, define the random function $Z_N^\omega:\Lambda_N\to\mathbb R$: 
\begin{equation}Z_N^\omega(t,x):= C_{N,N^{-1}t, N^{-1/2}x}\mathsf P^\omega(t, x+N^{-1/4}t), \label{z_n}
\end{equation}where $C_{N,T,x}$ is defined in \eqref{cntx}. Note that $Z_N^{\omega}$ is the pointwise version of the rescaled field $\mathscr{U}_N$ defined in \eqref{field}. Roughly speaking, Theorem \ref{main} suggests that we wish to prove convergence to SHE for the fields $N^{1/2}Z^\omega_N(Nt,N^{1/2}x)$. The following lemma suggests $Z_N^\omega(t,x)$ 
satisfies a discrete version of \eqref{me1} on $\Lambda_N$.
 
\begin{lem}  For $(t,x)\in\Lambda_N$, define\begin{equation}\label{mart'}
    v_N(t,x):= \big(\mathcal L_NZ_N^\omega\big)(t,x).
    \end{equation} Then $v_N$ is a martingale-difference field in the filtration $\mathcal F^\omega_t = \sigma(\{\omega_{s,x}:x\in \mathbb Z, 0\leq s\leq t\}),$ i.e. \begin{equation}\mathbb E \big[ v_N(t,x)\big| \mathcal F^\omega_{t-1} ] =0.\label{mfield}\end{equation}
\end{lem}

\begin{proof} First note that $v_N(t,x)$ is indeed adapted to $\mathcal F^\omega_t$ because $Z_N^\omega(t+1,\cdot)$ is measurable with respect to $\mathcal F^\omega_t$. Note that the quenched probabilities satisfy the relation $$\mathsf P^\omega(t+1,x) = (1-\omega_{t,x+1})\mathsf P^\omega(t,x+1) +  \omega_{t,x-1}\mathsf P^\omega(t,x-1),$$ which can be derived from conditioning on the previous step. Using the definition of $Z_N^{\omega}$ and $C_{N,t,x}$ from \eqref{z_n} and \eqref{cntx} respectively, we thus have
\begin{align*}
   Z_N^\omega(t+1,x-N^{-1/4}) & = e^{N^{-\frac12} -\log\cosh(N^{-\frac14})} \big[(1- \omega_{t,x+N^{-1/4}t+1})e^{N^{-1/4}(-1-N^{-1/4})}Z_N^\omega(t,x+1)\\ & \hspace{3cm}+ \omega_{t,x+N^{-1/4}t-1}e^{N^{-1/4}(1-N^{-1/4})}Z_N^\omega(t,x-1)\big] \\ & =  2(1- \omega_{t,x+N^{-1/4}t+1})(1-\rho_N)Z_N^\omega(t,x+1)+ 2\omega_{t,x+N^{-1/4}t-1}\rho_NZ_N^\omega(t,x-1) , 
\end{align*}
 for $(t,x)\in\Lambda_N$. Consequently we have that \begin{equation}\label{mzero'}v_N(t,x) = (1 - 2 \omega_{t,x+N^{-1/4}t+1}) (1-\rho_N) Z^\omega_N(t,x+1) + (2\omega_{t,x+N^{-1/4}t-1}-1)\rho_N Z_N^\omega(t,x-1).
\end{equation}
Since the $\omega_{t,x}$ have mean 1/2 and are independent of $\mathcal F^\omega_{t-1}$, and since $Z_N^\omega(t,\cdot)$ is $\mathcal F^\omega_{t-1}$ measurable, \eqref{mfield} follows.
\end{proof}



\begin{defn}Define the rescaled martingale field as follows. For $\phi\in C_c^\infty(\mathbb R)$ and $t\in N^{-1}\mathbb Z_{\ge 0}$ let \begin{equation}\label{m_field}M_N(t,\phi):= \sum_{r=0}^{Nt}\sum_{x \in \mathbb Z-rN^{-1/4}} \phi(N^{-1/2}x) v_N(r,x).  
\end{equation}
\end{defn}
$M_N(t,\phi)$ is the macroscopic field corresponding to the microscopic variables $v_N(t,x)$. $M_N(t,\phi)$ is the discrete analog of $\mathcal Y(t,\phi)$ defined in \eqref{me1}. 
Note by \eqref{mfield} that $M_N(t,\phi)$ is a martingale indexed by $t\in N^{-1}\mathbb Z_{\ge 0}$. A substantial amount of effort in this paper will be spent studying the quadratic variations of $M_N(t,\phi)$, as we need to verify that \eqref{me2} holds for any limit point. By defining $\eta_N(r,x):=1-2 \omega_{r,x+N^{-1/4}r}$, we can use \eqref{mzero'} to collect neighboring terms and rewrite \eqref{m_field} as follows for $t\in N^{-1}\mathbb Z_{\ge 0}:$ \begin{equation}\label{grad_form}M_N(t,\phi) = \sum_{r=0}^{Nt} \sum_{x\in\mathbb Z - rN^{-1/4}}(\nabla_N \phi)\big(N^{-1/2}x\big) Z_N^\omega(r,x) \eta_N(r,x), \end{equation}
where $$(\nabla_N \phi)(x):= (1-\rho_N) \phi(x-N^{-1/2}) - \rho_N \phi(x+N^{-1/2}).$$ Consequently the optional quadratic variation is given for $t\in N^{-1}\mathbb Z_{\ge 0}$ by
    \begin{align}
         [M_N(\phi)]_t &= 
         \sum_{r=0}^{Nt}\bigg( \sum_{x \in \mathbb Z - rN^{-1/4}}(\nabla_N \phi)\big(N^{-1/2}x\big) Z_N^\omega(r,x) \eta_N(r,x)\bigg)^2.\label{optvar}
    \end{align}
    Since $\mathbb E[\eta_N(r,x)\eta_N(r,y)|\mathcal F^\omega_{r-1}]=4\sigma^2\ind_{\{x=y\}}$, the predictable quadratic variation of $M_N(\phi)$ is given for $t\in N^{-1}\mathbb Z_{\ge 0}$ by\begin{align}
        \notag \langle M_N(\phi)\rangle_t& = \sum_{r=0}^{Nt}\mathbb E\bigg[ \bigg( \sum_{x \in \mathbb Z - rN^{-1/4}}(\nabla_N \phi)\big(N^{-1/2}x\big) Z_N^\omega(r,x) \eta_N(r,x)\bigg)^2\bigg| \mathcal F^\omega_{r-1}\bigg] 
        \\ &= 4\sigma^2\sum_{r=0}^{Nt} \sum_{x \in \mathbb Z - rN^{-1/4}}\big[(\nabla_N \phi)\big(N^{-1/2}x\big) Z_N^\omega(r,x) \big]^2.\label{predvar} 
    \end{align}
    Note that $M_N(t,\phi)^2-[M_N(\phi)](t)$ and $M_N(t,\phi)^2-\langle M_N(\phi)\rangle(t)$ are both martingales indexed by $N^{-1}\mathbb Z_{\ge 0}$, which will be relevant in later sections. We decompose \eqref{predvar} as
   \begin{align} \langle M_N(\phi)\rangle_t &= \mathcal E_{N}(t,\phi)
   +\big((2\rho_N-1)^2\sqrt{N}\big)Q_N(t,\phi^2) \label{m=q}
    \end{align}where for $t\in N^{-1}\mathbb Z_{\ge 0}$ \begin{align}\notag Q_N(t,\phi) &:= \frac{4\sigma^2}{\sqrt N} \sum_{r=0}^{Nt} \sum_{x \in \mathbb Z-rN^{-1/4}} \phi(N^{-1/2}x)  Z_N^\omega(r,x)^2 \\ &= \frac{4\sigma^2}{\sqrt N} \sum_{r=0}^{Nt}  \mathsf E_{(2)}^\omega \bigg[ \phi( N^{-1/2} (S(r) -N^{-1/4} r)) \ind_{\{S(r)=R(r)\}} C_{N,N^{-1}r,N^{-1/2}(S(r)-N^{-1/4}r)}^2\bigg] \label{qfield}
    \end{align} 
    where $\mathsf E_{(2)}^\omega$ denotes quenched expectation for two independent motions $(R(r),S(r))_{r\ge 0}$ in the fixed realization of the environment $\omega$, and where $\mathcal E_{N}(T,\phi)$ 
    is an ``error term" given by 
    \begin{align}\label{e_n}
        \mathcal E_N(t,\phi) := 4\sigma^2 \sum_{r=0}^{Nt} \sum_{x\in \mathbb Z-rN^{-1/4}}\big[ (\nabla_N\phi)\big(N^{-1/2}x\big)^2 - (2\rho_N-1)^2\phi(N^{-1/2}x)^2 ] Z_N^\omega(r,x)^2.
    \end{align}

\begin{defn}We shall call $Q_N(t,\phi)$ the \textbf{Quadratic martingale field} (QMF), and $\mathcal E_N(t,\phi)$ the \textbf{error term}. \end{defn}It will be shown that $Q_N$ contributes meaningfully in the limit, while the error term $\mathcal{E}_N$ vanishes in the limit. Let us see why $\mathcal E_N$ should vanish. Notice that a first order Taylor expansion gives $\nabla_N\phi(x) = (1-2\rho_N)\phi(x) - O(N^{-1/2}),$ and $1-2\rho_N$ is of order $N^{-1/4}.$ Denote by $\|\phi\|_{L^\infty}$ the supremum of $\phi$ on $\mathbb R$, and denote by $\|\phi\|_{C^k}:=\sum_{j=0}^k \|\phi^{(j)}\|_{L^\infty}.$ Using the Taylor expansion, one verifies that the term in the square brackets of \eqref{e_n} is bounded above absolutely by $N^{-3/4}$ multiplied by $2\|\phi\|_{L^\infty}\|\phi'\|_{L^\infty} \leq \|\phi\|_{L^\infty}^2 + \|\phi'\|_{L^\infty}^2 \leq \|\phi\|_{C^1}^2.$ For $\phi \in C_c^\infty(\mathbb R)$ let $A_\phi:= \sup\{|x|:x\in supp(\phi)\}$, and note that the summands in \eqref{e_n} vanish whenever $|N^{-1/2}x|>A_\phi+1$. We have thus shown that the error term satisfies a pathwise bound written in \eqref{ebound} just below.

\begin{lem} For all $\phi\in C_c^\infty(\mathbb R)$ and $t\in N^{-1}\mathbb Z_{\ge 0}$ the predictable quadratic variation of $M_N(\phi)$ may be written as in \eqref{m=q}, where $Q_N$ and $\mathcal E_N$ are defined by \eqref{qfield} and \eqref{e_n} respectively. Furthermore, we have uniformly over all $s,t \in N^{-1}\mathbb Z_{\ge 0}$, $\phi \in C_c^{\infty}(\mathbb R)$ and $N\ge 1$ the bound
    \begin{align}\notag |\mathcal E_{N}(t,\phi)- \mathcal E_{N}(s,\phi)| &\leq 8N^{-3/4}\sigma^2 \|\phi\|_{C^1}^2\sum_{r=Ns}^{Nt} \sum_{x \in \mathbb Z-rN^{-1/4}}Z_N^\omega(r,x)^2\ind_{[-A_\phi-1,A_\phi+1]}(N^{-1/2}x) \notag \\&\leq 8N^{-1/4}\sigma^2 \|\phi\|_{C^1}^2\big(Q_N(t,\ind_{[-A_\phi-1,A_\phi+1]})-Q_N(s,\ind_{[-A_\phi-1,A_\phi+1]})\big), \label{ebound}
    \end{align}
    where $A_\phi:= \sup\{|x|:x\in supp(\phi)\}$.
    \end{lem}
    What this estimate shows is that the error term $\mathcal E_N$ is of the same form as $Q_N$, but with an extra factor of $N^{-1/4}$ in front. The field $Q_N$ will be shown to be tight in later sections, and therefore the irrelevance of $\mathcal E_N$ will then be immediate from \eqref{ebound} without any further work needed. This error calculation illustrates a remarkable property of the model under consideration, which is that the error terms behave very nicely in relation to the original object itself, which is rare among KPZ-related models where a martingale characterization has been used, see e.g. \cite{BG97, dembo, yang22} where extremely careful analysis was needed to show vanishing of error terms.

Now that the important objects $M_N$ and $Q_N$ have been introduced, we give a brief discussion on how the proof of Theorem \ref{main} will be completed. The next section is heavily devoted to studying the quadratic martingale field $Q_N(t,\phi)$. In Proposition \ref{tight1}, we will establish tightness-related bounds of the form 
$$\mE[|Q_N(t,\phi)-Q_N(s,\phi)|^p]^{1/p} \leq C\|\phi\|_{L^\infty}^2 |t-s|^{1/2}.$$ This will then lead to a similar regularity bound for $[ M_N(\cdot,\phi)]$ and $\mathscr U_N(\cdot,\phi)$ as well.

Together with these bounds, a Kolmogorov-type lemma will tell us that the triple $(\mathscr{U}_N,M_N,Q_N)$ is tight. This will be the subject of Section \ref{iden}. Since the prelimiting object satisfies a lattice stochastic heat equation, one can obtain that any limit point of the triple satisfies \eqref{me1} by showing that the discrete and continuous heat operators are close. 

To justify why the limit point satisfies the other part \eqref{me2} of the martingale problem, we study the QMF extensively in Section 5. In particular, 
informally speaking, we shall show the ``key estimate" that
\begin{align*}
    Q_N(t,\phi)-\frac{8\sigma^2}{1-4\sigma^2} \frac1N\sum_{s\in (N^{-1}\mathbb Z_{\ge 0})\cap [0,t]}\mathscr U_N(s,\phi)^2 \;\;\;\;\stackrel{L^2(\mathbb P)}{\longrightarrow}\;\;\;\; 0
\end{align*}
if we take $N\to \infty$ and then $\phi \to \delta_a$ in that order, for any $a\in \mathbb{R}\backslash\{0\}$. The precise statement is interpreted in terms of well-chosen Gaussian bump functions (see Proposition \ref{4.1}). This type of estimate will allow us to conclude that the limit point satisfies \eqref{me2}.

    \section{Formulas and estimates for the martingales}\label{sec5}
This section will heavily focus on obtaining crucial formulas and estimates for the quadratic martingale field (QMF) $Q_N$ defined in \eqref{qfield}. Later, these formulas and estimates will allow us to show the tightness of \eqref{field} and also identify the limit points. We first have two lemmas before stating the key estimate of this section. 

 \begin{lem}[Moment formulas]\label{l:Qmom} Fix any bounded functions $\psi,\phi$ on $\R$ and $N\ge 1$. Suppose that $(R^1,\ldots, R^{2k})$ denotes the canonical process on $(\mathbb Z^{2k})^{\mathbb Z_{\ge 0}}$. Recall $\Delta_k(s,t)$ from \eqref{deltadef}, and define $\Delta_k^N(s,t):= (N^{-1}\mathbb Z_{\ge 0})^k\cap \Delta_k(s,t)$. {For any $\vec{s}=(s_1,s_2,\ldots,s_k)\in \Delta_k(0,t)$, we define
 \begin{align}\label{barg}
     \bar{g}_N(\vec{s}) = \prod_{r\in N^{-1}\mathbb{Z}_{\ge 0}}\frac1{(\#\{i : s_i=r\})!}.
 \end{align}
 Note that $\#\{i : s_i=r\}=0$ for all but finitely many $r \in N^{-1}\mathbb{Z}_{\ge 0}$.}
 We have the following moment formulas.
		\begin{enumerate}[label=(\alph*),leftmargin=15pt]
			\item \label{l:QXmom} For all $t\in N^{-1}\mathbb Z_{\ge 0}$ and $\gamma>0$, we have 
			\begin{equation}
				\label{e:QXmom}
				\begin{aligned}
					& 
     \Ex\left[\bigg(Q_N(t,\psi)- \frac{\gamma}N\sum_{s\in (N^{-1}\mathbb Z_{\ge 0})\cap [0,t]} \mathscr U_N(s,\phi)^2 \bigg)^{k}\right] \\ & = k! N^{-k}\sum_{(s_1,\ldots,s_k)\in\Delta_k^N(0,t)}  {\bar{g}_N(\vec{s})} \mathbf E_{RW_{\nu}^{(2k)}}\bigg[ \prod_{i=1}^k C_{N,s_i,N^{-1/2}(R^{2i-1}({Ns_i})- N^{3/4} s_i)} \\ & \hspace{6cm}\cdot C_{N,s_i,N^{-1/2}(R^{2i}({Ns_i})- N^{3/4} s_i)}
     \Upsilon_N({Ns_i};R^{2i-1},R^{2i})\bigg],
				\end{aligned}
			\end{equation}
			where
			\begin{align*}
				\Upsilon_N(r;X,Y) & := 4N^{1/2}\sigma^2\psi\big(N^{-1/2}(X(r)-N^{-1/4}r)\big)\ind_{\{X(r)=Y(r)\}} \\ & \hspace{1cm}-\gamma \phi\big(N^{-1/2}(X(r)-N^{-1/4}r)\big)\phi\big(N^{-1/2}(Y(r)-N^{-1/4}r)\big).
			\end{align*}
			\item \label{l:Qincmom} For all $0\le s<t$ we have the following moment formula for the increment of the QMF
			\begin{equation}
				\label{e:Qincmom}
				\begin{aligned}
					& \Ex\left[\bigg(Q_N(t,\psi)-Q_N(s,\psi)\bigg)^{k}\right] \\ 
					&  = (4\sigma^2)^k N^{-k/2}k! \sum_{(s_1,\ldots,s_k)\in\Delta_k^N(s,t)} {\bar{g}_N(\vec{s})}\mathbf E_{RW_{\nu}^{(2k)}}\bigg[ \prod_{i=1}^k C_{N,s_i,N^{-1/2}(R^{2i}(Ns_i)- N^{3/4} s_i)}^2 \\ & \hspace{6cm}\psi(N^{-1/2}(R^{2i}(Ns_i)-N^{3/4}s_i))\ind_{\{R^{2i-1}(Ns_i)=R^{2i}(Ns_i)\}}\bigg].
				\end{aligned}
			\end{equation}
		\end{enumerate}
	\end{lem}

 \begin{proof} Note by \eqref{field} that
     \begin{align}\notag \mathscr U_N(t,\phi)^2& = \sum_{(x,y)\in \mathbb Z^2} C_{N,t,N^{-1/2}(x- N^{3/4} t)}C_{N,t,N^{-1/2}(y- N^{3/4} t)} \\ \notag &\;\;\;\;\;\;\;\;\;\;\;\;\;\;\;\;\;\cdot \mathsf P^\omega(Nt,x)\mathsf P^\omega(Nt,y) \phi\big( N^{-1/2}(x- N^{3/4} t)\big)\phi\big( N^{-1/2}(y- N^{3/4} t)\big) \\ &=\mathsf E_{(2)}^\omega\notag \big[C_{N,t,N^{-1/2}(S({Nt})- N^{3/4} t)}C_{N,t,N^{-1/2}(R({Nt})- N^{3/4}t)}\\ &\;\;\;\;\;\;\;\;\;\;\;\;\;\;\;\;\;\cdot\phi\big( N^{-1/2}(S({Nt})- N^{3/4}t)\big)\phi\big( N^{-1/2}(R({Nt})-N^{3/4} t)\big)\big]. \nonumber
\end{align}
Here $\mathsf E_{(2)}^\omega$ denotes quenched expectation for two independent motions $(R(r),S(r))_{r\ge 0}$ in the fixed realization of the environment $\omega$. Note that if we apply $\frac1N\sum_{s\in (N^{-1}\mathbb Z_{\ge 0})\cap [0,t]}$ to the last expression, then it is of a similar form to the definition \eqref{qfield} of $Q_N$, hence we have that 
\begin{align*}Q_N&(t,\psi)- \frac{\gamma}N\sum_{s\in (N^{-1}\mathbb Z_{\ge 0})\cap [0,t]} \mathscr U_N(s,\phi)^2  \\&= \frac1N\sum_{s\in (N^{-1}\mathbb Z_{\ge 0})\cap [0,t]}\mathsf E_{(2)}^\omega[C_{N,s,N^{-1/2}(R({Ns})- N^{3/4} s)}C_{N,s,N^{-1/2}(S({Ns})- N^{3/4} s)} \Upsilon_N (Ns;R,S)].
\end{align*}
From here one expands out the $k^{th}$ power of both sides of this equation, and then one applies the annealed expectation over the quenched expectation to deduce \eqref{e:QXmom}, {noting in general that by the multinomial theorem one has the following expansion for a function $f$:  $$\bigg(\sum_{s=0}^{r-1} f(s) \bigg)^k = k! \sum_{0\le s_1\le ...\le s_k<r} \bar{g}_N(\vec{s}) f(s_1)\cdots f(s_k).$$} The proof of \eqref{e:Qincmom} is similar. 
 \end{proof}

The following proposition will allow us to further analyze the moment formulas obtained in Lemma \ref{l:Qmom}.

 \begin{prop}\label{add} Assume that $\nu$ is a probability measure on $[0,1]$ of mean $1/2$. Fix any $k\in \mathbb N$. Let $\{\psi_i,\phi_i\}_{i=1}^{k}$  be bounded continuous functions on $\mathbb R$.

  Let $A$ be a subset of $\{1,2,\ldots, k\}$. Let $B=\{1,2\ldots,k\}\cap A^c$. Define
\begin{align}
       &  E_1(\vec{t}):=  \prod_{i\in A} C_{N,t_i,N^{-1/2}(R^{2i-1}(Nt_i)- N^{3/4} t_i)}\phi_i\big(N^{-1/2}(R^{2i-1}(Nt_i)- N^{3/4} t_i)\big) \\ & \hspace{3cm} \cdot \prod_{i\in A} C_{N,t_i,N^{-1/2}(R^{2i}(Nt_i)- N^{3/4} t_i)}\phi_i\big(N^{-1/2}(R^{2i}(Nt_i)- N^{3/4} t_i)\big) \\
   &  E_2(\vec{t}):=\prod_{i\in B} C_{N,t_i,N^{-1/2}(R^{2i}({Nt_i})-N^{3/4}t_i)}^2 \psi_i(N^{-1/2}(R_{}^{2i}(Nt_{i}) - N^{3/4}t_i))\ind_{\{R^{2i-1}({Nt_{i}}) =R^{2i}({Nt_{i}})\}} 
\end{align}
{Consider any uniformly bounded sequence of deterministic functions $g_N: \Delta_{k}^N(s,t)\to [0,\infty)$ that converge uniformly to 1 on the interior of the simplex (consisting of those points $(t_1,...,t_k)$ with all $t_i$ distinct) as $N\to \infty$. }For each $0\le s< t\le T<\infty$ we have
		\begin{equation}
			\label{e:add2}
			\begin{aligned}
				& \lim_{N\to\infty} N^{-|A|-\tfrac12|B|}(1-4\sigma^2)^{|B|}\cdot\mathbf E_{RW^{(2k)}_{\nu}}\bigg[\sum_{(t_1,\ldots,t_{k})\in \Delta_{k}^N(s,t)} {g_N(\vec{t})} E_1(\vec{t})E_2(\vec{t})\bigg]
      \\ & = \mathbf E_{B^{\otimes (2k)}} \bigg[\int_{\Delta_{k}(s,t)} e^{\frac{4\sigma^2}{1-4\sigma^2} \V_{k}(\vec{t})} \prod_{i\in A} \phi_i(U_{t_{i}}^{2i-1}) \phi_i(U_{t_{i}}^{2i})dt_i \cdot \prod_{i\in B} \psi_i(U_{t_{i}}^{2i}) dL_0^{U^{2i-1}-U^{2i}}(t_i) \bigg], 
			\end{aligned}
		\end{equation}
	where the expectation on the right is with respect to a $2k$-dimensional standard Brownian motion $(U^1,\ldots,U^{2k})$, and 
		\begin{align}
			\label{def:v}
			\V_k(\vec{t}):=\sum_{i=1}^k\sum_{2i-1\le p<q\le 2k} \left[L_0^{U^p-U^q}(t_{i})-L_0^{U^p-U^q}(t_{i-1})\right].
		\end{align}	Here  $\int_0^t f(s) dL_0^{U^i-U^j}(s)$ denotes the integration of the continuous function $f:[0,t]\to\mathbb R$ against the random Lebesgue-Stieltjes measure $dL_0^{U^i-U^j}$ induced from the increasing function $t\mapsto L_0^{U^i-U^j}(t).$

	\end{prop}

{We will be most interested in the case where $g_N(\vec{t}) =\bar{g}_N(\vec{t})$ defined in \eqref{barg}.}

 \begin{proof}   {We are going to assume that $g_N\equiv 1$; the general case follows by slight modifications.} 
Let $\nu_N^*$ be the \stt\ sequence of measures corresponding to $\nu$, as in Definition \ref{def:stt}.  For simplicity we write $\mE$ and $\mE_N$ for $\mE_{RW_\nu^{(2k)}}$ and $\mE_{RW_{\nu_N^*}^{(2k)}}$ respectively. Let us fix any $\vec{t}\in \Delta_k^N(s,t)$ and a subset $A$ of $\{1,2,\ldots,k\}$. Just as in the proof of moment convergence, the main idea of the proof is to first apply a change of measure to get rid of the divergent terms inside the expectation of $\mE[E_1(\vec{t})E_2(\vec{t})]$ and then use the weak convergence result from Theorem \ref{converge}. The change of measure here will be based on a glorified version of \eqref{tilt}.

\medskip

\noindent\textbf{Step 1. Post-processing the expectation via tilting.}
To carry out the tilting procedure, we first note that 
\begin{align}\label{tower}
\mE[E_1(\vec{t})E_2(\vec{t})] = \mE\left[\prod_{i=1}^k \widetilde{C}_N^i \Lambda_N^i\right]=\mE\left[\widetilde{C}_N^1\Lambda_N^1\mE\left[\widetilde{C}_N^2\Lambda_N^2 \cdots \mE\left[\widetilde{C}_N^k\Lambda_N^k\mid \mathcal{F}_{Nt_{k-1}}\right] \cdots \mid \mathcal{F}_{Nt_1}\right]\right]\!.\!
\end{align}
where
\begin{align*}
    \Lambda_N^i:=\begin{cases}
        \displaystyle \phi_i\big(N^{-1/2}(R^{2i-1}(Nt_i)- N^{3/4} t_i)\big)\phi_i\big(N^{-1/2}(R^{2i}(Nt_i)- N^{3/4} t_i)\big) & i\in A \\ \displaystyle  \psi_i(N^{-1/2}(R_{}^{2i}(Nt_{i}) - N^{3/4}t_i))\ind_{\{R^{2i-1}({Nt_{i}}) =R^{2i}({Nt_{i}})\}} & i\in B,
    \end{cases}
\end{align*}
and
\begin{align}
    \widetilde{C}_N^i:=\prod_{j=2i-1}^{2k} C_{N,t_i-t_{i-1},N^{-1/2}(R^j(Nt_i)-R^j(Nt_{i-1})-N^{3/4}(t_i-t_{i-1}))}.
\end{align}
We shall now apply tilting to each of the expectations on the r.h.s.~of \eqref{tower} and go from $\mathbf P$ to $\mathbf P_N$. To perform the tilting for conditional expectations we need generalized versions of $\mathpzc{M}^{\lambda}(\cdot)$ (defined in \eqref{m_n}) and $\widetilde{\mathcal{G}}_N$ (defined in \eqref{g}): 

  Let us define
\begin{align}
    \mathpzc M^{[k_1:k_2],\lambda}(r) = \exp\bigg( \lambda\sum_{j=k_1}^{k_2} R^j(r) - \sum_{\ell =0}^{r-1} f_{\ell}^{\lambda,k_2-k_1+1,\nu}(R^{k_1}(\ell),\ldots,R^{k_2}(\ell))\bigg)
\end{align}
The above martingale is obtained in the same spirit as  $\mathcal{M}^{\lambda}$ by considering only  the motion of 
 $(k_2-k_1+1)$ particles: $(R^j(\cdot))_{j=k_1}^{k_2}$.

Let us also define $\mathbf P_N$-martingales $\mathcal G^1_N,\ldots,\mathcal G^k_N$ by $\mathcal G^i_N =  \widetilde{\mathcal{H}}^i_N + \widetilde{\mathcal{D}}^i_N,$ where the $\mathbf P_N$-martingales $\widetilde{\mathcal{D}}^i_N$ and $\widetilde{\mathcal{H}}^i_N$ are defined by $\widetilde{\mathcal{H}}^i_N(0)=\widetilde{\mathcal{D}}^i_N(0)=0$, and 
     \begin{align*}\widetilde{\mathcal{H}}^i_N(r+1)-\widetilde{\mathcal{H}}^i_N(r) &= \begin{cases} \displaystyle N^{-1/4}\sum_{j=2i-1}^{2k}(R^j(r+1) -R^j(r) - (2\mu_{\nu_N^*}-1)), & Nt_i \le r< Nt_{i+1}\\ 0 & \mbox{otherwise} 
     \end{cases}.\\ \widetilde{\mathcal{D}}^i_N(r+1)-\widetilde{\mathcal{D}}^i_N(r) &= \begin{cases} \displaystyle \sum_{q=1}^{v} \big(\log (m_{b_q,n_q-b_q}) - \mathbf E_N [\log (m_{b_q,n_q-b_q}) |\mathcal F_{r}]\big), & Nt_i \le r< Nt_{i+1}\\ 0,  & \mbox{otherwise} 
     \end{cases}
     \end{align*} where $m_{b,n-b}$ is defined in \eqref{def:diff} and the parameters $v=v(i,r), n_q=n_q(i,r),$ and $b_q=b_q(i,r)$ are obtained deterministically from the path $(R^{2i-1},\ldots,R^{2k})$ as follows. Assume that at time $r$, the $2k-2i+2$ ``particles" of $(R^{2i-1}(r),\ldots, R^{2k}(r))\in \mathbb Z^{2k-2i+2}$ may be grouped into $v=v(i,r)$ disjoint groups with all particles in each distinct group at the same site in $\mathbb Z$. Assume that the $v$ respective groups contain $n_1,\ldots,n_v$ respective particles, where $n_q = n_q(i,r)$ are positive integers such that $n_1+\cdots+n_v=2k-2i$. From time $r$ to $r+1$ in the path $(R^{2i-1},\ldots,R^{2k})$, assume (for each $1\le q \le v$) in the $q^{th}$ group, that $b_q=b_q(i,r)$ of the $n_j$ particles go up one step, so that $n_j-b_j$ particles go down one step. 

     The processes $\mathcal{G}_N^i$ are martingales of the same form as those in the proof of Proposition \ref{g.exists}, but where only the last $2k-2i+2$ of the $2k$ particles are taken into account on the time interval $[Nt_i,Nt_{i+1})\cap \mathbb Z$. In particular, by the exact same argument as in that proof, each of the martingales $\mathcal G^i_N$ satisfies the bound \eqref{assn1} for some absolute constant $C>0$ independent of $N$. Define $W_N^i(r)=\sum_{s\leq r} \log \mathbf E_N[e^{\mathcal G_N^i(s)-\mathcal G_N^i(s-1)}\mid \mathcal{F}_{s-1}]$. Set $\widetilde{\mathcal{G}_N^i}:= \mathcal{G}_N^i-W_N^i$. We will now see that the exponentials of $\widetilde{\mathcal{G}_N^i}$ are the ``correct" martingales to tilt the expressions in the lemma statement in such a way that Theorem \ref{converge} is applicable.

Just like \eqref{tilt}, one can check that for each $i$ the conditional expectation $\mE[\mathbf{F} \mid \mathcal{F}_{Nt_i}]$ can be tilted as follows:
 \begin{align}\label{tilt2}
     \mathbf{E}[\mathbf{F} \mid \mathcal{F}_{Nt_i}]  =\mathbf{E}_N\left[\frac{\mathpzc M^{[2i-1:2k]}(Nt_{i-1})}{\mathpzc M^{[2i-1:2k]}(Nt_i)}\exp\left(\widetilde{\mathcal{G}_N^i}(Nt_i)\right)\mathbf{F}\mid \mathcal{F}_{Nt_i}\right],  
 \end{align}
where $\mathbf{F}$ is measurable with respect to~$\sigma(\{(R^{2i-1}(r),\ldots,R^{2k}(r))_{r\in [0,Nt]}\})$.

Using this, we tilt each of the conditional expectations in \eqref{tower} to get
\begin{align}
\label{newexp}\mE\left[E_1(\vec{t})E_2(\vec{t})\right]=\mE_N\left[\prod_{i=1}^k \frac{\mathpzc M^{[2i-1:2k]}(Nt_{i-1})}{\mathpzc M^{[2i-1:2k]}(Nt_i)}\exp\left(\widetilde{\mathcal{G}_N^i}(Nt_i)\right) \widetilde{C}_N^i {\Lambda}_N^i \right]
\end{align}

\medskip

\noindent\textbf{Step 2. Convergence.} We shall now study the weak convergence of the sum of random variables inside the expectation on the r.h.s.~of \eqref{newexp}.
We use the notations of rescaled processes from \eqref{def:resc}. Recalling the identity \eqref{twort} and the bound in \eqref{errbd} we see that
\begin{align*}
   & \prod_{i=1}^k\frac{\mathpzc M^{[2i-1:2k],N^{-1/4}}(Nt_{i-1})}{\mathpzc M^{[2i-1:2k],N^{-1/4}}(Nt_{i})}\widetilde{C}_N^i =\exp\left(4\sigma^2 \sum_{i=1}^k\sum_{2i-1<p<q\le 2k} [\mathscr{V}_N^{pq}(t_i)-\mathscr{V}_N^{pq}(t_{i-1})]-\widetilde{\mathsf{Err}}_N(\vec{t})\right).
\end{align*}
where
\begin{align}
    |\widetilde{\mathsf{Err}}_N(\vec{t})| \le \Con \left(\mathscr{T}_N(t)+N^{-1/2}\sum_{1\le p<q\le 2k} \mathscr{V}_N^{pq}(t)\right).
\end{align}
Let us define 
    $W_N^{\vec{t}}(s):=  W_N^1(Ns)+\cdots+W_N^k(Ns)$ and $\mathcal{G}_N^{\vec{t}}(s):=  \mathcal{G}_N^1(Ns)+\cdots+\mathcal{G}_N^k(Ns)$
for $s\in N^{-1}\mathbb{Z}$ and linearly interpolated for $s\notin N^{-1}\mathbb{Z}$. Let us view $W_N$ and $\mathcal{G}_N$ as random continuous functions from $\Delta_k(0,t)\times [0,t]\to \mathbb R$. Using the same arguments from the proof of Theorem \ref{converge}, it follows that  $W_N$ and $\{\mathcal{G}_N\}_{N\ge 1}$ are tight in the space of $C(\Delta_k(0,t)\times [0,t])$ equipped with uniform topology, and furthermore $\mathcal{G}_N^{\vec{t}}$ satisfies \eqref{assn1} with constant $C$ independent of $\vec t$. By Theorem \ref{converge} and Corollary \ref{convc} we know that any limit point as $N\to \infty$ of the sequence
\begin{align}\bigg(\mathcal{G}_N^{\vec{t}}, W_N^{\vec{t}}, \mathbf{X}_N, \mathscr{V}_N, \mathscr{T}_N, N^{-|B|/2}\hspace{-0.6cm}\sum_{u_1\le \cdots \le u_{|B|}\in N^{-1}\mathbb{Z}}\prod_{i\in B} \ind_{\{\mathbf{X}_N^{2i-1}(u_i)=\mathbf{X}_N^{2i}(u_i)\}}\delta_{(u_1,\ldots,u_{|B|})}\bigg)\label{tuple}\end{align}
(considered under the measure $\mathbf E_N$) is of the form
\begin{equation}\label{tuple2}\bigg(\mathcal{G}^{\vec{t}}, \tfrac12\langle \mathcal{G}^{\vec{t}}\rangle, \mathbf{U}, (L^{ij})_{1\le i<j\le 2k}, \mathbf{0}, (1-4\sigma^2)^{-|B|}\prod_{i\in B} dL^{2i-1,2i}(t_i)\bigg),\end{equation}
   where $\mathbf U$ is a standard Brownian motion in $\mathbb R^k$, where $L^{i,j}$ are its pairwise local times, and $\mathcal G^{\vec t}$ are martingales satisfying \eqref{e:fil2}. 
    Here the last coordinate has a topology of $\mathcal M(\Delta_k(0,T))$ which was defined in the statement of Corollary \ref{convc}, while all other coordinates have a uniform topology with respect to all relevant variables. Notice that the map from $\mathcal M(\Delta_k(0,T)) \times C(\Delta_k(0,T))\to \mathbb R$ given by $(f,\mu)\mapsto \int_{\Delta_k(0,T)} f\;d\mu$ is a continuous map. Using the continuous mapping theorem, we therefore see that for the limit point \eqref{tuple2} of \eqref{tuple} one has

\begin{align*}
   & N^{-|A|-\tfrac12|B|}(1-4\sigma^2)^{|B|}\sum_{\vec{t}\in \Delta_k^N(s,t)} \prod_{i=1}^k \frac{\mathpzc M^{[2i-1:2k]}(Nt_{i-1})}{\mathpzc M^{[2i-1:2k]}(Nt_i)}\exp\left(\widetilde{\mathcal{G}_N^i}(Nt_i)\right) \widetilde{C}_N^i {\Lambda}_N^i \\ & \hspace{2cm} \stackrel{d}{\longrightarrow} \int_{\Delta_k(s,t)} e^{\mathcal{G}^{\vec{t}}(t)-\frac12\langle \mathcal{G}^{\vec{t}}\rangle (t)+\frac{4\sigma^2}{1-4\sigma^2}\V_k(\vec{t})}\prod_{i\in A} \phi_i(U_{t_{i}}^{2i-1}) \phi_i(U_{t_{i}}^{2i})dt_i \cdot \prod_{i\in B} \psi_i(U_{t_{i}}^{2i})dL^{2i-1,2i}(t_i),
\end{align*}
where we again emphasize that the objects in the prelimit are viewed as observables under $\mathbf E_N$. Thanks to the estimates in Lemma \ref{lte}, Proposition \ref{exp}, and \eqref{gmass1}, we see that the prelimiting sum above is uniformly integrable. Now we finally prove the proposition. Taking into account the identity in \eqref{newexp}, the preceding observations imply that for every subsequence of indices $N\to \infty$, there is a further subsequence along which we have
    \begin{align*}
   & N^{-|A|-\tfrac12|B|}(1-4\sigma^2)^{|B|}\sum_{\vec{t}\in \Delta_k^N(s,t)} \mE[E_1(\vec{t})E_2(\vec{t})]
   \\ & \hspace{0.2cm} {\longrightarrow} \mE_{\text{lim}}\left[\int_{\Delta_k(s,t)} e^{\mathcal{G}^{\vec{t}}(t)-\frac12\langle \mathcal{G}^{\vec{t}}\rangle (t)+\frac{4\sigma^2}{1-4\sigma^2}\V_k(\vec{t})}\prod_{i\in A} \phi_i(U_{t_{i}}^{2i-1}) \phi_i(U_{t_{i}}^{2i})dt_i \cdot \prod_{i\in B} \psi_i(U_{t_{i}}^{2i})dL^{2i-1,2i}(t_i)\right] \\ & \hspace{0.2cm} =\mE_{\text{lim}}\bigg[\int_{\Delta_k(s,t)} \mE_{\text{lim}}[e^{\mathcal{G}^{\vec{t}}(t)-\frac12\langle \mathcal{G}^{\vec{t}}\rangle (t)}\mid \mathcal{F}_t(\mathbf{U})] \\ & \hspace{3cm}\cdot e^{\frac{4\sigma^2}{1-4\sigma^2}\V_k(\vec{t})}\prod_{i\in A} \phi_i(U_{t_{i}}^{2i-1}) \phi_i(U_{t_{i}}^{2i})dt_i \cdot \prod_{i\in B} \psi_i(U_{t_{i}}^{2i})dL^{2i-1,2i}(t_i)\bigg].
\end{align*}
Here $\mE_{\text{lim}}$ denotes a possible limit point on the canonical space, of the entire tuple of processes given by \eqref{tuple}, and we are viewing \eqref{tuple2} as the canonical process on that space. The inner conditional expectation in the last expression equals $1$ due to \eqref{e:fil2}. Thus the last expression must be equal to the right side of \eqref{e:add2}, completing the proof.
 \end{proof}

 \begin{prop}[Key estimate for the QMF] \label{4.1}Let $a \in \mathbb R$ let $\xi(x):= \frac1{\sqrt{\pi}}e^{-x^2}$, and let $\xi_\e^a (x):= \e^{-1}\xi(\e^{-1} (x-a))$.  
  Then for all $t>0$ and $a\in\mathbb R\setminus\{0\}$,
		\begin{align}
			\label{Qllim}
			\limsup_{\e \to 0} \limsup_{N \to \infty}  \Ex\bigg[ \bigg( Q_N(t,\xi_\e^a)-\frac{8\sigma^2}{1-4\sigma^2} \frac1N\sum_{s\in (N^{-1}\mathbb Z_{\ge 0})\cap [0,t]}\mathscr U_N(s,\xi_{\e\sqrt{2}}^a)^2\bigg)^2 \bigg] = 0.
		\end{align}
		Furthermore, we have the bound \begin{equation}\label{e:QXpolylog}\sup_{\substack{\e>0\\a\in\mathbb R\setminus\{0\}}} \limsup_{N \to \infty}  \left[1\wedge\frac1{|\log a|^{2}}\right]\!\cdot\!\Ex\bigg[ \bigg( Q_N(t,\xi_\e^a)-\frac{8\sigma^2}{1-4\sigma^2} \frac1N\sum_{s\in (N^{-1}\mathbb Z_{\ge 0})\cap [0,t]}\mathscr U_N(s,\xi_{\e\sqrt{2}}^a)^2\bigg)^2 \bigg] <\infty.
		\end{equation}
	\end{prop}

 The above proposition is the key estimate that will allow us to identify limit points. This proposition also illustrates why we cannot hope to obtain convergence in a space of continuous functions. Indeed if this were possible, then by first equality in the definition \eqref{qfield} of the field $Q_N$ it is easy to see that in \eqref{Qllim}, the correct coefficient would have to be $8\sigma^2 = 2\cdot 4\sigma^2$ rather than $\frac{8\sigma^2}{1-4\sigma^2}$ (where the extra factor of 2 comes from the fact that $Z_N^\omega(t,x)$ as defined in \eqref{z_n} are nonzero only at those points such that $t-x$ is even). Clearly, both coefficients cannot be correct unless $\mathscr U_N$ vanishes in the limit which is certainly not the case, as we have already shown that its moments of all orders converge to a nontrivial limit. 

 \begin{proof}
     Applying Lemma \ref{l:Qmom} \ref{l:QXmom} with $\gamma = \frac{8\sigma^2}{1-4\sigma^2}$, and Proposition \ref{add} with $k=2$, we get 
		\begin{align}
			\nonumber    & \lim_{N\to \infty}  \Ex\left[\bigg(Q_N(t,\psi)-\frac{8\sigma^2}{1-4\sigma^2} \frac1N\sum_{s\in (N^{-1}\mathbb Z_{\ge 0})\cap [0,t]} \mathscr U_N(s,\phi)^2\bigg)^{2}\right] \\ \nonumber & =2\bigg(\frac{8\sigma^2}{1-4\sigma^2}\bigg)^2 \cdot \mE_{B^{\otimes 4}}\left[\int_{\Delta_2(0,t)} e^{\frac{4\sigma^2}{1-4\sigma^2} \V_2(s_1,s_2)}\prod_{i=1}^2 \left(\psi(X_{s_i}^i) \tfrac12 dL_0^{X^i-Y^i}(s_i)-\phi(X^i)\phi(Y^i)ds_i\right)\right]   \\ & = 2\bigg(\frac{8\sigma^2}{1-4\sigma^2}\bigg)^2 \hspace{-0.1cm}\!\cdot\! \mE_{B^{\otimes 4}}\left[\int\limits_{\Delta_2(0,t)} \hspace{-0.3cm}e^{\frac{4\sigma^2}{1-4\sigma^2}\V_2(s_1,s_2)}\prod_{i=1}^2 \hspace{-0.1cm}\left(\psi\big(\tfrac12(X_{s_i}^i+Y_{s_i}^i)\big) \tfrac12 dL_0^{X^i-Y^i}(s_i)-\phi(X^i)\phi(Y^i)ds_i\right)\hspace{-0.1cm}\right] \label{e:smom1}
		\end{align}
		for all $\psi,\phi \in \mathcal{S}(\R)$, where $\V_2$ is defined in \eqref{def:v}, and $(X^1,X^2,Y^1,Y^2)$ is a 4d standard BM under $\mathbf P_{B^{\otimes 4}}$.
  The second equality in the above equation follows by observing that $X_u^i=Y_u^i$ for $u$ in the support of $L_0^{X^i-Y^i}(du)$. 
		
		\medskip
		
		We shall now write $\mE$ instead of $ \mE_{B^{\otimes 4}}$ for convenience. Let us now take  $$\psi(x) := \xi^a_\e(x)= \frac1{\sqrt{\pi \e^2 }}e^{-(x-a)^2/\e^2}, \quad \phi(x) := \xi^a_{\e\sqrt{2}}(x)= \frac1{\sqrt{2\pi \e^2 }}e^{-(x-a)^2/2\e^2},$$
		in \eqref{e:smom1}. Using the identity $\xi^a_{\e\sqrt{2}}(x)\xi^a_{\e\sqrt{2}}(y)= \xi^a_\e((x+y)/2)\xi^0_{2\e}(x-y)$, we may now write \eqref{e:smom1} as 
		\begin{align}
			\label{e:smom2}
			2\bigg(\frac{8\sigma^2}{1-4\sigma^2}\bigg)^2\hspace{-0.1cm} \!\cdot\! \mE\left[\int\limits_{\Delta_2(0,t)} \hspace{-0.2cm}e^{\frac{4\sigma^2}{1-4\sigma^2}\V_2(s_1,s_2)}\prod_{i=1}^2 \left(\xi_\e^a\big(\tfrac12(X_{s_i}^i+Y_{s_i}^i)\big)\!\cdot\!(\tfrac12 dL_0^{X^i-Y^i}(s_i)-\xi_{2\e}^{0}(X_{s_i}^i-Y_{s_i}^i)ds_i)\right)\hspace{-0.1cm}\right]\!.\!
		\end{align}
		Let us write $U^{i,-}:=X^i-Y^i$ and $U^{i,+}:=X^i+Y^i$. Note that under $\mathbf P_{B^{\otimes 4}}$ the four processes $U^{1,-},U^{1,+},U^{2,-},U^{2,+}$ are independent Brownian motions with diffusion coefficient $2$. This enables us to view \eqref{e:smom2} as
		\begin{align}
			\nonumber & 2\bigg(\frac{8\sigma^2}{1-4\sigma^2}\bigg)^2 \cdot \mE\left[\int_{\Delta_2(0,t)} e^{\frac{4\sigma^2}{1-4\sigma^2}\V_2(s_1,s_2)}\prod_{i=1}^2 \left(\xi_\e^a\big(\tfrac12U_{s_i}^{i,+})\big)\cdot(\tfrac12 dL_0^{U^{i,-}}(s_i)-\xi_{2\e}^{0}(U_{s_i}^{i,-})ds_i)\right)\right]
			\\ & =: 2\bigg(\frac{8\sigma^2}{1-4\sigma^2}\bigg)^2 [A_1(\e)-A_2(\e)-A_3(\e)+A_4(\e)], \label{e:smom3}
		\end{align}
		where
		\begin{align*} 
			A_1(\e) & := \mE\left[\int_{\Delta_2(0,t)}  e^{\frac{4\sigma^2}{1-4\sigma^2}\V_2(s_1,s_2)}\xi_\e^a\big(\tfrac12U_{s_1}^{1,+}\big)\xi_\e^a\big(\tfrac12U_{s_2}^{2,+}\big)\,\tfrac12 dL_0^{U^{1,-}}(s_1)\,\tfrac12dL_0^{U^{2,-}}(s_2)\right] \end{align*}
   \begin{align*}
			A_2(\e) & := \mE\left[\int_{\Delta_2(0,t)} e^{\frac{4\sigma^2}{1-4\sigma^2}\V_2(s_1,s_2)}  \xi_\e^a\big(\tfrac12U_{s_1}^{1,+}\big)\xi_\e^a\big(\tfrac12U_{s_2}^{2,+}\big)\xi_{2\e}^{0}(U_{s_2}^{2,-})\, \tfrac12 dL_0^{U^{1,-}}(s_1)\, ds_2\right]
   \end{align*}
   \begin{align*}
			A_3(\e) & := \mE\left[\int_{\Delta_2(0,t)} e^{\frac{4\sigma^2}{1-4\sigma^2}\V_2(s_1,s_2)} \xi_\e^a\big(\tfrac12U_{s_1}^{1,+}\big)\xi_\e^a\big(\tfrac12U_{s_2}^{2,+}\big)\xi_{2\e}^{0}(U_{s_1}^{1,-})\, \tfrac12 dL_0^{U^{2,-}}(s_2)\, ds_1\right]
   \end{align*}
   \begin{align*}
			A_4(\e) & := \mE\left[\int_{\Delta_2(0,t)} e^{\frac{4\sigma^2}{1-4\sigma^2}\V_2(s_1,s_2)}\xi_\e^a\big(\tfrac12U_{s_1}^{1,+}\big)\xi_\e^a\big(\tfrac12U_{s_2}^{2,+}\big)\xi_{2\e}^{0}(U_{s_1}^{1,-})\xi_{2\e}^{0}(U_{s_2}^{2,-})\,ds_1\,ds_2\right].
		\end{align*}
  From here, the goal is to show that \eqref{e:smom3} vanishes as $\varepsilon\to 0$ as long as $a\neq 0$, as well as establish a bound given by the right side of \eqref{e:QXpolylog} for each of the four terms $A_i(\varepsilon)$. 
  
  Note that informally $\tfrac12 dL_0^{U^{i,\pm}}(s_i)$ may be written as $\delta_0(U^{i,\pm}) ds_i$ which suggests that each of the $A_i$ may be written in terms of Brownian bridge expectations. Indeed this is the case, and consequently the proofs of the desired convergence statements and bounds for the terms $A_i(\varepsilon)$ rely purely on elementary (albeit lengthy) disintegration formulas for Brownian motion at its endpoint, and these proofs can be copied verbatim from \cite[Proof of Proposition 5.3: Steps 2 and 3]{DDP23}, replacing the coefficient $\sigma$ appearing there with our coefficient $\frac{8\sigma^2}{1-4\sigma^2}$ throughout the proof. For brevity, we do not reproduce the details here.
 \end{proof}

 With the ``key estimate" proved, next we focus on obtaining bounds that will be useful for proving the tightness of the rescaled field \eqref{field}.

 \begin{prop}[Estimates for moments of the increments of QMF]\label{tight1} Fix $k\in\mathbb N$ and $T>0$. Then there exists a constant $\Con=\Con(k,T)>0$ such that for all bounded measurable functions $\phi$ on $\mathbb R$ and all $0\leq s<t\leq T$ with $s,t\in  N^{-1}\mathbb Z_{\ge 0}$ one has that 
		\begin{align}
			\label{e.tight1}
			\sup_{N\ge 1}\Ex\big[ (Q_N(t,\phi)-Q_N(s,\phi))^{k}\big] \leq \Con\|\phi\|^{k}_{L^\infty(\mathbb R)}(t-s)^{k/2}.
		\end{align}
		Furthermore fix $p>1$ and $\e>0$. Then there exists $\Con=\Con(p,\e,k,T)>0$ such that for all functions $\phi \in L^p(\R)$ and all $\e \leq s<t\leq T$ one has 
		\begin{align}
			\label{e.tight2}
			\lim_{N\to\infty}\Ex\big[ (Q_N(t,\phi)-Q_N(s,\phi))^{k}\big] \leq \Con\|\phi\|_{L^p(\mathbb R)}^k(t-s)^k.
		\end{align}
	\end{prop}

 \begin{proof}
We are going to use the same notation and the same family of martingales from the proof of Proposition \ref{add}. Using \eqref{e:Qincmom} together with the trivial bound $|\phi(N^{-1/2}(R^j(Nt_j)-N^{3/4}t_j))|\leq \|\phi\|_{L^\infty},$ we obtain that 
		\begin{align}
			 \notag &\Ex[\big(Q_N(t,\phi)-Q_N(s,\phi)\big)^k]  \\ & \notag \le  N^{-k/2} k!\|\phi\|_{L^\infty}^k\sum_{(t_1,\ldots,t_k)\in \Delta_k^N(s,t)} \mathbf E_{RW^{(2k)}_{\nu}}\bigg[\prod_{j=1}^kC_{N,t_j,N^{-1/2}(R^{2j}({Nt_j})- N^{3/4} t_j)}^2  \ind_{\{R^{2j-1}(Nt_j) =R^{2j}(Nt_j)\}} \bigg]\\ & =  k!\|\phi\|_{L^\infty}^k \cdot N^{-k/2}\sum_{(t_1,\ldots,t_k)\in\Delta_k^N(s,t)}\mathbf E_{RW^{(2k)}_{\nu_N^*}}\bigg[\mathbf{A}_N(\vec{t})\cdot \mathbf{B}_N(\vec{t}) \cdot \prod_{j=1}^k  \ind_{\{R^{2j-1}(Nt_j) =R^{2j}(Nt_j)\}} \bigg], \label{eqg}
		\end{align}
  where
  \begin{align*}
     \mathbf{A}_N(\vec{t})&:=\exp\bigg(\mathcal G_N^{\vec{t}}(t) -W_N^{\vec{t}}(t)\bigg), \\\quad \mathbf{B}_N(\vec{t})&:=\exp\left(4\sigma^2 \sum_{i=1}^k\sum_{2i-1<p<q\le 2k} [\mathscr{V}_N^{pq}(t_i)-\mathscr{V}_N^{pq}(t_{i-1})]-\widetilde{\mathsf{Err}}_N(\vec{t})\right),
 \end{align*}
and where all of these objects are exactly the same as those introduced in the proof of Proposition \ref{add}. The equality in \eqref{eqg} is due to the argument given in the proof of Proposition \ref{add}. As before, let us abbreviate $\mathbf E_N:=\mathbf E_{RW^{(2k)}_{\nu_N^*}}.$ Recalling the measure $\gamma_N$ introduced in \eqref{gamman}, let us rewrite \eqref{eqg} as $$k!\|\phi\|_{L^\infty}^k \mathbf E_N \bigg[ \int_{\Delta_k^N(s,t)} \mathbf{A}_N(\vec{t})\cdot \mathbf{B}_N(\vec{t}) \gamma_N(d\vec t)\bigg].$$ 
Now using a trivial bound $$\int_{\Delta_k^N(s,t)} \mathbf{A}_N(\vec{t})\cdot \mathbf{B}_N(\vec{t}) \gamma_N(d\vec t)\le  \|\mathbf A_N\|_{L^\infty(\Delta_k^N(s,t))} \|\mathbf B_N\|_{L^\infty(\Delta_k^N(s,t))} \gamma_N(\Delta_k^N(s,t)),$$
we have by Holder's inequality that 
\begin{align}
   & \mathbf E_N \bigg[ \int_{\Delta_k^N(s,t)} \mathbf{A}_N(\vec{t})\cdot \mathbf{B}_N(\vec{t}) \gamma_N(d\vec t)\bigg] \\ & \hspace{2cm} \leq \mathbf E_N\big[ \|\mathbf A_N\|_{L^\infty(\Delta_k^N(s,t))}^3\big]^{1/3}\mathbf E_N \big[ \|\mathbf B_N\|_{L^\infty(\Delta_k^N(s,t))}^3\big]^{1/3}\mathbf E_N\big[ \gamma_N(\Delta_k^N(s,t))^3\big]^{1/3}. 
\end{align}
 Recalling $\mathscr V^{ij}$ from Theorem \eqref{def:resc}, note that $\gamma_N(\Delta_k^N(s,t)) \leq \big( \sum_{1\leq i<j\leq k} \mathscr V^{ij}_N(t) - \mathscr V^{ij}_N(s)\big)^k. $ By Lemma \ref{lte}, we can bound $\mathbf E_N\big[ \gamma_N(\Delta_k^N(s,t))^3\big]^{1/3} \leq C(t-s)^{k/2}$ where $C$ does not depend on $N$. This is because of the fact that $\mu_{\nu_N^*}(1-\mu_{\nu_N^*}) - \sigma_{\nu_N^*}^2$ is positive for large enough $N$. Thus it suffices to bound $I_1,I_2$ independently of $N$. 

Note that the martingales $\mathcal G_N^{\vec{t}}(t)$ satisfy \eqref{assn1}, by their construction in the proof of Proposition \ref{add}. Consequently an application of Doob's $L^p$ inequality combined with Proposition \ref{exp} 
allows us to bound the expected maximum $\mathbf E_N\big[ \|\mathbf A_N\|_{L^\infty(\Delta_k^N(s,t))}^3\big]^{1/3}$ independently of $N$. Finally using Lemma \ref{traps} we can bound $\mathbf E_N \big[ \|\mathbf B_N\|_{L^\infty(\Delta_k^N(s,t))}^3\big]^{1/3}$ for large enough $N$ as well (as again $\mu_{\nu_N^*}(1-\mu_{\nu_N^*}) -\sigma_{\nu_N^*}^2$ is positive for large enough $N$). This verifies the bound in \eqref{e.tight1}.

\medskip

For \eqref{e.tight2}.  appealing to the moment formula for the increment of QMF from \eqref{e:Qincmom} and the convergence from \eqref{e:add2} we have
		\begin{align}
  \label{e.idef}
  & \lim_{N\to \infty}\Ex\bigg[\big(Q_N(t,\phi)-Q_N(s,\phi)\big)^k\bigg] \\ & \hspace{2cm}= k!\bigg(\frac{4\sigma^2}{1-4\sigma^2}\bigg)^k \cdot \mathbf E_{B^{\otimes 2k}} \bigg[\int_{\Delta_k(s,t)}e^{\frac{4\sigma^2}{1-4\sigma^2}\V_k(\vec{t})}\prod_{j=1}^k \phi(U_{t_j}^{2j})dL_0^{U^{2j-1}-U^{2j}}(t_j) \bigg],
		\end{align}
  where $\V_k(\vec{t})$ is defined in \eqref{def:v}. The rest of the proof is analogous to the proof of  \cite[Eq.~(5.43)]{DDP23}.
 \end{proof}

Recall the martingale field $M_N$ from \eqref{m_field}. The next estimate will obtain a $L^p$ bound on its optional quadratic variation from \eqref{optvar}, which will be the main tool in obtaining tightness for the field \eqref{field}. 

    \begin{prop}[Optional quadratic variation bound] \label{optbound}
        Fix $p\ge 1$. For all $s,t\in N^{-1}\mathbb Z_{\ge 0}$ and all $\phi\in C_c^\infty(\mathbb R)$ we have that 
         \begin{equation}\mathbb E\big[ \big|[M_N(\phi)]_{t}-[M_N(\phi)]_s\big|^p\big]^{1/p}\leq C \|\phi\|_{C^1}^2 |t-s|^{1/2}. \label{o2}
        \end{equation}
        Here $C$ is independent of $N,\phi,s,t$.
    \end{prop}

    \begin{proof}

        The main idea of the proof will be to split the quadratic variation $[M_N(\phi)]$ into a ``predictable part" (denoted $A_N$ below) which is easy to control, and a ``discontinuity part" (denoted $B_N$ below) which we expect to vanish as $N\to\infty$. Then we bound each of these separately.
        
        It suffices to prove the claim when $p=2k$ for some positive integer $k$. Letting $\eta_{N}(r,x):= 1-2\omega_{r,x+N^{-1/4}r}$, recall from \eqref{grad_form} that for $t\in N^{-1}\mathbb Z_{\ge 0}:$ $$M_N(t,\phi) = \sum_{r=0}^{Nt} \sum_{x\in\mathbb Z - rN^{-1/4}}(\nabla_N \phi)\big(N^{-1/2}x\big) Z_N^\omega(r,x) \eta_{N}(r,x). $$
        where $(\nabla_N \phi)(x):= (1-\rho_N) \phi(x-N^{-1/2}) - \rho_N \phi(x+N^{-1/2}).$ A first-order Taylor expansion of $\phi$ yields that $\|\nabla_N\phi\|_{L^\infty(\mathbb R)}\leq \|\phi\|_{C^1}N^{-1/4} $. Then for $t\in N^{-1}\mathbb Z_{\ge 0}$  we have that 
\begin{align*} \notag (M_{N}&(t+N^{-1},\phi)- M_N(t,\phi))^2 = A_N(Nt,\phi)+B_N(Nt,\phi). 
        \end{align*}
        where
        \begin{align*}
            A_N(r,\phi) & :=\sum_{x\in \mathbb Z - rN^{-1/4}} (\nabla_N \phi)\big(N^{-1/2}x\big)^2 Z_N^\omega(r,x)^2 \eta_{N}(r,x)^2, \\
            B_N(r,\phi) & :=\sum_{\substack{x,y\in \mathbb Z - rN^{-1/4}\\x\neq y}} (\nabla_N \phi)\big(N^{-1/2}x\big) (\nabla_N \phi)\big(N^{-1/2}y\big) Z_N^\omega(r,x)Z_N^\omega(r,y) \eta_{N}(r,x)\eta_N(r,y).
        \end{align*}

     We will separately obtain bounds of the desired form for both $A_N$ and $B_N$. 
     Using the fact that $\eta_N(r,x)^2 \leq 1$ deterministically, 
        and using the definition \eqref{qfield} of $Q_N$, we obtain the pathwise bound \begin{equation*} \sum_{r=Ns}^{Nt}A_N(r,\phi) \leq \sigma^{-2}\|\phi\|_{C^1}^2 \big( Q_N(t,\ind) - Q_N(s,\ind)\big).\end{equation*}
        In particular, by \eqref{e.tight1} we see that \begin{equation}\label{an} \mathbb E \bigg[ \bigg(\sum_{r=Ns}^{Nt}A_N(r,\phi) \bigg)^{2k}\bigg] \leq\|\phi\|_{C^1}^4\mathbb E\big[ \big( Q_N(t,\ind) - Q_N(s,\ind)\big)^{2k} \big] \leq C \|\phi\|_{C^1}^{2k}|t-s|^k.\end{equation}
        Now let us bound the moments of  $B_N$. 
        Note that
\begin{align*}
    (B_N(r,\phi))^2 & \le N^{-1}\|\phi\|_{C^1}^4\bigg(\sum_{x} Z_N^{\omega} (r,x)\bigg)^4=N^{-1}\|\phi\|_{C^1}^{4}\cdot\left(\mathsf{E}^{\omega}\big[C_{N,N^{-1}r, N^{-1/2}(R(r)-N^{-1/4}r)}\big]\right)^4.
\end{align*}
Taking the annealed expectation of a product of $B_N$'s we get
\begin{align}
    \Ex\left[\prod_{j=1}^k B_N(r_j,\phi)^2\right] & \le N^{-k}\|\phi\|_{C^1}^{4k}\cdot \Ex\left[\prod_{j=1}^k\left(\mathsf{E}^{\omega}\big[C_{N,N^{-1}r_j, N^{-1/2}(R(r_j)-N^{-1/4}r_j)}\big]\right)^4\right].
\end{align}
By H\"older's inequality the right-hand side of the above equation is less than
\begin{align}
    N^{-k}\|\phi\|_{C^1}^{4k}\prod_{j=1}^k\left(\mathbf{E}_{RW_{\nu}^{(4k)}}\bigg[\prod_{i=1}^{4k}C_{N,N^{-1}r_j, N^{-1/2}(R_i(r_j)-N^{-1/4}r_j)}\bigg]\right)^{1/k}.
\end{align}
Using the same arguments as in the proof of Proposition \ref{tight1}, the above expectation can be bounded uniformly over $r_j$'s in $[0,Nt]$.
        Using the fact that $\mathbb E[B_N(r,\phi)|\mathcal F^\omega_{r-1}]=0,$ we then obtain from Burkholder-Davis-Gundy that
        \begin{align*}
            \mathbb E \bigg[ \bigg(\sum_{r=Ns}^{Nt} &B_N(r,\phi) \bigg)^{2k} \bigg]  \leq C_k \mathbb E\bigg[\bigg(\sum_{r=Ns}^{Nt} B_N(r,\phi)^2\bigg)^k\bigg] \\ &= C_k\mathbb E \bigg[ \sum_{r_1,...,r_k=Ns}^{Nt}  B_N(r_1,\phi)^2\cdots B_N(r_k,\phi)^2 \bigg] \leq C_k\|\phi\|_{C^1}^{4k}|t-s|^k.
        \end{align*}
        This gives the desired claim after combining this bound with the other bound \eqref{an} for $A_N(r,\phi)$.
    \end{proof}

    \section{Tightness and identification of limit points}\label{iden}

    Throughout this section, we are going to fix a terminal time $T>0$. To prove Theorem \ref{main} we need to introduce a number of preliminaries. 
    
    \begin{defn}Recall the lattice $\Lambda_N$ from \eqref{lattice} and define its rescaled version $$\Psi_{N,T}:= \{ (t,x)\in \mathbb R^2 : (Nt,N^{1/2}x) \in \Lambda_N, \;\;\;0\le t\leq T\}.$$ 
    \end{defn}
    
    \begin{defn}\label{dnlnkn} For $(s,y) \in \Lambda_N$ define $p_N(s,y)$ to be the transition density at time $s$ and position $y$ of a random walker on $\Lambda_N$ with increment distribution given by going from $0$ to $-1+N^{-1/4}$ with probability $\rho_N$ (as defined in \eqref{def:rhon}), and from $0$ to $1+N^{-1/4}$ with probability $1-\rho_N$. 
    
    Define linear operators $D_N,L_N,K_N$ on $\mathcal S'(\mathbb R^2)$ by
    \begin{align*} D_Nf(t,x) &:= N \big[ f(t+N^{-1},x)-f(t,x)\big],\\
        L_N f(t,x) &: = N\big[f(t+N^{-1},x-N^{-3/4}) - \big(\rho_Nf(t,x-N^{-1/2}) + (1-\rho_N) f(t,x+N^{-1/2}) \big)\big],\\
        K_N f(t,x) &:= N^{-1}\sum_{(s,y) \in \Psi_{N,T+1}} p_N(Ns,N^{1/2}y) f(t-s,x-y).
    \end{align*}
    These equalities should be understood by integration against smooth functions $\varphi \in \mathcal S(\mathbb R^2)$.
    \end{defn}

    The sum defining $K_Nf$ is well defined because it is actually a finite sum, as $p_N$ is finitely supported for bounded time intervals. Intuitively it is clear that $L_N$ is a diffusively rescaled version of the discrete heat operator $\mathcal L_N$ from Section \ref{hopf}, but which acts on tempered distributions rather than functions on $\Lambda_N$. Indeed if $\varphi \in \mathcal S(\mathbb R^2)$ then a second-order Taylor expansion shows that $L_N\varphi$ converges in $\mathcal S(\mathbb R^2)$ as $N\to\infty$ to $(\partial_t-\frac12\partial_x^2)f$, i.e., $L_N$ approaches the continuum heat operator. Then $K_N$ is the inverse operator to $L_N$, in the following sense.
    
    \begin{lem}\label{inv} $L_NK_Nf=K_NL_N f=f$ whenever $f$ is a tempered distribution supported on $[a,b]\times \mathbb R$ with $b-a<T+1$.
    \end{lem}

    \begin{proof}
        Note that for each $N$ the operators $K_N,L_N$ are continuous on $\mathcal S'(\mathbb R^2)$, since they are finite linear combinations of translation operators.  Therefore it suffices to prove the claim for all smooth functions $f$ that have compact support contained in $(a,b)\times \mathbb R$, since these are dense in the subset of distributions supported on $[a,b]\times \mathbb R$, with respect to the topology of $\mathcal S'(\mathbb R^2)$. The smooth claim is true by direct calculations since $p_N(s,y)$ is the kernel for the inverse operator to the discrete heat operator $\mathcal L_N$ introduced in Section \ref{hopf}.
    \end{proof}

\begin{defn}\label{mqv} 
    With the fields $M_N$ and $Q_N$ as defined in \eqref{m_field} and \eqref{qfield} respectively, 
    we will associate random elements of $C([0,T+1],\mathcal S'(\mathbb R))$ by the formulas $\hat M_N(t,\phi)=0$ for $t\in [0,N^{-1}]$ and
    \begin{align*}\hat{M}_N(t,\phi) &:= M_N(t-N^{-1},\phi)-M_N(0,\phi),\;\;\;\;\; t\in N^{-1}\mathbb Z_{\ge 1}.
    \\ \hat{Q}_N(t,\phi)&:= Q_N(t,\phi),\;\;\;\;\; t\in N^{-1}\mathbb Z_{\ge 0}.
    \end{align*}
    for $\phi \in \mathcal S(\mathbb R)$. We define these fields by linear interpolation for $t\notin N^{-1}\mathbb Z_{\ge 0}.$
    \end{defn}

    For any $T$, note that $C([0,T],\mathcal S'(\mathbb R))$ embeds naturally into the linear subspace of $\mathcal S'(\mathbb R^2)$ consisting of distributions supported on $[0,T]\times \mathbb R$, thus we can make sense of $D_Nf,L_Nf,K_Nf$ for all $f\in C([0,T],\mathcal S'(\mathbb R)),$ and these will be elements of $\mathcal S'(\mathbb R^2)$ in general.

    \begin{defn}
        We will say that two tempered distributions $f,g \in \mathcal S'(\mathbb R^2)$ agree on $[0,T]$ if there exists $\e>0$ such that $(f,\varphi) = (g,\varphi)$ for all $\varphi$ supported on $[-\e,T+\e] \times \mathbb R.$
    \end{defn}

    \begin{defn}
        Sample the environment $\omega$ and then define a collection of coefficients $a_N(s,y)$ with $(s,y)\in \Lambda_N$ as $a_N(0,y) = \ind_{\{y=0\}}$ and for $s\geq 1$, $$a_N(s,y) = \omega_{0,0} p_N(s-1,y-1-N^{-1/4}) + (1-\omega_{0,0}) p_N(s-1, y+1-N^{-1/4}).$$ 
        Then define the distribution $\mathfrak p_N\in C([0,T+1],\mathcal S'(\mathbb R))$ for $t\in N^{-1}\mathbb Z_{\ge 0}$ by $$\mathfrak p_N(t,\cdot) := N^{-1} \sum_{y:(t,y) \in \Psi_{N,T+1}} a_N(Nt,N^{1/2}y)  \delta_y$$and linearly interpolated for $t\notin N^{-1}\mathbb Z_{\ge 0}$.
    \end{defn}
    Although its definition is tedious, $\mathfrak p_N$ admits the following simple description: it agrees with the density field $\mathscr U_N$ from \eqref{field} on the time interval $[0,N^{-1}]$, and it is a linearly interpolated solution of the discrete heat equation $\mathcal L_Na_N=0$ at all positive times in $N^{-1}\mathbb Z_{\ge 0}$. The purpose of this will be to make certain technical details work later. Since $\mathfrak p_N$ behaves deterministically after time $N^{-1}$, its scaling limit will simply be the continuum heat kernel.

    \begin{lem}\label{u=kdm}
        Let $\hat M_N$ be as defined above. Furthermore, let $\mathscr U_N$ be as defined in \eqref{field}. Restrict $\mathscr U_N$ to $[0,T]$ thus viewed as an element of $C([0,T],\mathcal S'(\mathbb R))$. Then $L_N(\mathscr U_N-\mathfrak p_N)$ agrees with $D_N\hat{M}_N$ on $[0,T]$. Consequently $\mathscr U_N$ agrees with $\mathfrak p_N+K_N D_N\hat{M}_N$ on $[0,T]$.
    \end{lem}

    \begin{proof}  Note that $\mathscr U_N-\mathfrak p_N$ and $\hat M_N$ both vanish on $[0,N^{-1}]$, consequently $D_N\hat M_N (0,\phi) = L_N(\mathscr U_N-\mathfrak p_N)(0,\phi)=0$ for all $\phi\in C_c^\infty(\mathbb R)$. On the other hand, $L_N \mathfrak p_N (t,\cdot) = 0$ for $t\in N^{-1}\mathbb Z_{\ge 0}$ with $t>0$, consequently $ L_N(\mathscr U_N-\mathfrak p_N)(t,\phi)= L_N\mathscr U_N(t,\phi)$ for such $t$. 
    
     Recall $v_N$ defined in \eqref{mart'}. Then it is clear from \eqref{m_field} that 
    $$D_N\hat{M}_N(t,\phi) = N\big[ M_N(t,\phi)-M_{N}(t-N^{-1},\phi)\big]=\bigg(\sum_{x: (t,x)\in \Psi_{N,T}} N v_N(Nt,N^{1/2}x) \delta_x\;\;,\;\;\phi\bigg),$$
    for all $ t\in N^{-1}\mathbb Z_{\ge 1}$ and $\phi\in \mathcal S(\mathbb R)$.
    Now, with $Z^\omega_N$ defined in \eqref{z_n} we have that for $t\in N^{-1}\mathbb Z_{\ge 0}$ we have $$\mathscr U_N(t,\cdot) = \sum_{x: (t,x)\in \Psi_{N,T}} Z^\omega_N(Nt,N^{1/2}x) \delta_x.$$ With $\mathcal L_N$ as defined in Section \ref{hopf}, it is then clear that for $t\in N^{-1}\mathbb Z_{\ge 0}$ the expression for $L_N\mathscr U_N(t,\cdot)$ can be written as \begin{align*} 
    & N\sum_{x: (t,x)\in \Psi_{N,T}} (\mathcal L_NZ^\omega_N)(Nt,N^{1/2}x)\delta_x = N\sum_{x: (t,x)\in \Psi_{N,T}} v_N(Nt,N^{1/2}x)\delta_x,
    \end{align*}
        which is the same as the expression for $D_N\hat{M}_N(t,\phi).$  Since linear operators respect linear interpolation, and since all of the fields $\mathscr U_N, D_N\hat{M}_N, \mathfrak{p}_N$ are defined via linear interpolation, we have $L_N(\mathscr{U}_N-\mathfrak{p}_N)=D_N\hat{M}_N$ for $t\notin N^{-1}\mathbb Z_{\ge 0}$ as well.
        
        Finally, if we view the restriction of $D_N\hat{M}_N$ to $[0,T]$ as an element of $\mathcal S'(\mathbb R^2)$ supported on $[0,T]\times \mathbb R$, then we can apply $K_N$ to both sides and we obtain that $(\mathscr U_N-\mathfrak p_N)(t,\cdot) =  K_ND_N\hat{M}_N(t,\cdot)$ for all $t\in [0,T]$ by Lemma \ref{inv}.
    \end{proof}

     \subsection{Weighted H\"older spaces and Schauder estimates} 

    We now introduce various natural topologies for our field $\mathscr{U}_N$ and its limit points. We then discuss how the heat flow affects these topologies and record Kolmogorov-type lemmas that will be key in showing tightness under these topologies. We begin by recalling many familiar and useful spaces of continuous and differentiable functions that have natural metric structures. 
For $d\ge 1$, we denote by $C_c^{\infty}(\R^d)$ the space of all compactly supported smooth functions on $\R^d$. For a smooth function on $\R^d$, we define its $C^r$ norm as
\begin{align}\label{crnorm}
    \|f\|_{C^r}:=\sum_{\vec{\alpha}} \sup_{\mathbf x\in \R^d} |D^{\vec\alpha} f(\mathbf x)|
\end{align}
 where the sum is over all $\vec{\alpha}\in \mathbb{Z}_{\ge 0}^d$ with $\sum \alpha_i\le r$ and $D^{\vec{\alpha}}:=\partial_{x_1}^{\alpha_1}\cdots \partial_{x_d}^{\alpha_d}$ denotes the mixed partial derivative.

\smallskip

We now recall the definition of weighted H\"older spaces from \cite[Definitions 2.2 and 2.3]{HL16}.  For the remainder of this paper, we shall work with \textit{elliptic and parabolic} weighted H\"older spaces with polynomial weight function
 \begin{align}
 \label{defwt}
     w(x):=(1+x^2)^\sig
 \end{align}
  for some fixed $\sig>1$. We introduce these weights because weighted spaces will be more convenient to obtain tightness estimates. Since the solution of \eqref{she} started from Dirac initial condition is known to be globally bounded away from $(0,0)$, we expect that it is possible to remove the weights throughout this section, but this would require more precise moment estimates than the ones we derived in previous sections, which take into account spatial decay of the fields.

	\begin{defn}[Elliptic H\"older spaces]\label{ehs}
		For $\alpha\in(0,1)$ we define the space $C^{\alpha,\sig}(\mathbb R)$ to be the completion of $C_c^\infty(\mathbb R)$ with respect to the norm given by $$\|f\|_{C^{\alpha,\sig}(\mathbb R)}:= \sup_{x\in\mathbb R} \frac{|f(x)|}{w(x)} + \sup_{|x-y|\leq 1} \frac{|f(x)-f(y)|}{w(x)|x-y|^{\alpha}}.$$
		For $\alpha<0$ we let $r=-\lfloor \alpha\rfloor$ and we define $C^{\alpha,\sig}(\mathbb R)$ to be the closure of $C_c^\infty(\mathbb R)$ with respect to the norm given by $$\|f\|_{C^{\alpha,\sig}(\mathbb R)}:= \sup_{x\in\mathbb R} \sup_{\lambda\in (0,1]} \sup_{\phi \in B_r} \frac{(f,S^\lambda_{x}\phi)_{L^2(\mathbb R)}}{w(x)\lambda^\alpha}$$ where the scaling operators $S^\lambda_{x}$ are defined by 
  \begin{align}\label{escale}
  S^\lambda_{x}\phi (y) = \lambda^{-1}\phi(\lambda^{-1}(x-y)),\end{align} 
  and where $B_r$ is the set of all smooth functions of {$C^r$ norm} less than 1 with support contained in the unit ball of $\mathbb R$.
	\end{defn}

 \begin{defn}[Function spaces] \label{fsp} Let $C^{\alpha,\sig}(\mathbb R)$ be as in Definition \ref{ehs}. We define $C([0,T],C^{\alpha,\sig}(\mathbb R))$ to be the space of continuous maps $g:[0,T]\to C^{\alpha,\sig}(\mathbb R),$ equipped with the norm $$\|g\|_{C([0,T],C^{\alpha,\sig}(\mathbb R))} := \sup_{t\in[0,T]} \|g(t)\|_{C^{\alpha,\sig}(\mathbb R)}.$$ 
\end{defn}

	Here and henceforth we will define $\Psi_{[a,b]}:=[a,b]\times\mathbb R$ and we will define $\Psi_T:=\Psi_{[0,T]}.$

\smallskip

 So far we have used $\phi,\psi$ for test functions on $\R$. To make the distinction clear between test functions on $\R$ and $\R^2$, we shall use variant Greek letters such as $\varphi, \vartheta, \varrho$ for test functions on $\R^2$. In some instances, we will explicitly write $(f,\varphi)_{\mathbb R^2}$ or $(g,\phi)_{\mathbb R}$ when we want to be clear about the space in which we are pairing. 
 
	\begin{defn}[Parabolic H\"older spaces]\label{phs}
		We define $C_c^\infty(\Psi_T)$ to be the set of functions on $\Psi_T$ that are restrictions to $\Psi_T$ of some function in $C_c^\infty(\mathbb R^2)$ (in particular we do not impose that elements of $C_c^\infty(\Psi_T)$ vanish at the boundaries of $\Psi_T).$ 
  
    For $\alpha\in(0,1)$ we define the space $C^{\alpha,\sig}_\mathfrak s(\Psi_T)$ to be the completion of $C_c^\infty(\Psi_T)$ with respect to the norm given by $$\|f\|_{C^{\alpha,\sig}_\mathfrak s(\Psi_T)}:= \sup_{(t,x)\in \Psi_T} \frac{|f(t,x)|}{w(x)} + \sup_{|s-t|^{1/2}+|x-y|\leq 1} \frac{|f(t,x)-f(s,y)|}{w(x)(|t-s|^{1/2}+|x-y|)^{\alpha}}.$$
		For $\alpha<0$ we let $r=-\lfloor \alpha\rfloor$ and we define $C^{\alpha,\sig}_\mathfrak s(\Psi_T)$ to be the closure of $C_c^\infty(\Psi_T)$ with respect to the norm given by $$\|f\|_{C^{\alpha,\sig}_\mathfrak s(\Psi_T)}:= \sup_{(t,x)\in\Psi_T} \sup_{\lambda\in (0,1]} \sup_{\varphi \in B_r} \frac{(f,S^\lambda_{(t,x)}\varphi)_{L^2(\Psi_T)}}{w(x)\lambda^\alpha}$$ where the scaling operators are defined by 
		\begin{align}
			\label{scale}
			S^\lambda_{(t,x)}\varphi (s,y) = \lambda^{-3}\varphi(\lambda^{-2}(t-s),\lambda^{-1}(x-y)),
		\end{align}
		and where $B_r$ is the set of all smooth functions of $C^r$ norm less than 1 with support contained in the unit ball of $\mathbb R^2$.
	\end{defn}

An important property of both the parabolic and elliptic spaces is that one has a continuous embedding $C^{\alpha,\tau} \hookrightarrow C^{\beta,\tau}$ whenever $\beta<\alpha.$ In fact this embedding is compact, though we will not use this. We also have the following embedding of function spaces inside parabolic spaces.

 \begin{lem}\label{embed}
     For $\alpha<0,\tau>0$ one has a continuous embedding $C([0,T],C^{\alpha,\sig}(\mathbb R))\hookrightarrow C^{\alpha,\sig}_\mathfrak s(\Psi_T)$ given by identifying $v=(v(t))_{t\in [0,T]}$ with the tempered space-time distribution given by $$(v,\varphi)_{\mathbb R^2} = \int_0^T (v(t) ,\varphi(t,\cdot))_\mathbb Rdt.$$ 
 \end{lem}

 The proof is straightforward from the definitions and is omitted.
 
 \begin{rk}[Derivatives of distributions] \label{d/dx}  Let $\alpha<0$. By definition, any element $f\in C_\mathfrak s^{\alpha,\sig}(\Psi_T)$ admits an $L^2$-pairing with any smooth function $\varphi: \Psi_T\to \mathbb R$ of rapid decay, which we can write as $(f,\varphi)_{\Psi_T}.$ Consequently there is a natural embedding $C_\mathfrak s^{\alpha,\sig}(\Psi_T)\hookrightarrow \mathcal S'(\mathbb R^2)$ which is defined by formally setting $f$ to be zero outside of $[0,T]\times \mathbb R$. More rigorously, this means that the $L^2(\mathbb R^2)$-pairing of $f$ with any $\varphi \in \mathcal S(\mathbb R^2)$ is defined to be equal to $(f,\varphi|_{\Psi_T})_{\Psi_T}$.

 \smallskip
 
 The image of this embedding consists of some specific collection of tempered distributions that are necessarily supported on $[0,T]\times \mathbb R.$ Consequently we can sensibly define $\partial_tf$ and $\partial_xf$ as elements of $\mathcal S'(\mathbb R^2)$ whenever $f\in C_\mathfrak s^{\alpha,\sig}(\Psi_T)$. Explicitly these derivatives are defined by the formulas$$(\partial_tf,\varphi)_{\Psi_T} := -(f,\partial_t\varphi)_{\Psi_T},\;\;\;\;\;\;\;\;\;\;(\partial_xf,\varphi)_{\Psi_T} := -(f,\partial_x\varphi)_{\Psi_T}$$for all smooth $\varphi:\Psi_T\to\mathbb R$ of rapid decay. This convention on derivatives will be useful for certain computations later. From the definitions, when $\alpha<0$ one can check that for such $f$ one necessarily has $\partial_tf \in C_\mathfrak s^{\alpha-2,\sig}(\Psi_T)$ and $\partial_x f \in C_\mathfrak s^{\alpha-1,\sig}(\Psi_T).$

\smallskip

The latter statement fails for $\alpha>0$. Indeed, by our convention of derivatives, $\partial_tf$ may no longer be a smooth function (or even a function) even if $f\in C_c^\infty(\Psi_T)$. This is because such an $f$ gets extended to all of $\mathbb R^2$ by setting it to be zero outside $\Psi_T$. In particular, if $f$ does not vanish on the boundary of $\Psi_T$, then it may become a discontinuous function under our convention of extension to $\mathbb R^2$. Due to these discontinuities, the distributional derivative $\partial_tf$ can be a tempered distribution with singular parts along the boundaries (one may verify that $\partial_t f$ can be at best an element of $C^{-2,\sig}_\mathfrak s(\Psi_T)$ for generic $f\in C_c^\infty(\Psi_T)$). In our later computations, we never take derivatives of functions in $C_\mathfrak s^{\alpha,\sig}(\Psi_T)$ with $\alpha>0$. 
\end{rk}

 We now discuss the smoothing effect of heat flow on these weighted H\"older spaces.
 
	\begin{prop}[Smoothing effect on elliptic spaces]\label{p65}  For $f\in C_c^\infty(\mathbb R)$ and $t>0$ define $$\mathpzc P_tf(x):= \int_\mathbb R p(t,x-y)f(y)dy.$$ Then for all $\alpha \le \beta <1$, there exists $\Con=\Con(\alpha,\beta,T)>0$ such that $$\|\mathpzc P_tf\|_{C^{\beta,\sig}(\mathbb R)} \leq \Con t^{-(\beta-\alpha)/2}\|f\|_{C^{\alpha,\sig}(\mathbb R)} $$ uniformly over $f\in C_c^\infty$ and $t\in [0,T]$. In particular, $\mathpzc P_t$ extends to a globally defined linear operator on $C^{\alpha,\sig}(\mathbb R)$ which maps boundedly into $C^{\beta,\sig}(\mathbb R).$
	\end{prop}
	
	A proof may be found in \cite[Lemma 2.8]{HL16} in the case of an exponential weight. {The proof for polynomial weights is identical}.

\smallskip

 For the parabolic H\"older spaces, the following lemma states that the heat flow improves the regularity by a factor of $2$ and provides a Schauder-type estimate.
	\begin{prop}[Schauder estimate]\label{sch}
		For $f\in C_c^\infty(\Psi_T)$ let us define 
		\begin{align}
  \label{e:kf}
			Kf(t,x):= \int_{\Psi_T} p_{}(t-s,x-y)f(s,y)dsdy,
		\end{align} where $p$ is the standard heat kernel for $t>0,x\in\mathbb R$ and $p(t,x):=0$ for $t<0.$ Then for all $\alpha<-1$ with $\alpha\notin\mathbb Z$ there exists $\Con=\Con(\alpha)>0$ independent of $f$ such that $$\|Kf\|_{C^{\alpha+2,\sig}_\mathfrak s(\Psi_T)}\leq \Con\cdot\|f\|_{C^{\alpha,\sig}_\mathfrak s(\Psi_T)}.$$ In particular $K$ extends to a globally defined linear operator on $C^{\alpha,\sig}_\mathfrak s(\Psi_T)$ that maps boundedly into $C^{\alpha+2,\sig}_\mathfrak s(\Psi_T).$ Furthermore, if $f\in C^{\alpha,\sig}_\mathfrak s(\Psi_{T})$, then $K(\partial_t-\frac12\partial_x^2)f= f$.
	\end{prop}

We remark that the last statement is only true because of our convention on distributional derivatives that we have explained in Remark \ref{d/dx}. Without that convention, that statement would be false even for a smooth function $f$ that does not vanish on the line $\{t=0\}$.

	\begin{proof} See \cite[Proposition 6.6]{DDP23}.
  \end{proof}

\begin{cor}\label{sme} Define $J: C^{\alpha,\sig}(\mathbb R) \to C_\mathfrak s^{\alpha,\sig}(\Psi_T)$ by $$Jf(t,x):= \big(f,p(t,x-\cdot)\big)_\mathbb R,$$  $J$ is a bounded linear operator for any $\alpha <1$ and any $\sig>0.$  Consider the operator $\hat{\mathpzc P}_t:= J \mathpzc P_t$ with $\mathpzc P_t$ defined in Proposition \ref{p65}. By the semigroup property of the heat kernel, we have that 
 \begin{align}\label{e.hatjrel}
     \hat{\mathpzc P}_{s}f(t,x)=(Jf)(t+s,x).
 \end{align}
 Furthermore, the operators satisfy 
 \begin{align}
     \label{e.schhat}
     \|\hat{\mathpzc P}_tf\|_{C^{\beta,\sig}_\mathfrak s(\Psi_T)} \leq \Con \cdot t^{-(\beta-\alpha)/2}\|f\|_{C^{\alpha,\sig}(\mathbb R)},
 \end{align}
where $\Con=\Con(\alpha,\beta,T)>0$ is independent of $t$ and $f$.
	\end{cor}
 \begin{proof} The first part follows by using the Schauder estimate in Proposition \ref{sch}, noting that $Jf = K(\delta_0\otimes f)$ where for $f \in C^{\alpha,\sig}(\mathbb R)$ the latter distribution is defined by $(\delta_0\otimes f, \varphi)_{\mathbb R^2}:= (f,\varphi(0,\cdot))_{\mathbb R}.$ Directly from the definitions one can check that $f\mapsto \delta_0\otimes f$ is bounded from $C^{\alpha,\sig}(\mathbb R)\to C_\mathfrak s^{\alpha-2,\sig}(\Psi_T).$ Finally, \eqref{e.schhat} follows from Proposition \ref{p65}.     
 \end{proof}

We next define a space-time distribution that is supported on a single temporal cross-section:

\begin{defn}\label{otimes} Let $\alpha<0$. Given some $f\in C^{\alpha,\sig}(\mathbb R)$ and $b \in [0,T]$ we define $\delta_b \otimes f\in C_\mathfrak s^{\alpha-2,\sig}(\Psi_T)$ by the formula
$(\delta_b\otimes f, \varphi)_{\mathbb R^2}:= (f,\varphi(b,\cdot))_{\mathbb R}.$
 \end{defn}

 Directly from the definitions of the scaling operators in \eqref{escale} and \eqref{scale}, it is clear that for fixed $b\in [0,T]$, the linear map $f\mapsto \delta_b\otimes f$ is bounded from $C^{\alpha,\sig}(\mathbb R)\to C_\mathfrak s^{\alpha-2,\sig}(\Psi_T)$ as long as $\alpha<0.$ It is also clear that $\delta_b \otimes f$ is necessarily supported on the line $\{b\}\times \mathbb R$. The following lemma shows that under mild conditions, $\delta_T\otimes f$ vanishes upon the action of the heat flow $K$ defined in \eqref{e:kf}.

 We end this subsection by recording a Kolmogorov-type lemma for the function spaces and parabolic H\"older spaces introduced at the beginning of this subsection. It will be crucial in proving tightness in those respective spaces.
	
	\begin{lem}[Kolmogorov lemma] \label{l:KC} Let $L^2(\Omega,\mathcal{F},\Pr)$ be the space of all random variables defined on a probability space  $(\Omega,\mathcal{F},\Pr)$ with finite second moment. We have  the following:
 \begin{enumerate}[label=(\alph*),leftmargin=15pt]
 \setlength\itemsep{0.5em}


 \item \label{KC2} (Function space) 
		Let $(t,\phi) \mapsto V(t,\phi)$ be a map from $[0,T]\times \mathcal S(\mathbb R)$ into $L^2(\Omega,\mathcal F,\mathbb P)$ which is linear and continuous in $\phi$. Fix a non-negative integer $r$. Assume there exists some $\kappa>0, p>1/\kappa$ and $\alpha<-r$ and $\Con=\Con(\kappa,\alpha,p,{T})>0$ such that one has 
		\begin{align*}\mE[|V(t,S^\lambda_{x}\phi)|^p]^{1/p}&\leq \Con\lambda^{\alpha },\\  \mE[|V(t,S^\lambda_{x}\phi)-V(s,S^\lambda_{x}\phi)|^p]^{1/p}&\leq \Con\lambda^{\alpha -\kappa} |t-s|^{\kappa },
		\end{align*}uniformly over all smooth functions $\phi$ on $\mathbb R$ supported on the unit ball of $\mathbb R$ with {$\|\phi\|_{C^r}\leq 1$}, and uniformly over $\lambda\in(0,1]$ and $0\leq s,t\leq T$. Then for any $\sig>1$ and any $\beta<\alpha-\kappa$ there exists a random variable $\big(\mathscr V(t)\big)_{t\in[0,T]}$ taking values in $C([0,T],C^{\beta,\sig}(\mathbb R))$ such that $(\mathscr V(t),\phi)=V(t,\phi)$ almost surely for all $\phi$ and $t$. Furthermore, one has that $$\mE[\|\mathscr V\|^p_{C([0,T],C^{\beta,\sig}(\mathbb R))}]\leq \Con',$$ where $\Con'$ depends on the choice of $\alpha,\beta,p,\kappa,$ and the constant $\Con$ appearing in the moment bound above but not on $V,\Omega,\mathcal F,\mathbb P$.
 
     \item \label{KC}
		(Parabolic H\"older Space) Let $\varphi \mapsto V(\varphi)$ be a bounded linear map from  $\mathcal S(\mathbb R^2)$ to $L^2(\Omega,\mathcal F,\mathbb P)$. Assume $V(\varphi)=0$ for all $\varphi$ with support contained in the complement of $\Psi_T$. Recall $S^\lambda_{(t,x)}$ from  \eqref{scale}. Fix a non-negative integer $r$. Assume there exists some $p>1$ and $\alpha <0$ and $\Con=\Con(\alpha,p)>0$ such that one has $$\mE[|V(S^\lambda_{(t,x)}\varphi)|^p]^{1/p}\leq \Con\lambda^{\alpha },$$ uniformly over all smooth functions $\varphi$ on $\mathbb R^2$ supported on the unit ball of $\mathbb R^2$ with {$\|\varphi\|_{C^r}\leq 1$}, and uniformly over $\lambda\in(0,1]$ and $(t,x)\in\Psi_T$. Then for any $\sig>1$ and any $\beta<\alpha-3/p$ there exists a random variable $\mathscr V$ taking values in $C^{\beta,\sig}_\mathfrak s(\Psi_T)$ such that $(\mathscr V,\varphi)=V(\varphi)$ almost surely for all $\varphi.$ Furthermore one has that $$\mE[\|\mathscr V\|^p_{C^{\beta,\sig}_\mathfrak s(\Psi_T)}]\leq \Con',$$ where $\Con'$ depends on the choice of $\alpha,p,$ and the constant $\Con$ appearing in the moment bound above but not on $V,\Omega,\mathcal F,\mathbb P$.

 \end{enumerate}

	\end{lem}
	{A proof of the above results may be adapted from the proof of Lemma 9 in Section 5 of \cite{WM}.} We remark that we do not actually need uniformity over a large class of test functions as we have written above, just a single well-chosen test function would suffice (e.g. the Littlewood-Paley blocks as used in \cite{WM} or the Daubechies wavelets in \cite{HL16}). 
 \subsection{Tightness} Here we will prove tightness of the rescaled field from \eqref{field}.

 \begin{lemma}\label{kndn}
     Fix $\alpha<0,\sig>0$. Let $K_N,D_N$ be the operators from Definition \ref{dnlnkn}. Recall the function spaces from Definition \ref{fsp}, and let $\mathcal X_N^{\alpha,\tau}$ denote the closed linear subspace of $C([0,T+1],C^{\alpha,\sig}(\mathbb R))$ consisting of those paths $v=(v(t))_{t\in [0,T+1]}$ such that $v(t)=0$ for all $t\in [0,N^{-1}]$. Then we have the operator norm bound $$\sup_{N\ge 1}\|K_ND_N\|_{\mathcal X_N^{\alpha,\tau}\to C([0,T],C^{\alpha,\sig}(\mathbb R))}<\infty.$$
 \end{lemma}


 \begin{proof}
     Note that $v(t)=0$ for $t\in [0,N^{-1}]$ ensures that $D_Nv$ is actually a continuous path taking values in $C([0,T+1-N^{-1}], C^{\alpha,\sig}(\R))$, consequently so is $K_ND_Nv$. Let us now define four families of linear operators on $\mathcal S'(\mathbb R)$, indexed by $N$. For $f\in \mathcal S'(\mathbb R)$ and $s\in N^{-1}\mathbb Z_{\ge 0}$ let
     \begin{align*}P_N(s)f(x) &:= \sum_{y\in N^{-1/2}\mathbb Z-N^{3/4}s} p_N(Ns,N^{1/2}y)f(x-y), \\ \delta_Nf(x) &:= N(P_N(N^{-1})-I)f (x) = N\big[\rho_N f(x+N^{-\frac12} -N^{-\frac34}) 
     + (1-\rho_N) f(x-N^{-\frac12} - N^{-\frac34}) -f(x)\big] \\ \delta_N^* f(x)  &:= N\big[\rho_N f(x-N^{-1/2} +N^{-3/4}) + (1-\rho_N) f(x+N^{-1/2} + N^{-3/4}) -f(x)\big], \\ P^*_N(s) f(x)& := \big(I+N^{-1}\delta_N^*\big)^{Ns}f(x)=\sum_{y\in N^{-1/2}\mathbb Z-N^{3/4}s} p_N(Ns,N^{1/2}y)f(x+y).
     \end{align*}
     Then for $f\in \mathcal S'(\mathbb R)$ each of the four expressions $P_N(s)f, \delta_N f, \delta_N^*f, P_N^*(s)f$ also make sense as elements of $\mathcal S'(\mathbb R)$. Note that $\delta_N$ and $\delta^*_N$, both of which approximate the spatial Laplacian $\partial_x^2$, are adjoint to each other on $L^2(\mathbb R)$. Therefore $P_N$ and $P^*_N$ are also adjoint to each other. 
     Now let $v = (v(t))_{t\in [0,T+1]} \in \mathcal X_N^{\alpha,\sig}$. For $\phi\in C_c^\infty(\mathbb R)$ and $t\in [0,T]$, if we apply summation by parts and use the fact that $v$ vanishes on $[0,N^{-1}]$ we can write 
     \begin{align*}\big( K_ND_N v(t),\phi \big) &= \bigg( N^{-1} \sum_{s\in [0,t]\cap( N^{-1}\mathbb Z_{\ge 0})} NP_N(t-s)\big[ v(s+N^{-1})-v(s)\big]\;\;, \;\; \phi \bigg) \\ &= \bigg( v(t+N^{-1}) - 0 + \sum_{s\in [0,t]\cap( N^{-1}\mathbb Z_{\ge 0})} \big[ P_N(t-s)-P_N(t-s-N^{-1}) \big] (v(s))\;\;, \;\; \phi \bigg) \\ &= (v(t+N^{-1}),\phi) + \frac1N \sum_{s\in [0,t]\cap( N^{-1}\mathbb Z_{\ge 0})} (\delta_N P_N(t-s-N^{-1}) v(s) , \phi) \\ &= (v(t+N^{-1}),\phi) + \frac1N \sum_{s\in [0,t]\cap( N^{-1}\mathbb Z_{\ge 0})} (v(s), P_N^*(t-s-N^{-1})\delta^*_N \phi).
     \end{align*}

Here we are using the $L^2(\mathbb R)$-pairing between distributions and smooth functions. By the definition of the function spaces (Definition \ref{fsp}), we must replace $\phi$ by $S^\lambda_x\phi$ (where the scaling operators are given in Definition \ref{ehs}) in the last expression and then study the growth as $\lambda$ becomes close to 0. 

The first term on the right side is completely straightforward to deal with: the growth is at worst $\lambda^\alpha (1+x^2)^\tau$ uniformly over $\lambda, \phi, x, T$, since we assumed $v\in C([0,T],C^{\alpha,\sig}(\mathbb R)).$ To deal with the second term, we claim that one has 
\begin{equation}\label{bd5}|(v(s),P_N^*(t-s-N^{-1})\delta_N^* S^\lambda_x\phi) |\lesssim \big(\lambda^{\alpha-2} \wedge (t-s)^{\frac{\alpha}2-1}\big)(1+x^2)^\sig\end{equation} uniformly over $s<t\in [0,T]$, as well as $\lambda,\phi,x, N$ and $v$ with $\|v\|_{C([0,T+1],C^{\alpha,\sig}(\mathbb R))}\leq 1$. Indeed the bound of the form $\lambda^{\alpha-2}$ follows by noting that when $t-s$ is much smaller than $\lambda$, $P_N^*(t-s)$ is essentially the identity operator, so we can effectively disregard the heat kernel and note that $\delta_N^* S^\lambda_x\phi$ paired with $v(s)$ satisfies a bound of order $\lambda^{\alpha-2}$, uniformly over $N$ by e.g. second-order Taylor expansion. Likewise the bound of the form $(t-s)^{\frac{\alpha}2-1}$ is obtained by noting that when $\lambda$ is very small compared to $t-s$, $P_N^*(t-s-N^{-1})\delta^*_N S^\lambda_x\phi$ behaves like $S_x^{\sqrt{t-s}}\phi$, giving a bound of order $(\sqrt{t-s})^{\alpha-2}$ after applying $\delta^*_N$ and pairing with $v(s)$. This proves the bound \eqref{bd5}.

Now the fact that $\|K_ND_Nv\|$ can be controlled by $\|v\|$ follows simply by noting that uniformly over $0<\lambda^2 \leq t\leq T$ one has $$\int_0^t \big(\lambda^{\alpha-2} \wedge (t-s)^{\frac{\alpha}2-1}\big) ds = \int_0^{t-\lambda^2}(t-s)^{\frac{\alpha}2-1}ds + \int_{t-\lambda^2}^t \lambda^{\alpha-2}ds \leq \Con(\alpha)\cdot  \lambda^\alpha. $$
This implies the uniform bound on the operator norm of $K_ND_N$.
 \end{proof}

 \begin{prop}[Tightness of all relevant processes]\label{mcts} The following are true.
 \begin{enumerate}
		\item The fields $\hat{M}_N$ from Definition \ref{mqv} may be realized as an element of $C([0,T],C^{\alpha,\sig}(\mathbb R))$ for any $\alpha<-3$ and $\sig>1$. Moreover, they are tight with respect to that topology.  \item The fields $\mathscr U_N$ from \eqref{field} may be realized as an element of as an element of $C([0,T],C^{\alpha,\sig}(\mathbb R))$ for any $\alpha<-3$ and $\sig>1$. Moreover, they are tight with respect to that topology.
        \item The fields $\hat Q_N$ from Definition \ref{mqv} may be realized as an element of $C([0,T],C^{\gamma,\sig}(\mathbb R))$ for any $\gamma<-1$ and $\sig>1$. Moreover, they are tight with respect to that topology. 
        
        \item Let $\alpha < -3, \gamma<-1,$ and $\tau > 1$. Let $(M^\infty,Q^\infty, U^\infty)$ be a joint limit point of $(\hat M_N,\hat Q_N,\mathscr U_N)$ in $C([0,T],C^{\alpha,\sig}(\mathbb R)\times C^{\gamma,\sig}(\mathbb R)\times C^{\alpha,\sig}(\mathbb R)).$ For all $\phi\in C_c^\infty(\mathbb R)$ the process $(M_t^\infty(\phi))_{t\in[0,T]}$ is a continuous martingale with respect to the canonical filtration on that space, and moreover 
  \begin{equation}
      \label{e:mcts}
      \langle M^\infty(\phi)\rangle_t = Q_t^\infty(\phi^2).
  \end{equation}
  \end{enumerate}
	\end{prop}
 \begin{proof}
     Take any $\phi\in C_c^\infty(\mathbb R)$ with $\|\phi\|_{L^\infty}\leq 1.$ Recall $S_{(t,x)}^{\lambda}$ from \eqref{scale}. Using the first bound in Proposition \ref{tight1}, we have $$\Ex[ (Q_N(t,S^\lambda_{x}\phi)-Q_N(s,S^\lambda_{x}\phi))^{k}]\leq \Con |t-s|^{k/2}\lambda^{-k},$$ uniformly over $x\in\mathbb R, \lambda\in (0,1], 0\leq s,t\leq T$ with $s,t\in N^{-1}\mathbb Z_{\ge 0}$. Since $Q_N(0,\phi)=0$ by definition, the assumptions of Lemma \ref{l:KC} \ref{KC2} are therefore satisfied for any $\kappa\leq 1/4$, any $p>1/\kappa$, and any $\alpha\leq -1$, and we conclude the desired tightness for $\hat Q_N=Q_N$. This proves Item \textit{(3)}.

  \smallskip
  
		Now we address the tightness of the $\hat M_N$. Using the Burkholder-Davis-Gundy inequality and then Proposition \ref{optbound}, we have that \begin{equation}\label{m1}\Ex[ (M_N(t,\phi)-M_N(s,\phi))^{8}]^{1/8} \leq \Con \cdot \Ex\big[ \big([M_N(\phi)]_t - [M_N(\phi)]_s)\big)^{4}\big]^{1/8}\leq \Con(t-s)^{1/4} \|\phi\|_{C^1(\mathbb R)},\end{equation} where $\phi\in C_c^\infty(\mathbb R)$, $\Con=\Con(k)>0$ is free of $\phi,s,t,N.$
		This gives $$\Ex[ (M_N(t,S^\lambda_{x}\phi)-M_N(s,S^\lambda_{x}\phi))^{8}]^{1/8}\leq \Con(t-s)^{1/4} \lambda^{-2} $$
		uniformly over $x\in\mathbb R, \lambda\in (0,1], 0\leq s,t\leq T$, and $\phi\in C_c^\infty(\mathbb R)$ with $\|\phi\|_{C^1}\leq 1$ with support contained in the unit interval. Moreover $\hat M_N(0,\phi)=0$ by definition, therefore the assumptions of Lemma \ref{l:KC} \ref{KC2} are satisfied with $\kappa=1/4$, $p=8>1/\kappa$, and any $\alpha\leq -2$. Hence, we conclude the desired tightness for $\hat M_N$. This proves Item \textit{(1)}.

  Now tightness for the fields $\mathscr U_N$ is immediate from Lemma \ref{kndn}, since we know from Lemma \ref{u=kdm} that $\mathscr U_N = \mathfrak p_N + K_ND_N\hat M_N$ where we view $K_ND_N $ as a bounded operator from $\mathcal X_N^{\alpha,\tau}\to C([0,T],C^{\alpha,\sig}(\mathbb R))$ and the convergence of $\mathfrak p_N$ in this topology is straightforward to deal with. This proves Item \textit{(2).}

  \smallskip
  
	We next show that the limit point $M^{\infty}(\phi)$ is a martingale indexed by $t\in N^{-1}\mathbb Z_{\ge 0}$. Since $M_N(0,\phi)=0$, from \eqref{m1}, we see that $\sup_N \Ex[M_N(t,\phi)^{2k}]<\infty.$ Thus $M^{\infty}(\phi)$ is a martingale since martingality is preserved by limit points under the uniform integrability assumption. Continuity is guaranteed by the definition of the spaces in which we proved tightness. In the prelimit we know from \eqref{m=q} that $$M_N(t,\phi)^2-\big((2\rho_N-1)^2\sqrt{N}\big)Q_N(t,\phi^2) - \mathcal E_N(t,\phi)$$ is a martingale indexed by $t\in N^{-1}\mathbb Z_{\ge 0}$, where the error term $\mathcal E_N$ is defined in \eqref{e_n} and satisfies the bound \eqref{ebound}. Note that $\big((2\rho_N-1)^2\sqrt{N}\big)\to 1$ as $N\to \infty$. By \eqref{ebound} and tightness estimates  \eqref{e.tight1} of $Q_N$, it follows that $\mathcal E_N(t,\phi)$ vanishes in probability in the topology of $C[0,T]$, so we conclude (again by uniform $L^p$ boundedness guaranteed by Proposition \ref{tight1}) that $M_t^\infty(\phi)^2-Q_t^\infty(\phi^2)$
	is a martingale. This verifies \eqref{e:mcts} completing the proof of Item \textit{(4)}.
 \end{proof}

 \begin{lem}[Controlling the difference between the discrete and continuum heat operators] \label{kn-k}
     Fix $\alpha<0, \tau>1$. Let $K_N$ be as in Definition \ref{dnlnkn}, and let $K$ be as in \eqref{e:kf}. Then we have the operator norm bound $$\|K_N-K\|_{C^{\alpha,\sig}_\mathfrak s(\Psi_T) \to C^{\alpha-1,\sig}_\mathfrak s(\Psi_T)} \leq CN^{-1/4}.$$
  Here $C$ is independent of $N$.
 \end{lem}

 The above bound is crude, we do not claim optimality of the Holder exponents here.

 \begin{proof}
     Define $f_\varphi^\lambda(t,x):= (f,S^\lambda_{(t,x)}\varphi)_{L^2(\Psi_T)}$. It suffices to show that 
     $$\sup_{\|f\|_{C^{\alpha,\sig}_\mathfrak s(\Psi_T)}\leq 1} \sup_{(t,x)\in\Psi_T}\sup_{\lambda\in [(0,1]} \sup_{\varphi \in B_r} 
     \frac{|(K_N-K)f_\varphi^\lambda(t,x)|}{(1+x^2)^{\sig}\lambda^{\alpha-1}}\leq CN^{-1/4},$$ 
     where the scaling operators are defined by 
		$S^\lambda_{(t,x)}\varphi (s,y) = \lambda^{-3}\varphi(\lambda^{-2}(t-s),\lambda^{-1}(x-y)) ,$
		and where $B_r$ is the set of all smooth functions of $C^r$ norm less than 1 with support contained in the unit ball of $\mathbb R^2$.
     We shall use a probabilistic interpretation of the kernels to prove this. Note that $$(K_N-K)f_\varphi^\lambda (t,x) = \frac1N \hspace{-0.2cm}\sum_{s\in [0,t]\cap (N^{-1}\mathbb Z_{\ge 0})}\hspace{-0.2cm} \mathbf E[f_\varphi^\lambda(t-s,x- W_N(s))] - \int_0^t \mathbf E[f_\varphi^\lambda (t-s,x-W(s))]ds, $$ where $W_N(s) = N^{-1/2}R_N(Ns)$ for the discrete-time random walk $(R_N(r))_{r\ge 0}$ with increment distribution described in Definition \ref{dnlnkn}, and $W$ is a standard Brownian motion. Define $W_N(s)$ by linear interpolation for $s\notin N^{-1}\mathbb Z_{\ge 0}$. 
     
     By the KMT coupling \cite{KMT}, we may assume that $W_N$ and $W$ are all coupled onto the same probability space so that $\sup_{s\in [0,T]} |W_N(s) - W(s)| \leq G N^{-1/4}$ where $G$ is a random variable on that same probability space independent of $N$ such that $\mathbf E|G|^p<\infty$ for all $p$. Note that $-1/4$ is the optimal exponent here because, despite KMT giving a better exponent in principle, the increments of the original random walk $R_N$ are not centered but rather they have a non-zero mean of order $N^{-3/4}$. Consequently, by the definition of the parabolic spaces, we have uniformly over $s,t \in [0,T]$ the bound
     \begin{align*} & |f_\varphi^\lambda(t-s,x- W_N(s))-f_\varphi^\lambda (t-s,x-W(s))| \\ & \leq \bigg(
       \sup_{u \in \big[-|x| - |W(s)|-|W_N(s)|, \;|x| + |W(s)|+|W_N(s)|\big]} | \partial_x f_\varphi^{\lambda}(t-s,u)|\bigg)\cdot | W_N(s) - W(s)| \\ &\leq \bigg(
       \|f\|_{C^{\alpha,\sig}_\mathfrak s(\Psi_T)} \lambda^{\alpha-1} (1+4x^2 + 4W(s)^2 + 4(W_N(s))^2)^\sig \bigg)\cdot GN^{-1/4}.
     \end{align*}
      Here the factor $\|f\|_{C^{\alpha,\sig}_\mathfrak s(\Psi_T)} \lambda^{\alpha-1}$ can be deduced using e.g. Remark \ref{d/dx} which says that $\partial_x$ boundedly reduces the parabolic regularity by 1 exponent. We also used the fact that $(|x|+|y|+|z|)^2 \leq 4x^2+4y^2+4z^2.$
      
      We claim that $$\mathbf E \big|G\cdot (1+4x^2 + 4W(s)^2 + 4(W_N(s))^2)^\sig \big|< C(1+x^2)^{\tau} $$ for a constant $C=C(\alpha,T,\tau)$ independent of $N,x,s.$ To prove this, one uses $(a+b+c+d)^\tau \leq 4^\tau(a^\tau +b^\tau+c^\tau+d^\tau)$, then one notes that $\sup_{N} \mathbf E|G \cdot \sup_{s\leq T} (|W(s)|^{2\tau} +|W_N(s)|^{2\tau})|<\infty$ by e.g. Cauchy-Schwartz. 
      
      Consequently the above expression for $(K_N-K)f_\varphi^\lambda (t,x)$ can be bounded in absolute value by a universal constant times $\|f\|_{C^{\alpha,\sig}_\mathfrak s(\Psi_T)} \lambda^{\alpha-1} N^{-1/4}$. 
 \end{proof}

 \begin{lem}[Controlling the difference between the discrete and continuum time-derivatives] \label{ds} Fix $\alpha<0, \tau>1$. The derivative operator $\partial_s : C^{\alpha,\sig}_\mathfrak s(\Psi_T) \to C_\mathfrak s^{\alpha-2,\sig}(\Psi_T)$ which was defined in Remark \ref{d/dx}, is a bounded linear map. Furthermore, let $D_N$ be as in Definition \ref{dnlnkn}. Then we have the operator norm bounds 
 \begin{align*}\sup_{N\ge 1}\|D_N\|_{C^{\alpha,\sig}_\mathfrak s(\Psi_T)\to C_\mathfrak s^{\alpha-2,\sig}(\Psi_T)}&<\infty. \\ \|D_N-\partial_s\|_{C^{\alpha,\sig}_\mathfrak s(\Psi_T) \to C^{\alpha-4,\sig}_\mathfrak s(\Psi_T)} &\leq \frac12 N^{-1}.
 \end{align*}
\end{lem}

 \begin{proof}The proof that $\partial_s$ is bounded is straightforward from the definitions of the spaces. Note that $(D_Nf,\varphi) = (f,D_N^-\varphi)$ where $D_N^-\varphi(t,x) = N\big[\varphi(t,x)-\varphi(t-N^{-1},x)\big]$. Recall $S^\lambda_{(t,x)}$ from \eqref{scale}. From the definition of the space $C^{\alpha,\sig}(\mathbb R)$ one verifies directly that if $v\in C([0,T],C^{\alpha,\sig}(\mathbb R))$ with norm less than or equal to 1, then for $N>\lambda^{-1}$ one has $$\int_0^T |v_t\big(D_N^- (S^\lambda_{(s,y)} \varphi) (t,\cdot)\big)|dt \leq (1+y^2)^\tau \int_{s-\lambda^2}^{s+\lambda^2} \lambda^{\alpha-4}dt \leq (1+y^2)^\tau\lambda^{\alpha-2}$$ uniformly over $(s,y)\in \Psi_T$ and $\varphi\in C_c^\infty(\mathbb R^2)$ supported on the unit ball of $\mathbb R^2$ with $\|\varphi\|_{C^r}\leq 1$ where $r=\lceil -\alpha \rceil.$ Here the bound of order $\lambda^{\alpha-4}$ may be deduced by writing $D_N^-\varphi(t,x) = N\int_{t-N^{-1}}^t \partial_s\varphi(s,y)ds$, interchanging the two integrals, and noting that $\partial_s (S^\lambda_{(s,y)} \varphi) (t,\cdot)$ satisfy such a bound. This proves the first operator norm bound.

 Now we prove the second bound. Define $f_\varphi^\lambda(t,x):= (f,S^\lambda_{(t,x)}\varphi)_{L^2(\Psi_T)}$. It suffices to show that 
     $$\sup_{\|f\|_{C^{\alpha,\sig}_\mathfrak s(\Psi_T)}\leq 1} \sup_{(t,x)\in\Psi_T}\sup_{\lambda\in (0,1]} \sup_{\varphi \in B_r} 
     \frac{|(D_N-\partial_s)f_\varphi^\lambda(t,x)|}{(1+x^2)^{\sig}\lambda^{\alpha-4}}\leq \frac12 N^{-1},$$ 
     where the scaling operators are defined by 
		$S^\lambda_{(t,x)}\varphi (s,y) = \lambda^{-3}\varphi(\lambda^{-2}(t-s),\lambda^{-1}(x-y)) ,$
		and where $B_r$ is the set of all smooth functions of $C^r$ norm less than 1 with support contained in the unit ball of $\mathbb R^2$. Note that
  \begin{align*}
      |(D_N-\partial_s)f_\varphi^\lambda(t,x)| &= \bigg| N\big( f^\lambda_\varphi (t+N^{-1},x)- f^\lambda_\varphi(t,x)\big) - \partial_s  f^\lambda_\varphi(t,x)\bigg| \\ &=  \bigg| N\int_t^{t+N^{-1}} \big( \partial_s  f^\lambda_\varphi(u,x) - \partial_s f^\lambda_\varphi(t,x)\big) du\bigg| \\ &\leq N\int_t^{t+N^{-1}} \|\partial_s^2  f^\lambda_\varphi\|_{L^\infty([t,t+N^{-1}]\times \{x\})} |t-u|du \leq \frac12 N^{-1}\|\partial_s^2  f^\lambda_\varphi\|_{L^\infty([t,t+N^{-1}]\times \{x\})} .
  \end{align*}
  From the definitions one has $\|\partial_s^2  f^\lambda_\varphi\|_{L^\infty([t,t+N^{-1}]\times \{x\})} \leq (1+x^2)^\tau \lambda^{\alpha-4} \|f\|_{C^{\alpha,\sig}_\mathfrak s(\Psi_T)}.$
 \end{proof}

 Since we expect to obtain a Dirac initial data in the limiting SPDE, there is an additional singularity at the origin that we have not yet taken into account. To fix this issue, we now formulate a tightness result taking into account this singularity, by starting the field \eqref{field} from some positive time $\e$ (which should be thought of as being close to $0$).

 \begin{prop}[Regularity of limit points] \label{reg}
Fix $\e>0,$ and consider the fields $\mathscr U_{N,\e}(t,\cdot):=\mathscr U_N(t+\e,\cdot),$ for $t\ge 0$. Set $\mathscr U_{N,\e}(t,\cdot):=0$ for $t<0$. The fields $\mathscr U_{N,\e}$ may be realized as random variables taking values in $C([0,T], C^{\alpha,\sig}(\mathbb R))$ for any $\sig>1$ and $\alpha<-3$. Furthermore, they are tight with respect to that topology, and any limit point is necessarily supported on $C_\mathfrak s^{-\kappa+1/2,\sig}(\Psi_T)$ for any $\kappa>0$.
 \end{prop}

 \begin{proof}
     We already know from Item (2) of Proposition \ref{mcts} that the $\mathscr U_{N}$ are tight in $C([0,T],C^{\alpha,\sig}(\mathbb R))$. We also know from Item (1) of Proposition \ref{mcts} that $\hat M_N$ are tight in $C([0,T+1],C^{\alpha,\sig}(\mathbb R))$, which is embeds continuously into $C^{\alpha,\sig}_\mathfrak s(\Psi_T)$ by Lemma \ref{embed}. Thus using the first bound in Lemma \ref{ds}, we have that $D_N\hat M_N$ are tight in $C_\mathfrak s^{\alpha-2}(\Psi_T)$.
     
     Consider any joint limit point $(U^\infty,\mathscr M^\infty)$ as $N\to\infty$ of the pair $(\mathscr U_N,D_N\hat M_N)$ in the product space $C([0,T],C^{\alpha,\sig}(\mathbb R))\times C^{\alpha-2,\sig}_\mathfrak s(\Psi_T)$. By Lemmas \ref{u=kdm} and \ref{kn-k} we may conclude that 
     \begin{equation}\label{u=p+km}U^\infty = p +K\mathscr M^\infty,
     \end{equation} where $p(t,x) = (2\pi t)^{-1/2}e^{-x^2/2t} \ind_{\{t\ge 0\}}$
     is the (deterministic) standard heat kernel, and $K$ is the operator defined in \eqref{e:kf}.

     Fix $\e>0$. Now we consider the $\e$-translates of these statements. It is automatic from Item (2) of Proposition \ref{mcts} that $\mathscr U_{N,\e}$ are tight in $C([0,T],C^{\alpha,\sig}(\mathbb R))$. Likewise, it is automatic from Item (1) that the family $\hat M_{N,\e} := \hat M_N(\e +\cdot)$ is tight in $C([0,T],C^{\alpha,\sig}(\mathbb R))$ so that from Lemma \ref{ds} we see that $D_N\hat M_{N,\e}$ are tight in $C^{\alpha-2,\sig}_\mathfrak s(\Psi_T).$ Let us now take a joint limit point $(U^\infty, \mathscr M^\infty, U^{\infty,\e}, \mathscr M^{\infty,\e})$ of the family $(\mathscr U_N, D_N\hat M_N, \mathscr U_{N,\e} , D_N\hat M_{N,\e}).$ On one hand we necessarily have $U^{\infty,\e}(t) = U^{\infty}(t+\e)$ and on the other hand we necessarily have that $U^\infty = p+K\mathscr M^\infty.$ From the last statement in Proposition \ref{sch}, this then forces the relation
     $$\mathscr U^{\infty,\e} = K(\delta_0\otimes U^\infty(\e,\cdot )) + K\mathscr M^{\infty,\e},$$ where the tensor product was defined in Definition \ref{otimes}. Now we notice $K(\delta_0\otimes  U^\infty(\e,\cdot ))= J \big(U^\infty(\e,\cdot)\big) $, where $J$ is defined in Corollary \ref{sme}. We thus have the Duhamel equation 
     \begin{equation}\label{e.xie2}\mathscr U^{\infty,\e} = J \big(U^\infty(\e,\cdot)\big) + K\mathscr M^{\infty,\e}.
     \end{equation}
For technical reasons that will be made clear below, we now replace $\e$ by $\e/2$ and $T$ by $T+1$. We now claim that for all $q>0$
\begin{align}
\label{e.tm2}
    \Ex[\|\mathscr M^{\infty,\e/2}\|_{C_\mathfrak s^{-\kappa-3/2,\sig}(\Psi_{T+1})}^q]<\infty.
\end{align} 

Let us now complete the proof of regularity assuming \eqref{e.tm2}.
\begin{itemize}[leftmargin=20pt]
\itemsep\setlength{0.5em}
    \item Using \eqref{e.tm2} and Proposition \ref{sch}, it follows that 
$$\mathbb E\bigg[\big\|K\mathscr M^{\infty,\e/2}\big\|_{C_\mathfrak s^{-\kappa+1/2,\sig}(\Psi_{T+1})}^q\bigg]<\infty.$$ This implies that the restriction of $K\mathscr M^{\infty,\e/2}$ to $[\e/2,T+\e/2]\times\mathbb R$ lies in $C_\mathfrak s^{-\kappa+1/2,\sig}(\Psi_{[\e/2,T+\e/2]})$.

\item Let $h:= \mathscr U_N(\e/2,\cdot)$. By Proposition \ref{mcts} Item (2) we have $h\in C^{-2-\kappa,\sig}$ almost surely. So from \eqref{e.schhat} with $\alpha:=-2-\kappa$ and $\beta := \frac12-\kappa,$ it follows that $\hat{\mathpzc P}_{\e/2}h \in C_\mathfrak s^{-\kappa+1/2,\sig}(\Psi_{T+1})$ almost surely. 
Since from \eqref{e.hatjrel}, we know $Jh(t+\e/2,x) = \hat{\mathpzc P}_{\e/2} h(t,x)$, we thus have that the restriction of $Jh$ to $[\e/2,T+\e/2]\times\mathbb R$ lies in $C_\mathfrak s^{-\kappa+1/2,\sig}(\Psi_{[\e/2,T+\e/2]})$.
\end{itemize}
Thanks to the above two bullet points and the relation \eqref{e.xie2}, we have showed that the restriction of $U^{\infty,\e/2}$ to $[\e/2,T+\e/2]\times\mathbb R$ lies in $C_\mathfrak s^{-\kappa+1/2,\sig}(\Psi_{[\e/2,T+\e/2]})$. This is equivalent to the fact that $U^{\infty,\e}$ lies in $C_\mathfrak s^{-\kappa+1/2,\sig}(\Psi_T)$. This completes the proof modulo \eqref{e.tm2}. 

\medskip

\noindent\textbf{Proof of \eqref{e.tm2}}. We shall show \eqref{e.tm2} with $\e/2$ and $T+1$ replaced by $\e$ and $T$ respectively. For $\alpha<0$, the derivative operator $\bar \partial_s : C([0,T],C^{\alpha,\sig}(\mathbb R)) \to C_\mathfrak s^{\alpha-2,\sig}(\Psi_T)$ which is defined by sending $(v_t)_{t\in[0,T]}$ to the distribution 
\begin{equation}\label{ds2}(\partial_s v,\varphi)_{L^2(\Psi_T)}:=v_T(\varphi(T,\cdot)) -\int_0^T v_t(\partial_t\varphi(t,\cdot))dt -v_0(\varphi(0,\cdot)),\end{equation} whenever $\varphi\in C_c^\infty(\Psi_T),$ is a bounded linear map. Indeed this operator is just the composition of the embedding map of Lemma \ref{embed} with the $\partial_s$ operator which we proved was bounded in Lemma \ref{ds}.

Let $\hat M_{N,\e}:= \hat M_N(\e+\cdot)-\hat M_N(\e)$ and $\hat Q_{N,\e}:= \hat Q_N(\e+\cdot)-\hat Q_N(\e)$. Let us consider any joint limit point $(Q^{\infty,\e}, M^{\infty,\e},\mathscr M^{\infty,\e})$ of $(\hat Q_{N,\e}, \hat M_{N,\e},D_N\hat M_{N,\e}$) in the space $C([0,T],C^{\alpha,\sig}(\mathbb R))\times C([0,T],C^{\alpha,\sig}(\mathbb R))\times C^{\alpha-2,\sig}(\Psi_T),$ where $\gamma<-1$. Then one may verify from the second bound in Lemma \ref{ds} that one necessarily has $\mathscr M^{\infty,\e} = \bar \partial_s M^{\infty,\e}$. Furthermore by Proposition \ref{mcts} one has martingality of $M_t^{\infty,\e}(\phi)$, with $\langle M^{\infty,\e}(\phi)\rangle = Q^{\infty,\e}(\phi^2)$.

Let us consider any smooth function $\varphi$ supported on the unit ball of $\mathbb R^2$ such that $\|\varphi\|_{C^1}\leq 1.$ Recall  $S_{(t,x)}^{\lambda}$ from \eqref{scale}, and notice that $\big(\partial_t S_{(t,x)}^\lambda \varphi\big)^2 \leq \lambda^{-10} \ind_{[t-\lambda^2,t+\lambda^2]\times [x-\lambda,x+\lambda]}$. We will now estimate the $q^{th}$ moments of $\mathscr M^{\infty,\e}(S_{(t,x)}^{\lambda}\varphi)=\bar \partial_s M^{\infty,\e}(S_{(t,x)}^{\lambda}\varphi)$. By making $\e$ smaller and $T$ larger we may simply ignore the boundary terms in \eqref{ds2}. Note that for fixed $s,t,x,\lambda$ the martingale $u \mapsto M_u^{\infty,\e} (\partial_tS_{(t,x)}^\lambda \varphi(s,\cdot))$ is constant outside of the interval $[t-\lambda^2,t+\lambda^2]$. Thus we may apply the inequalities of Minkowski and then Burkholder-Davis-Gundy in \eqref{ds2} to obtain 
		\begin{align*}
			\Ex[ |\mathscr M^{\infty,\e}(S_{(t,x)}^{\lambda}\varphi)|^q]^{1/q} & = \Ex[ |\bar \partial_s M^{\infty,\e}(S_{(t,x)}^{\lambda}\varphi)|^q]^{1/q}\\&\leq \int_{t-\lambda^2}^{t+\lambda^2} \mathbb E\big[\big| M_s^{\infty,\e} (\partial_tS_{(t,x)}^\lambda \varphi(s,\cdot)) \big|^q\big]^{1/q}ds \\ &\leq \Con_q \int_{t-\lambda^2}^{t+\lambda^2} \mathbb E \bigg[ \bigg( Q_{t+\lambda^2}^{\infty,\e} \big((\partial_tS_{(t,x)}^\lambda\varphi)^2 (s,\cdot)\big) - Q_{t-\lambda^2}^{\infty,\e} \big((\partial_tS_{(t,x)}^\lambda\varphi)^2 (s,\cdot)\big)\bigg)^{q/2}\bigg]^{1/q}\\ &\leq 2\Con_q\lambda^{-3} \mathbb E \bigg[ \bigg( Q_{t+\lambda^2}^{\infty,\e} \big(\ind_{[x-\lambda,x+\lambda]} \big) - Q_{t-\lambda^2}^{\infty,\e} \big(\ind_{[x-\lambda,x+\lambda]}\big)\bigg)^{q/2}\bigg]^{1/q}
		\end{align*}
where $\Con_q>0$. 
For any $\delta>0$, by the second bound in Proposition \ref{tight1} we find that the last expression may be bounded above by
	$$ 2\Con \lambda^{-3}  (2\lambda^2)^{1/2} \| \ind_{[x-\lambda,x+\lambda]}\|_{L^{1+\delta}(\mathbb R)}^{1/2}= \Con \lambda^{-\frac32 - \big[\frac{\delta}{2(1+\delta)}\big] }.$$
		Given any $\kappa>0$, by making $\delta=\delta(\kappa)$ close to $0$ and making $q=q(\kappa)$ large enough, we may then apply Lemma \ref{l:KC}\ref{KC} to conclude \eqref{e.tm2}. 
	\end{proof}

 \subsection{Identification of the limit points}\label{sec:iden}

	In this subsection, we identify the limit points as the solution of the stochastic heat equation \eqref{she} and thereby prove our weak convergence theorem: Theorem \ref{main}.

 \begin{lem}\label{5.1} 
Let $\mu$ be a random variable in $\mathcal S'(\mathbb R)$ such that for \textit{some} smooth even (deterministic) $\phi \in \mathcal S(\mathbb R)$ and some $\delta>0$ one has $$\sup_{a,\e}(1\wedge a^{1-\delta})\mathbb E[(\mu,\phi^a_\e)^2]<\infty$$ $$\limsup_{\e \to 0} \mathbb E[(\mu,\phi_\e^a)^2]=0, \text{ for all } a\in\mathbb R\backslash\{0\},$$ where $\phi_\e^a (x):= \e^{-1}\phi(\e^{-1} (x-a)).$ Then $(\mu,\psi)=0$ almost surely for all $\psi \in \mathcal S(\mathbb R)$.
\end{lem}

\begin{proof}
Define $\mu^\e(x):= \mu * \phi_\e(x)$ so that $(\mu,\phi^a_\e) = \mu^\e(a).$ Given some smooth $\psi:\mathbb R\to \mathbb R$ of compact support note that $(\mu^\e , \psi) \to (\mu,\psi)$ a.s. as $\e \to 0$. This is a purely deterministic statement. Thus it suffices to show that $(\mu^\e,\psi) \to 0$ in probability. To prove that, suppose that the support of $\psi$ is contained in $[-S,S]$ and note by Cauchy-Schwarz that $$|(\mu^\e, \psi)| =\bigg|\int_{\mathbb R} \mu^\e(a) \psi(a)da\bigg|\leq \bigg[ \int_{[-S,S]} \mu^\e(a)^2 da\bigg]^{1/2} \|\varphi\|_{L^2(\mathbb R)}, $$ so that by taking expectation and applying Jensen we find $$\mathbb E[|(\mu^\e, \psi)|] \leq \|\varphi\|_{L^2} \bigg[\int_{[-S,S]} \mathbb E[\mu^\e(a)^2]da \bigg]^{1/2}.$$ Since by assumption $\mathbb E[\mu^\e(a)^2]\leq Ca^{\delta-1}$ and $\mathbb E[\mu^\e(a)^2]\to 0$ as $\e\to 0$ the dominated convergence theorem now gives the result by letting $\e \to 0$ on the right side. This proves the claim for $\psi$ of compact support. For general $\psi \in \mathcal S(\mathbb R)$ we may find a sequence $\psi_n \to \psi$ in the topology of $\mathcal S(\mathbb R)$, with each $\psi_n$ compactly supported. Then $0=(\mu,\psi_n)\to (\mu,\psi)$ a.s. as $n\to\infty$. 
\end{proof}

\begin{thm}[Solving the martingale problem]\label{solving_mp}
		Consider the triple of processes $(\mathscr U_N,M_N,Q_N)_{t\ge 0}$, where $\mathscr U_N$, $M_N$, and $Q_N$ are defined in \eqref{field}, \eqref{m_field}, and \eqref{qfield} respectively. Fix $\alpha<-3, \gamma<-1,$ and $\sig>1.$ These triples are jointly tight in the space $$C\big([0,T]\;,\;C^{\alpha,\sig}(\mathbb R) \times C^{\gamma,\sig}(\mathbb R)\times C^{\alpha,\sig}(\mathbb R)\big).$$
		Consider any joint limit point $(U^\infty, M^\infty, Q^\infty)$. Then for any $s>0$, the process $(t,x)\mapsto U_{s+t}^\infty(x)$ is necessarily supported on the space $C_\mathfrak s^{-\kappa+1/2,\sig}( \Psi_T)$. Furthermore, $(M_t^\infty(\phi))_{t\ge 0}$ is a continuous martingale for all $\phi\in C_c^\infty(\mathbb R)$, and moreover for all $0< s\leq t< T$ one has the almost sure identities 
		\begin{align}\label{mp1}
			M_t^\infty(\phi)-M_s^\infty(\phi) &= \int_\mathbb R (U_t^\infty(x)-U_s^\infty(x))\phi(x)dx -\frac12 \int_s^t \int_\mathbb R U_u^\infty(x)\phi''(x)dx du\\\label{mp2}
			\langle M^\infty (\phi) \rangle_t &= Q_t^\infty(\phi^2) \\\label{mp3}
			Q_t^\infty(\phi)-Q_s^\infty(\phi) &= \frac{8\sigma^2}{1-4\sigma^2} \int_\mathbb R\int_s^t (U_u^\infty(x))^2 \phi(x)\,du\,dx.
		\end{align}
	\end{thm}

Before going into the proof, we remark that with some inspection it may be verified that all quantities make sense given the spaces they lie in, as long as we choose $\kappa<1/2$ to ensure that $U^\infty$ is a continuous function in space-time away from $t=0$. 

\begin{proof} Most parts of the following theorem are already established in previous propositions and lemmas. The tightness of the triple was shown in Proposition \ref{mcts}.  Consider any limit point $(U^\infty,M^\infty,Q^\infty).$ The fact that for any $s>0$, the process $(t,x)\mapsto U^\infty(s+t,x)$ is necessarily supported on the space $C_\mathfrak s^{-\kappa+1/2,\sig}( \Psi_T)$ was proved in Proposition \ref{reg}. From Proposition \ref{mcts} we have that for all $\phi\in C_c^\infty(\mathbb R)$ the process $(M_t^\infty(\phi))_{t\ge 0}$ is a martingale. \eqref{mp2} is already proven in Proposition \ref{mcts} as \eqref{e:mcts}. All we are left to show is \eqref{mp1} and \eqref{mp3}.

 \bigskip

\noindent\textbf{Proof of \eqref{mp1}.} 
  In \eqref{u=p+km} we obtained that $U^\infty = p +K\mathscr M^\infty$ where (just as we observed after \eqref{ds2}) one necessarily has $\mathscr M^\infty =\bar \partial_s M^\infty$. By disregarding the boundary terms implies that $$(\bar \partial_s M^\infty, \varphi) =((\partial_t - \tfrac12\partial_x^2)U^\infty, \varphi)$$ for all $\varphi$ of compact support contained in $[\e,T-\e]\times \mathbb R$ for some $\e>0$. 
  Since both $M^\infty$ and $U^\infty$ lie in spaces with strong enough topologies, taking $\varphi(u,x)$ to approach the function $(u,x) \mapsto \ind_{[s,t]}(u) \phi(x)$ in the above equation leads to \eqref{mp1}.
  
\bigskip

\noindent\textbf{Proof of \eqref{mp3}.} Fix any $0\le t\le T$ and let $\mu$ be a $\mathcal{S}'(\R)$ valued random variable defined as
\begin{align}
    (\mu,\phi):= Q_t^{\infty}(\phi)-\frac{8\sigma^2}{1-4\sigma^2}\int_{\R}\int_0^t (U_s^{\infty}(x))^2\phi(x)\,ds\,dx. \label{mu1}
\end{align}
We claim that $ (\mu,\phi)=0$ a.s. for all $\phi\in \mathcal{S}(\R).$ This will validate \eqref{mp3}. To verify this, let us define $\xi_{\e}^a(x):=\e^{-1}\xi(\e^{-1}(x-a))$ with $\xi(x):=\frac1{\sqrt\pi}e^{-x^2}$. Using the ``key estimate" \eqref{Qllim} of Proposition \ref{4.1}, one verifies that the assumptions of Lemma \ref{5.1} hold for $\mu$ with this family of mollifiers. 

To see why Lemma \ref{5.1} is applicable, first note that for fixed $\phi\in C_c^\infty(\mathbb R)$, the difference between the sum appearing in \eqref{Qllim} and integral over $[0,t]$ of the same quantity tends to zero in probability with respect to the topology of $C[0,T]$, since we proved tightness of $\mathscr U_N$ in a topology given by the norm of $C([0,T],C^{\alpha,\sig}(\mathbb R))$. Consequently that sum in \eqref{Qllim} converges in law, jointly with $\mathscr U_N$, to the integral appearing in \eqref{mu1}. Then an application of \eqref{Qllim} and \eqref{e:QXpolylog} shows that $(\mu,\phi)$ as defined by \eqref{mu1} satisfies the conditions of Lemma \ref{5.1} for any $\delta\in (0,1)$, since \eqref{e:QXpolylog} gives a polylogarithmic bound which is less than $\Con a^{-\delta}$ for arbitrary $\delta>0$. This is enough to complete the proof.
 \end{proof}			
	We now complete the proof of our main theorem, Theorem \ref{main}.

	\begin{proof}[Proof of Theorem \ref{main}] We continue with the notation and setup of Theorem \ref{solving_mp}. We have already established the tightness of $\mathscr U_N$ in Proposition \ref{mcts}. Consider any limit point $ U^\infty$ of $\mathscr U_N$. From the previous theorem, we already know that $(t,x) \mapsto U_{t+\e}^\infty(x)$ is a continuous function in space and time. From the three equations \eqref{mp1}, \eqref{mp2}, and \eqref{mp3} in Theorem \ref{solving_mp} it follows that the martingale problem for \eqref{she} is satisfied by any limit point $U^\infty$. We refer the reader to \cite[Proposition 4.11]{BG97} for the characterization of the law of \eqref{she} as the solution to this martingale problem. 
		
		The result there is only stated for continuous initial conditions, so what this really shows is that for any $\e>0$ the law of the continuous field $(t,x) \mapsto U^\infty(t+\e,x)$ is that of the solution of \eqref{she} with initial condition $U^\infty(\e,\cdot).$ Thus we still need to pin down the initial data as $\delta_0$, by showing that we can let $\e\to 0$ and see that the limit of $U^\infty(\e,\cdot)$ is equal to $\delta_0$ in some sense. 
		
		In \cite[Section 6]{Par19} there is a general approach to do exactly this. Specifically, it suffices to show as in Lemma 6.6 of that paper the two bounds 
		\begin{align}\label{delta} & \Ex[|U_t^\infty(x)|^r]^{2/r}\leq \Con \cdot t^{-1/2}p(t,x),\\ \label{delta1}
			& \Ex[ |U_t^\infty(x)-p(t,x)|^r]^{2/r} \leq \Con \cdot p(t,x),
		\end{align}where $r>1$ is arbitrary, $p$ is the standard heat kernel, and $\Con$ is independent of $t>0$ and $x\in\mathbb R$. Clearly, it suffices to show this when $r$ is an even integer. The solution of \eqref{she} with $\delta_0$ initial condition certainly satisfies this bound, and by the moment convergence result in Section \ref{mom}, we know any limit point $U^\infty$ must satisfy $\Ex[(\int_\mathbb R U_t^\infty(x)\phi(x)dx)^k] = \mE[(\int_\mathbb R \mathcal{U}_t(x)\phi(x)dx)^k]$ for all $k\in\mathbb N$ and $\phi\in C_c^\infty(\mathbb R)$, where $(t,x) \mapsto \mathcal{U}_t(x)$ solves \eqref{she} with $\delta_0$ initial data. From here we can conclude by letting $\phi \to \delta_x$ that $\Ex[U_t^\infty(x)^k] = \mE[ \mathcal{U}_t(x)^k]$ for all $k\in \mathbb N$ and all $x\in\mathbb R$. Consequently, we may immediately deduce \eqref{delta} and \eqref{delta1} by the corresponding bounds for $\mathcal{U}_t$. This completes the proof. 
	\end{proof}

 \subsection{Creation of independent noise in the limit}

In this subsection, we prove our independent noise result: Theorem \ref{xi2}.

\begin{defn}
    Define for $t\in N^{-1}\mathbb Z_{\ge 0}$ and $\phi\in C_c^\infty(\mathbb R)$ the field \begin{equation}\label{wfield}\mathpzc W_N(t,\phi):= N^{-3/4} \sum_{r=0}^{Nt} \sum_{x\in \mathbb Z-rN^{-1/4}} \phi(N^{-1/2}x) \eta_N(r,x)
    \end{equation}
    where $\eta_N({r,x}):= 2\omega_{r,x+rN^{-1/4}}-1,$ whenever $r-(x+rN^{-1/4})$ is even, and $\eta_N({r,x})=0$ whenever $r-(x+rN^{-1/4})$ is odd. We also define $\mathpzc W_N(t,\phi)$ by linear interpolation whenever $t\notin N^{-1}\mathbb Z_{\ge 0}$.
\end{defn}

Recall the linear operator $\partial_s$ from Remark \ref{d/dx}, which was shown to be bounded in Lemma \ref{ds}. We prove the following version of Theorem \ref{xi2}.

\begin{thm}\label{create}
    Fix $\alpha<-3,\tau>1$. Then the pair $(\mathpzc W_N, \mathscr U_N)$ is tight with respect to the topology of $C([0,T], C^{\alpha,\sig}(\mathbb R) \times C^{\alpha,\sig}(\mathbb R)).$ Furthermore any joint limit point $(\mathcal W,\mathcal U)$ is uniquely characterized as follows:
    \begin{enumerate}
        \item If we define $\xi_1:=(2\sigma^2)^{-1/2} \partial_s\mathcal W\in C_\mathfrak s^{\alpha-2}(\Psi_T)$ then $\xi_1$ is a standard space-time white noise.

        \item $\mathcal U$ is the solution to the Ito SPDE given by $$\partial_t \mathcal U = \partial_x^2 \mathcal U + \sqrt{8\sigma^2} \cdot \mathcal U \xi_1 + \sqrt{\frac{32\sigma^4}{1-4\sigma^2}}\cdot  \mathcal U \xi_2,$$ where $\xi_2$ is a standard space-time white noise independent of $\xi_1$.
    \end{enumerate}
\end{thm}

Note that $\mathpzc W_N$ is simply a time-integrated and linearly interpolated version of the rescaled noise field $\Xi_N$ from \eqref{xi_field}. The above theorem readily implies Theorem \ref{xi2}.


\begin{proof} We are going to consider the martingales $\mathpzc W_N(t,\phi)$ as defined in \eqref{wfield}, and we are going to study the cross-variations with the martingales $M_N(t,\phi)$ from \eqref{m_field}, both indexed by $t\in N^{-1}\mathbb Z_{\ge 0}$.

First, let us note that if $\phi\in C_c^\infty(\mathbb R)$ then the process $\mathpzc W_N(t,\phi)$ indexed by $t\in N^{-1}\mathbb Z_{\ge 0}$ has independent increments, where each increment has mean zero and has variance $$4\sigma^2 N^{-3/2} \sum_{x\in \mathbb Z-rN^{-1/4}} \phi(N^{-1/2}x)^2\ind_{\{r-x-rN^{-1/4}\equiv 0\pmod 2\}}.$$ From here (using BDG) the tightness of $\mathpzc W_N$ in the space $C([0,T], C^{\alpha,\sig}(\mathbb R))$ will follow easily from Lemma 
\ref{l:KC}\ref{KC2}. Since martingality is preserved by limit points as long as uniform integrability holds, one may verify that for any limit point $\mathcal W$, the processes $\mathcal W_t(\phi)$ and $\mathcal W_t(\phi)^2 - 2\sigma^2 \|\phi\|_{L^2}^2 t$ will be martingales in the canonical filtration of $C([0,T], C^{\alpha,\sig}(\mathbb R))$, for all $\phi\in C_c^\infty(\mathbb R)$. By Levy's criterion, $\mathcal W$ must therefore be distributed as $\sigma \sqrt{2}$ times a standard cylindrical Wiener process over $L^2(\mathbb R)$, that is, a mean-zero Gaussian process such that $\mathbb E[\mathcal W_t(\phi) \mathcal W_s(\psi)] = 2\sigma^2 (s\wedge t) (\phi,\psi)_{L^2(\mathbb R)}.$ 
     
  Next, with the martingales $M_N(\phi)$ from \eqref{m_field}, we note that the processes $$M_N(t,\phi)\mathpzc W_N(t,\phi) -\langle M_N(\phi),\mathpzc W_N(\phi)\rangle_t$$ are martingales indexed by $t\in N^{-1}\mathbb Z_{\ge 0}$, where the predictable quadratic variation is defined by
  \begin{align} \notag
      \langle M_N(\phi),\mathpzc W_N(\phi)\rangle_t &:= \sum_{r=1}^{Nt} \mathbb E \big[ (M_N(r,\phi)-M_N(r-1,\phi))(\mathpzc W_N(r,\phi) - \mathpzc W_N(r-1,\phi))|\mathcal F_{r-1}] \\ &= 4\sigma^2 N^{-3/4}\sum_{r=0}^{Nt-1} \sum_{x\in \mathbb Z-rN^{-1/4}} \phi(N^{-1/2}x) \cdot \big( \nabla_N\phi\big)(N^{-1/2}x) \cdot Z_N^\omega(r,x) .\label{prvr}
  \end{align}
  We used \eqref{grad_form} and the fact that the $\eta_N({r,x})$ are iid in the last equality.

  Let us consider a joint limit point $(\mathcal W, \mathscr M,\mathcal U)$ of the triple $(\mathpzc W_N,\hat M_N,\mathscr U_N)$ in the function space $C\big([0,T],C^{\alpha,\sig}(\mathbb R) ^3 \big)$, where we recall that $\hat M_N$ is the interpolated version of $M_N$ from Definition \ref{mqv}, and tightness of $\hat M_N,\mathscr U_N$ in this topology was proved in Proposition \ref{mcts}. 
  
  Since martingality is preserved by limit points as long as one has uniform integrability, we see that $\mathscr M_t(\phi)$ and $\mathcal W_t(\phi)$ are martingales in the canonical filtration of the space $C\big([0,T],C^{\alpha,\sig}(\mathbb R) ^3 \big)$. Recall from Section \ref{hopf} that $\big( \nabla_N\phi\big)(N^{-1/2}x) = N^{-1/4} \phi(N^{-1/2}x) + O(N^{-1/2})$. 
  Thus by \eqref{prvr} we see in the limit that \begin{equation}\label{mgale}\mathscr M_t(\phi)\mathcal W_t(\phi) - 4\sigma^2 \int_0^t \mathcal U_s(\phi^2)ds\end{equation}
  is also a martingale with respect to the canonical filtration. Indeed the difference between the sum in \eqref{prvr} and integral in \eqref{mgale} is easily controlled, since we proved tightness of $\mathscr U_N$ in a topology given by the norm of $C([0,T],C^{\alpha,\sig}(\mathbb R)).$
  On the other hand, we know by Theorem \ref{main} that $\mathcal U$ solves \eqref{she} for some driving noise $\xi$ which can be deterministically recovered from $\mathcal U$ by the following argument. By \eqref{mp1} we have that 
  \begin{align}\notag 
      \mathscr M_t(\phi) - \mathscr M_s(\phi) &= \int_\mathbb R (\mathcal U_t(x) - \mathcal U_s(x)) \phi(x) dx - \frac12 \int_s^t \int_\mathbb R \mathcal U_u(x) \phi''(x)dxdu \\ &= \sqrt{\frac{8\sigma^2}{1-4\sigma^2}}\int_s^t \mathcal U_u(x) \phi(x) \xi(du,dx), \label{eq3}
  \end{align}
     where the latter is an Ito-Walsh stochastic integral against the noise (which a priori may be defined on an enriched probability space), and the last equality is a standard result for the stochastic heat equation \eqref{she}, see \cite{Wal86}. Now we may define $$\mathcal B_t(\phi):= \sqrt{\frac{1-4\sigma^2}{8\sigma^2}} \int_0^t \mathcal U_u(x)^{-1} \phi(x)\mathscr M(du,dx)$$ where the latter should be understood as an Ito-Walsh stochastic integral against the orthomartingale $\mathscr M$, as defined in \cite{Wal86}. The above stochastic integral is well-defined by strict positivity of $\mathcal U$ \cite{mue91}. Then from Levy's criterion, it is clear that $\mathcal B_t$ is a standard cylindrical Wiener process (hence can be realized in $C([0,T],C^{\alpha,\sig}(\mathbb R))$, such that one has $\partial_s\mathcal B = \xi$. Thus $\xi$ has been deterministically recovered from $\mathcal U$.

     Define $\bar{\mathcal W}_t:=(\sigma\sqrt{2})^{-1}\mathcal W_t$, which is also a standard cylindrical Wiener process. Combining \eqref{eq3} with the fact that \eqref{mgale} is a martingale, one sees that $$\mathcal B_t(\phi)\bar{\mathcal W}_t(\phi)- (1-4\sigma^2)^{1/2} \|\phi\|_{L^2(\mathbb R)}^2t$$ is a martingale for all $\phi \in C_c^\infty(\mathbb R).$ By Levy's criterion, this is enough to imply that $(\mathcal B_t, \bar{\mathcal W}_t)_{t\ge 0}$ is a \textit{jointly} Gaussian process, such that both coordinates are standard cylindrical Wiener processes, and furthermore $\mathbb E[ \bar{\mathcal W}_t(\phi)\mathcal B_s(\phi)] = (s\wedge t) (1-4\sigma^2)^{1/2} \|\phi\|^2_{L^2}.$ By elementary algebraic manipulations of Gaussian variables, we may therefore write $\mathcal B = (1-4\sigma^2)^{1/2} \bar{\mathcal W} + 2\sigma \mathcal X$ where $\mathcal X= (\mathcal X_t)_{t\ge 0}$ is a standard cylindrical Wiener process independent of $\bar{\mathcal W}$. Now the theorem is proved, setting $\xi_1 = \partial_s \bar{\mathcal W}$ and $\xi_2 = \partial_s \mathcal X.$
\end{proof}

With the theorem proved, we now discuss some different interpretations of the above result, which will be more heuristic than rigorous.

The term $\sqrt{\tfrac{32\sigma^4}{1-4\sigma^2}}\cdot \xi_2$ appearing in the preceding theorem can be viewed as the contribution which causes the chaos expansion method to fail in the introduction, as this term precisely contains the $L^2$-mass which escapes into the tails of the chaos expansion. Indeed one expects that this piece which escapes will decouple and become independent of the remaining noise in the limit, just as we see in Theorem \ref{create}.

Rather than ``creation of independent noise," there is an alternative interpretation of the above theorem. In the proof we showed $\mathbb E[ \bar{\mathcal W}_t(\phi)\mathcal B_s(\phi)] = (s\wedge t) (1-4\sigma^2)^{1/2} (\phi,\psi)_{L^2(\mathbb R)}.$ Thus, an equivalent way of formulating Theorem \ref{create} is that $\xi_1,\xi$ are jointly Gaussian such that $$\mathbb E\big[(\xi_1,\varphi)_{L^2(\mathbb R^2)}(\xi,\varphi)_{L^2(\mathbb R^2)}\big]= (1-4\sigma^2)^{1/2} \|\varphi\|_{L^2(\mathbb R^2)}^2,$$ where $\xi$ is the driving noise of $\mathcal U$ as in \eqref{she}.  The above equality strongly suggests that the rescaled density field $\mathscr U_N$ from \eqref{field} has a tendency to concentrate on certain favored sites, such that macroscopically these favored sites have Lebesgue density $1-4\sigma^2$ on average. In the limit, the noise $\xi$ is generated by only the weight variables $\omega_{t,x}$ from these favored sites, in other words, the memory of only a proportion $1-4\sigma^2$ of the weights is remembered in the limit. The concentration of the field $\mathscr U_N$ inside these favored sites is the reason why we fail to see convergence in a space of continuous functions.

Thus we see that heuristically, the three pathological phenomena that we have observed in this paper all seem to be equivalent: the failure of the naive chaos expansion, the failure of Theorem \ref{main} to be strengthened to a topology of continuous functions, and the creation of independent noise in the limiting SPDE. It may be interesting to seek a more precise statement establishing a rigorous correspondence between these phenomena.

 \section{Results for the quenched tail field}
\label{sec7}

 \begin{defn}\label{qtf}
    Fix a realization of the environment variables $\omega$, as in Theorem \ref{main}. For $t\in N^{-1}\mathbb Z_{\ge 0}$ and $x\in \mathbb R$ we define the quenched tail field by $$F_N(t,x):= N^{1/4} C_{N,t,x} \mathsf{P}^\omega\big(R(Nt) \geq N^{3/4}t+N^{1/2}x\big).$$
\end{defn}
   
  The main result of this section will be that the family $\{\log F_N\}_{N\ge 1}$ of space-time processes converges to the KPZ equation, in the sense of finite-dimensional distributions of pointwise values $(t,x)$, thus confirming the physics prediction of \cite[Section 4]{bld}.

  \begin{lem}\label{A1} For sufficiently large $N$ and for all $t\in N^{-1}\mathbb Z_{\ge 0}, x\in \R$, we have
    $$\mathbb E[F_N(t,x)]  \leq 4(\pi t)^{-1/2}.$$ 
\end{lem}

\begin{proof}
    Let $\mathbf P=\mathbf P_{RW^{(1)}_\nu}$ denote the law of a simple symmetric random walk on $\mathbb Z$. One takes the annealed expectation over the quenched expectation in the above definition of $F_N$ to obtain 
    \begin{align*}\mathbb E[F_N(t,x)] &= N^{1/4}C_{N,t,x}\mathbf P(N^{-1/2}(S(Nt)-N^{3/4}t)\geq x) \\&= N^{1/4} \hat{\mathbf E}[ e^{-N^{-1/4}(S(Nt)-N^{3/4}t-N^{1/2}x)} \ind_{\{N^{-1/2}(S(Nt)-N^{3/4}t) \geq x\}}].
    \end{align*}
    where $\hat{\mathbf P}$ is a change of measure induced by the exponential martingale $ C_{N,t,S(Nt)-N^{3/4}t},$ which changes the annealed law of the increments of $S_i$ from the usual symmetric law $\frac12(\delta_1+\delta_{-1})$ to the new law given by $\frac1{2\cosh(N^{-1/4})}(e^{N^{-1/4}}\delta_1+ e^{-N^{-1/4}}\delta_{-1})$. Denote the new expectation as $\hat{\mathbf E}.$ To prove the inequality, we let $\hat p_N(r,y)=\hat{\mathbf P}(S(r)-N^{-1/4}r=y)$. Note that under the tilted measure $\hat{\mathbf P}$, the increments of the random walk have mean $N^{-1/4} + O(N^{-3/4}). $ Then the local central limit theorem implies that for sufficiently large $N$, $$\sup_{y\in \mathbb Z-N^{-1/4}r}\hat p_N(r,y) \leq 2(\pi r)^{-1/2}.$$Consequently by setting $r=Nt$ we see that
    \begin{align*}
        & N^{1/4} \hat{\mathbf E}[ e^{-N^{-1/4}(S(Nt)-N^{3/4}t-N^{1/2}x)} \ind_{\{N^{-1/2}(S(Nt)-N^{3/4}t) \geq x\}}] \\ &= N^{1/4} \sum_{a \geq N^{1/2}x} \hat p_N(Nt,a) e^{-N^{-1/4}(a-N^{1/2}x)} \\ &\leq 2(\pi Nt)^{-1/2} N^{1/4} \sum_{a\geq N^{1/2}x} e^{-N^{-1/4}(a-N^{1/2}x)}\\ &\leq 2(\pi t)^{-1/2} N^{-1/4} \sum_{b\ge 0} e^{-N^{-1/4} b} =  2(\pi t)^{-1/2} \frac{N^{-1/4}}{1-e^{-N^{-1/4}}},
    \end{align*}
    where the sums over $a$ should be understood as being over $a\in \mathbb Z-N^{-1/4}Nt$ with $a\geq N^{1/2}x$. The last quotient is always bounded above by 2, thus proving the claim.
\end{proof}

  \begin{prop}\label{hi} Fix $m\in \mathbb{N}$, and $\phi_1,\ldots,\phi_m\in C_c^{\infty}(\R)$. Consider sequences $t_{N,1},\ldots,t_{N,m} \in N^{-1}\mathbb Z_{\ge 0}$ such that $t_{N,i} \to t_i>0$ for all $1\le i \le m$ as $N\to\infty.$ Then
\begin{align}\label{qqt}
	\bigg( \int_\mathbb R \phi_i(x) F_N(t_{N,i},x) dx\bigg)_{i=1}^m \stackrel{d}{\to} \bigg( \int_\mathbb R \phi_i(x) \mathcal U_{t_i}(x) dx\bigg)_{i=1}^m.
\end{align}
Here $(t,x)\mapsto \mathcal U_t(x)$ is the solution of \eqref{she}.
  \end{prop}

\begin{proof}For $t\in N^{-1}\mathbb Z_{\ge 0}$ let us define a family of measures on $\mathbb R$ given by$$\mu_{N,t}:= e^{(N^{1/2} - N\log\cosh(N^{-1/4}))t } \sum_{x\in N^{-1/2}\mathbb Z-N^{1/4}t}\mathsf P^{\omega}(Nt, N^{3/4}t + N^{1/2}x)\cdot  \delta_x,$$
Then we notice that Definition \ref{qtf} is equivalent to $$F_N(t,x):= N^{1/4} e^{N^{1/4}x} \mu_{N,t}[x,\infty).$$
An application of Fubini's theorem gives $$\int_{-\infty}^x N^{1/4}e^{N^{1/4}u} \mu_{N,t}[u,\infty)du = \int_{-\infty}^x e^{N^{1/4}u}\mu_{N,t}(du) + e^{N^{1/4}x}\mu_{N,t}[x,\infty).$$  
In the sense of distributions, it is clear that 
        \begin{align*}
           & F_N(t,x) = \partial_x \bigg[\int_{-\infty}^x N^{1/4}e^{N^{1/4}u} \mu_{N,t}[u,\infty)du\bigg], \quad e^{N^{1/4}x}\mu_{N,t}(dx)=\partial_x\bigg[\int_{-\infty}^x e^{N^{1/4}u}\mu_{N,t}(du)\bigg].
        \end{align*}
Consequently, for all $t \in N^{-1}\mathbb Z_{\ge 0}$ and $\phi \in C_c^\infty(\mathbb R)$, we can repeatedly integrate by parts to obtain that 
        \begin{align*}
            \int_{\mathbb R}\phi(x) F_N(t,x)dx &= -\int_{\mathbb R} \phi'(x) \bigg[\int_{-\infty}^x N^{1/4}e^{N^{1/4}u} \mu_{N,t}[u,\infty)du\bigg]dx \\ &= -\int_{\mathbb R} \phi'(x)\bigg[\int_{-\infty}^x e^{N^{1/4}u}\mu_{N,t}(du)+ e^{N^{1/4}x}\mu_{N,t}[x,\infty)\bigg]dx \\ &= \int_{\mathbb R}\phi(x) e^{N^{1/4}x}\mu_{N,t}(dx) - \int_{\mathbb R} \phi'(x)e^{N^{1/4}x}\mu_{N,t}[x,\infty)dx.
        \end{align*}

Now let us prove the proposition. In the setting of the proposition statement, by Theorem \ref{main} we know that as $N\to \infty$ $$\bigg(\int_{\mathbb R}\phi_i(x) e^{N^{1/4}x}\mu_{N,t_{N,i}}(dx)\bigg)_{i=1}^m \stackrel{d}{\to} \bigg(\int_{ \R}\phi_i(x) \mathcal U_{t_i}(x)dx\bigg)_{i=1}^m.$$ 
Thus, to show \eqref{qqt} it suffices to show that as $N\to \infty$, $\int_{\mathbb R} \phi_i'(x)e^{N^{1/4}x}\mu_{N,t_{N,i}}[x,\infty)dx$ goes to zero in $L^1(\mathbb P)$  for each $1\le i\le m$. Indeed, observe that
    \begin{align*}
       \mathbb E\bigg|\int_{\mathbb R} \phi_i'(x)e^{N^{1/4}x}\mu_{N,t_{N,i}}[x,\infty)dx\bigg| & \leq \int_{\mathbb R} |\phi_i'(x)|e^{N^{1/4}x}\mathbb E[\mu_{N,t_{N,i}}[x,\infty)]dx. 
    \end{align*}
  By Lemma \ref{A1} we get that $$e^{N^{1/4}x}\mathbb E[\mu_{N,t_{N,i}}[x,\infty)]\leq 4N^{-1/4}(\pi t_{N,i})^{-1/2}$$ for all $x\in\mathbb R$. Since the $t_{N,i}$s remain bounded away from $0$ as $N\to \infty$, this implies the desired $L^1$ convergence and completes the proof.
  \end{proof}

  Our next lemma shows uniform convergence of the two-point correlation function of the quenched tail field. This is the crucial result that allows us to improve Proposition \ref{hi} to obtain multipoint convergence in law of $F_N(t,x)$ to $\mathcal U_t(x)$ for \textit{individual} values of $(t,x)\in (0,\infty)\times \mathbb R$.
  
  \begin{lem}\label{7.5}
  Suppose that $t_N\in N^{-1}\mathbb Z_{\ge 0}$ and $x_N,y_N\in \mathbb R$. Assume that as $N\to \infty$ we have $t_N\to t>0$ and $(x_N,y_N)\to (x,y)\in \mathbb R^2$. Then
  $$\lim_{N\to\infty} \mathbb E[F_N(t_N,x_N)F_N(t_N,y_N)] = \mathbb E[\mathcal U_t(x)\mathcal U_t(y)]$$
 where $\mathcal{U}$ is as in \eqref{she}. 
  \end{lem}

  \begin{proof} By taking an annealed expectation over the quenched one, notice that the two-point correlation function given by $(t,x,y)\mapsto \mathbb E[F_N(t,x)F_N(t,y)]$ can be written as an expression involving only the 2-point motion $(R^1,R^2)$. More specifically, in the notation of Section \ref{sec:girt}, it equals $$Ne^{2(N^{1/2} - N\log\cosh(N^{-1/4}))t}\mathbf P_{RW^{(2)}_{\nu}}(N^{-1/2}(R^1({Nt})-N^{3/4}t) \geq x,N^{-1/2}(R^2(Nt)-N^{3/4}t) \geq y).$$ Consequently the two-point correlation function only depends on the mean and variance of the weight measure $\nu$. Without any loss of generality we shall henceforth assume $\nu$ is the law of Beta$(\alpha,\alpha)$ where $\alpha:= \frac{1-4\sigma^2}{8\sigma^2}.$
  
  We shall now invoke results from \cite{bc} which are for the Beta measure. In \cite{bc}, the authors consider a slightly different field $\hat F_N(t,x)$ which is ``dual" to ours in a certain sense. More specifically, they vary the starting point $x$ and fix the tail probability as $[0,\infty)$, whereas we fix the starting point $0$ and vary the tail probability as $[x,\infty)$. One may show that the distributions of both fields are the same as space-time processes, modulo a reflection of the weight measure. More precisely, we have $F_N^\nu \stackrel{d}{=} \hat F_N^{\mu}$, where $\mu$ is the pushforward of $\nu$ by $x\mapsto 1-x$, and the superscript highlights the dependence of each field on the underlying weight measure.\footnote{In particular, our results imply the space-time multi-point convergence in law of their field $\hat F_N$ to \eqref{she} as well.} This is proved using an explicit coupling of the two fields via the ``discrete Brownian web". See \cite[Equation (1.8)]{yu} for a discussion of the coupling construction, which is in turn based on \cite[Section 3]{sss}.
  
  Note that for the case of Beta$(\alpha,\alpha)$ we have $\mu = \nu$. Proposition 3.4 in \cite{bc} provides exact moment formulas for the unscaled version of $F_N(t,x)$ in this case. 
  Taking the scalings into account (see e.g. \cite[(34) and (40)-(43)]{bld}) we then have for all $t\in N^{-1}\mathbb Z_{\ge 0}$ and $x_1,x_2 \in N^{-1/2}(\mathbb Z-N^{3/4}t)$
\begin{align*}
    & \mathbb E[ F_N(t,x_1)F_N(t,x_2)] \\& =N^{1/2}\oint_{\gamma_1} \oint_{\gamma_2} 
    \frac{z_1-z_2}{z_1-z_2-1}
    \prod_{j=1}^2C_{N,t,x_j}\bigg(\frac{2\alpha+z_j}{z_j}\bigg)^{\frac12(Nt-N^{\frac34}t -x_jN^{\frac12})+1}\bigg(\frac{\alpha+z_j}{2\alpha+z_j}\bigg)^{Nt}\frac1{2\alpha+z_j}
    \frac{dz_1}{2\pi i}\frac{dz_2}{2\pi i},
\end{align*}
 where $\gamma_1,\gamma_2$ are positively oriented closed curves around 0 such that $\gamma_1$ contains $\gamma_2+1$
and both contours exclude $-2\alpha$.

Now assume that we have a sequence $(t_N,x_N,y_N)$ converging to $(t,x,y)\in (0,\infty)\times \mathbb R^2$. 
We claim that the integral expression for $\mathbb E[ F_N(t_N,x_N)F_N(t_N,y_N)]$ converges as $N\to \infty$ to $$\alpha^{-2}\oint_{r_1+i\mathbb R} \oint_{r_2 + i \mathbb R} \frac{z_2-z_1}{z_2-z_1 - 1} e^{\frac{t}{2\alpha^2} (z_1^2+z_2^2)+\frac1\alpha (xz_1+yz_2)}\frac{dz_1}{2\pi i}\frac{dz_2}{2\pi i}.$$ Here $r_1,r_2$ are chosen such that $r_2>r_1 + 1.$ This convergence can be justified by dominated convergence and Taylor expansions, see \cite[Proposition 5.4.2]{bigmac} for a similar argument.  
This is known to agree with $\mathbf E[\mathcal U_t(x)\mathcal U_t(y)]$, see \cite[Section 6.2]{bigmac}, thus proving the lemma.
  \end{proof}
With the above lemma in place, we prove our main convergence result for the quenched tail field: Theorem \ref{unifconv}.

  \begin{proof}[Proof of Theorem \ref{unifconv}] We first prove the case when $x_{N,i}=x_i$ for $1\le i\le m$.   We give a proof for $m=1$ to simplify the notation, but the generalization to larger $m$ is straightforward. We will simply write $(t_{N,1},x_1)$ as $(t_N,x)$, which will be fixed for the moment being. By Lemma \ref{A1} we get that $\mathbb E[F_N(t_N,x)]\leq 2(\pi t_N)^{-1/2}$ which is bounded independently of $N$, thus it follows that the $\{F_N(t_N,x)\}_{N\ge 1}$ are tight.  Consider any limit point $\mu_0$ of the laws of $\{F_N(t_N,x)\}_{N\ge 1}$. Fix a smooth compactly supported nonnegative function $\phi:\mathbb R\to \mathbb R$ which integrates to $1$, and define $\phi^\lambda_x(y) = \lambda^{-1}\phi(\lambda^{-1}(y-x)).$ By Proposition \ref{hi} we know that for each $\lambda,x\in \R$, $\big\{\int_\mathbb R \phi^{\lambda}_x(y) F_N(t_N,y) dy\big\}_{N\ge 1}$ is a tight sequence. For $r\ge 1$, consider any measure $\mu_r$ on $\mathbb R^{r+1}$ which is a joint limit point as $N\to \infty$ of $$\bigg(F_N(t_N,x) \ , \  \big(\int_\mathbb R \phi^{2^{-k}}_x(y) F_N(t_N,y) dy\big)_{k=1}^r\bigg),$$ such that the first marginal of $\mu_r$ is $\mu_0$. Using a diagonalization argument, these measures $\mu_r$ may be chosen so that they form a projective family, and therefore by the Kolmogorov extension theorem, we may consider any projective limit $\mu$ as $r\to \infty$, which will be a measure on the space of sequences $(a_k)_{k\ge 0}\in \mathbb R^{\mathbb Z_{\ge 0}}$, equipped with the $\sigma$-algebra generated by the projection maps. By Proposition \ref{hi}, we find that for such a measure $\mu$ the marginal distribution of $(a_k)_{k\ge 1}$ is simply equal to the law of $\big( \int_\mathbb R \mathcal U_t(y)\phi^{2^{-k}}_x(y)dy\big)_{k\ge 1}$.  We now claim that  
  \begin{align}\label{e:fbds}
          \limsup_{\lambda\to 0}\limsup_{N\to\infty} \mathbb E \bigg| F_N(t_N,x) - \int_\mathbb R F_N(t_N,y) \phi^\lambda_x(y)dy\bigg|  = 0.
      \end{align}
  Assuming this fact, one finds that such a measure $\mu$ is necessarily supported on those sequences $(a_k)_{k\ge 0}$ which satisfy $a_0 = \lim_{k\to \infty} a_k$ in $L^1(\mu)$, which means that $a_0$ must indeed have the law of $\mathcal U_t(x)$ under $\mu$, proving the lemma. We are thus left to check \eqref{e:fbds}. 
   By Lemma \ref{7.5}, we have that
      $$\limsup_{N\to\infty} \mathbb E[ (F_N(t,x)-F_N(t,y))^2] = \mE[(\mathcal U_t(x)-\mathcal U_t(y))^2] \le \Con|x-y|$$
uniformly over compacts sets of $(t,x,y)\in (0,\infty)\times\R^2$ where the above inequality is a known estimate for \eqref{she} (see \cite[Proposition 2.4-nw]{das} for example). Thus, by Fatou's lemma, we note that 
      \begin{align*}
          \limsup_{N\to\infty} \mathbb E \bigg| F_N(t_N,x) - \int_\mathbb R F_N(t_N,y) \phi^\lambda_x(y)dy\bigg|  &\leq \limsup_{N\to\infty} \int_\mathbb R\mathbb E| F_N(t_N,x)-F_N(t_N,y)| \phi^\lambda_x(y)dy \\ &\leq  \int_\mathbb R \limsup_{N\to\infty} \mathbb E| F_N(t_N,x)-F_N(t_N,y)| \phi^\lambda_x(y)dy\\&\leq\int_\mathbb R \limsup_{N\to\infty}\mathbb E[ (F_N(t_N,x)-F_N(t_N,y))^2]^{1/2} \phi^\lambda_x(y)dy \\ &\leq \Con\int_\mathbb R |x-y|^{1/2} \phi^\lambda_x(y)dy =\Con\lambda^{1/2} \int_\mathbb R |u|^{1/2} \phi(u)du.
      \end{align*}
Taking $\lambda\to 0$ leads to \eqref{e:fbds}. To justify the above application of Fatou's lemma with $\limsup$, one needs a bound on the $y$-integrand which is independent of $N$ and is in $L^1(\phi^\lambda_x(y)dy)$. For this we may use Lemma \ref{A1} to get $$\mathbb E| F_N(t_N,x)-F_N(t_N,y)| \leq \mathbb E[ F_N(t_N,x)+F_N(t_N,y)] \leq 2(2\pi t_N)^{-1/2}.$$
The right side may be bounded by a constant independent of $N$ since $t_N\to t>0$. Finally, the case when $x_{N,i}$ varies with $N$ follows from the fact that as $N\to \infty$ we have $$\mathbb E\big[\big(F_N(t_{N,i},x_{N,i}) - F_N(t_{N,i},x_i)\big)^2 \big]\to 0$$ due to the uniform convergence in Lemma \ref{7.5}. This completes the proof.
      \end{proof}

      \begin{thm}[Extremal particle limit theorem] \label{t:max}
Fix $c,t>0$ and $d\in \mathbb R$. Let $(R^1(r),\ldots,R^k(r))_{r\ge 0}$ denote the canonical process sampled from the measure $\mathbf P_{RW_\nu^{(k)}}.$ Set the number of particles $k=k(N):= \lfloor \exp(\frac12cN^{1/2} +dN^{1/4}+r_N)\rfloor$ where $r_N$ can be any sequence satisfying $r_N=o(N^{1/4})$. Consider any sequence $t_N\in N^{-1}\mathbb Z_{\ge 0}$ such that $t_N\to t>0$ as $N\to \infty$. Then
\begin{align*}
    \max_{1\le i\le k(N)} \big\{N^{-\frac14} R^i(Nt_N)\big\} -a_N(c,d,t_N) \stackrel{d}{\to} \sqrt{\tfrac{t}{c}} \big( G+\log \mathcal{U}_{c}(d)\big),
\end{align*}
where
\begin{align*}
   a_N(c,d,s):= \sqrt{csN} - dN^{\frac14}\sqrt{\tfrac{s}{c}} -\sqrt{\tfrac{c}{s}} \big(r_N-\tfrac14 \log N\big)-\frac{c^{3/2}}{6\sqrt{s}}.
\end{align*}
Here $G$ is a standard Gumbel random variable which is independent of $\mathcal U$, and $(t,x)\mapsto \mathcal U_t(x)$ solves \eqref{she} but with noise coefficient replaced by $\sqrt{\frac{8\sigma^2c}{(1-4\sigma^2)t}}$.
 \end{thm}

      \begin{proof} 
      Fix any $a\in \R$.
 Set $v_N:=\sqrt{ct_N}$ and $x_N:=d\sqrt{\frac{t_N}{c}}$ so that $c=\frac{v_N^2}{t_N}$ and $d=v_Nx_N/t_N$. Also define $v=\sqrt{ct}$ and $x=d\sqrt{\frac{t}c}$. It suffices to show 
 \begin{equation}
 \label{gbel}
     \begin{aligned}
     & \lim_{N\to\infty}\mathbf P_{RW^{(k(N))}_\nu} \bigg(\max_{1\le i \le k(N)} \big\{R^i(Nt_N)\big\} - v_NN^{3/4} - x_N\sqrt{N} -\tfrac{tN^{1/4}}{v_N} \big(r_N-\tfrac14 \log N\big) \leq a N^{1/4}\bigg) \\ & \hspace{6cm}= \mathbb P\bigg(t/v[G+\log \mathcal U_{v^2/t}(vx/t)] \leq a+\frac{v^3}{6t^2}\bigg),
 \end{aligned}
 \end{equation}
  where $\mathcal U$ solves \eqref{she} with noise coefficient replaced by $\frac{v}t\sqrt{\frac{8\sigma^2}{1-4\sigma^2}}$, and $G$ is an independent Gumbel random variable.  

     \smallskip
     
     For now, fix any $k\in \mathbb N$ which may be arbitrary. Fix a realization of the environment $\omega = \{\omega_{i,j}: (i,j)\in \mathbb Z^2\}.$ As before, we let $\mathsf P_{(k)}^\omega$ be the quenched probability of $k$ random motions $(R^i(r))_{i=1}^k$ sampled given the environment. Let us consider the random variable 
     \begin{equation}\label{prob}\mathsf P_{(k)}^\omega \bigg(\max_{1\le i \le k} \big\{R^i(Nt_N)\big\} \leq v_N N^{3/4}+x_N\sqrt{N}+\frac{t_N}{v_N} N^{1/4}(r_N-\frac14 \log N) + aN^{1/4}\bigg). 
     \end{equation}
     Notice that until this point, we have assumed in Theorem \ref{main} that the parameter $N$ is a \textit{discrete} parameter taking values in $\mathbb N$. However nothing in our proof actually used this discreteness, and in fact, there is no loss of generality in assuming it is actually a \textit{continuous} parameter indexed by $[1,\infty)$ for instance. In this context, the results of Theorem \ref{main} and therefore Theorem \ref{unifconv} still hold as $N\to\infty$.
     
     Now we define $\mathfrak N = \mathfrak N(N):= N\big(\frac{t_N}{v_N}\big)^4$, so that $\mathfrak N\to \infty$ is equivalent to $N\to\infty$. Notice by Definition \ref{qtf} that \eqref{prob} is pathwise equal to the quantity 
     \begin{align}\label{disp}
        \bigg( 1-\mathfrak N^{-1/4}C_{\mathfrak N, \tfrac{v_N^4}{t_N^3}, \big(\tfrac{v_N}{t_N}\big)^2x_N+\mathfrak N^{-\frac14} (\tfrac{av_N}{t_N} +r_N -\tfrac14\log N)}^{-1}  F_{\mathfrak N}\big(\tfrac{v_N^4}{t_N^3}, \;\big(\tfrac{v_N}{t_N}\big)^2 x_N+ \mathfrak N^{-\frac14} (\tfrac{av_N}{t_N} +r_N -\tfrac14\log N)\big)\bigg)^k.
     \end{align}
     Then we may choose \begin{align*}k=k(N)&= \bigg\lfloor \mathfrak N^{1/4}C_{\mathfrak N, \tfrac{v_N^4}{t_N^3}, \big(\tfrac{v_N}{t_N}\big)^2x_N+\mathfrak N^{-\frac14} (\tfrac{av_N}{t_N} +r_N -\tfrac14\log N))}\exp\bigg[ - \frac{v_N^4}{6t_N^3} - a\frac{v_N}{t_N}-\log(t_N/v_N)\bigg]\bigg\rfloor \\&= \bigg\lfloor \exp\bigg[ \frac{v_N^2}{2t_N} \sqrt{N} +\frac{v_Nx_N}{t_N}N^{1/4}+r_N+O(N^{-1/2})\bigg]\bigg\rfloor
     \end{align*}where we applied a fourth order Taylor expansion of $\log \cosh$ in order to deduce the $O(N^{-1/2})$ term. Note that the latter expression agrees with the choice of $k(N)$ in the theorem statement. Then \eqref{disp} is a quantity of the form $(1-\frac{u(N)}{k(N)+O_N(1)})^{k(N)}$, where the $O_N(1)$ term is deterministic and bounded between $0$ and $1$, and $u(N)$ is a sequence of random variables which by Proposition \ref{unifconv} converges in law to the strictly positive random variable $\tfrac{v}{t}e^{-\frac{av}t-\frac{v^4}{6t^3} }\mathcal U_{v^4/t^3}(v^2x/t^2)$ with $\mathcal U$ solving \eqref{she}. The latter has the same law as $e^{-\frac{av}t-\frac{v^4}{6t^3} }\mathcal U_{v^2/t}(vx/t)$ where now the noise coefficient in \eqref{she} has been replaced by $\frac{v}t\sqrt{\frac{8\sigma^2}{1-4\sigma^2}}$. Since the functions $u\mapsto (1-\frac{u}{k(N)+O_N(1)})^{k(N)}$ converge uniformly to $u\mapsto e^{-u}$ on compact subsets of $[0,\infty)$, we can conclude that the expression in \eqref{disp} therefore converges in law to the random variable $$\exp \big[ -e^{-\frac{av}t-\frac{v^4}{6t^3}}\mathcal U_{v^2/t}(vx/t)\big],$$ where $\mathcal U$ solves \eqref{she} with noise coefficient replaced by  $\frac{v}t\sqrt{\frac{8\sigma^2}{1-4\sigma^2}}$. Now this convergence in law also implies convergence of the associated annealed expectations thanks to the boundedness of all quantities involved, which implies that the limit of the left side of \eqref{gbel} equals $$\mathbb E\bigg[ -e^{-\frac{av}t-\frac{v^4}{6t^3}}\mathcal U_{v^2/t}(vx/t)\bigg],$$ which agrees with the right-hand side of \eqref{gbel} by the explicit form of the Gumbel distribution.
 \end{proof}

Note that in the above theorem, a natural choice is $t_N:=N^{-1}\lfloor Nt\rfloor$ where $t>0$ is fixed. However, we also have the liberty to set $t_N:=N^{-1}\lfloor Nt + N^\alpha \rfloor,$ where $\alpha \in (0,1)$, and this would have a nontrivial effect in the recentering coefficients $a_N(c,d,t_N)$ if $\alpha \in [1/2,1)$. 

\section*{Glossary}

{\renewcommand{\arraystretch}{1.2}
	\begin{longtable}[t]{lp{0.6\textwidth}l}
		\toprule
		\multicolumn{3}{l}{Processes used in text} \\
		\midrule
  $\mathcal{H}$ & KPZ equation & Eq.~\eqref{def:kpz} \\
  $\mathcal{U}$ & Stochastic heat equation (SHE) & Eq.~\eqref{she0} \\
  $\mathsf{P}^{\omega}(t,x)$ & Quenched transition probability of the random walk & Eq.~\eqref{pwtx} \\
  $C_{N,t,x}$ & Rescaling constants for the moderate deviation scaling & Eq.~\eqref{cntx} \\
		$\mathbf{R}$ & The canonical process on $(\mathbb{Z}^{k})^{\mathbb Z_{\ge 0}}$ & Sec.~\ref{sec:girt} \\
      $V^{ij}$ & Intersection process of $i^{th}$ and $j^{th}$ coordinate processes & Eq.~\eqref{vijai} \\ $\Gamma^{(v)}$ & Set of times with $v$ distinct particles & Eq.~\eqref{def:overlap}\\
      $\mathcal M$ & Martingale in the Tanaka formula for 2-point motion & Eq.~\eqref{m} \\
      $\mathpzc M^{\lambda}$ & Exponential martingale used to tilt path measure $\mathbf{P}_{RW_\nu^{(k)}}$ & Eq.~\eqref{m_n} \\
      $M_N(t,\phi)$ & Martingales arising in discrete SHE & Eq.~\eqref{m_field} \\
      $\mathscr{U}_N(t,\phi)$ & Rescaled field converging to $\mathcal U$ in main results & Eq.~\eqref{field} \\ $Q_N(t,\phi)$ & Quadratic martingale field & Eq.~\eqref{qfield}\\
  $\Xi_N$ & Prelimiting noise field & Eq.~\eqref{xi_field} \\ $\mathpzc W_N$ & Time-integrated and linearly interpolated version of $\Xi_N$ & Eq.~\eqref{wfield}\\
		\midrule	
		\multicolumn{3}{l}{Probability measures and filtrations} \\
		\midrule
 $\nu$ & Probability measure on $[0,1]$ representing law of each $\omega_{t,x}$ & Sec.~\ref{sec:girt} \\ $\mu_\nu$ & Mean of $\nu$ & Sec.~\ref{sec:girt} \\ $\sigma_\nu^2$ & Variance of $\nu$ & Sec.~\ref{sec:girt}\\
  $\mathbf{P}_{RW_\nu^{(k)}}$ & Probability measure on the path space $(\mathbb{Z}^{k})^{\mathbb Z_{\ge 0}},$ used to denote the annealed law of $k$ independent walks sampled from the environment $\omega$ whose weights are distributed as $\nu$, or more precisely $\mathbf{P}_{RW_\nu^{(k)}} = \mathbb P\mathsf P^\omega_{(k)} $ & Def.~\ref{def:prw} \\ $\mathsf P^\omega_{(k)}$ & Probability measure on the path space $(\mathbb{Z}^{k})^{\mathbb Z_{\ge 0}}$, used to denote the quenched law of $k$ independent walks sampled from the fixed realization of environment $\omega=\{\omega_{t,x}\}_{t,x}$\\
  $\mathbb P$ & Probability measure on the environment space $[0,1]^{\mathbb Z_{\ge 0}\times \mathbb Z}$, used to denote the annealed probability associated to any observable on the probability space of the environment $\{\omega_{t,x}\}_{t,x}$\\ $\mathbf P_{\text{lim}}$ and $\mathbf P_\infty$& Notation used for limit points of objects on the canonical space of continuous paths (or tuples of continuous paths) \\ $(\mathcal F_r)_{r\in \mathbb Z_{\ge 0}}$ & Canonical filtration on the path space $(\mathbb Z^k)^{\mathbb Z_{\ge 0}}$ & Def.~\ref{q} \\ $(\mathcal F_r^\omega)_{r\in \mathbb Z_{\ge 0}}$ & Canonical filtration on the environment space $[0,1]^{\mathbb Z_{\ge 0}\times \mathbb Z}$ & Eq.~\eqref{mfield}\\
\bottomrule
	\end{longtable}
}

\bibliographystyle{alpha}
\bibliography{ref.bib}

\begin{thebibliography}{HCGCC23}

\bibitem[AKQ14]{akq}
Tom Alberts, Konstantin Khanin, and Jeremy Quastel.
\newblock {The intermediate disorder regime for directed polymers in dimension $1+1$}.
\newblock {\em The Annals of Probability}, 42(3):1212 -- 1256, 2014.

\bibitem[BC95]{BC95}
Lorenzo Bertini and Nicoletta Cancrini.
\newblock The stochastic heat equation: {F}eynman-{K}ac formula and intermittence.
\newblock {\em J. Statist. Phys.}, 78(5-6):1377--1401, 1995.

\bibitem[BC14]{bigmac}
Alexei Borodin and Ivan Corwin.
\newblock Macdonald processes.
\newblock {\em Probability Theory and Related Fields}, 158(1-2):225--400, 2014.

\bibitem[BC17]{bc}
Guillaume Barraquand and Ivan Corwin.
\newblock Random-walk in beta-distributed random environment.
\newblock {\em Probability Theory and Related Fields}, 167(3-4):1057--1116, 2017.

\bibitem[BG97]{BG97}
Lorenzo Bertini and Giambattista Giacomin.
\newblock Stochastic {B}urgers and {KPZ} equations from particle systems.
\newblock {\em Comm. Math. Phys.}, 183(3):571--607, 1997.

\bibitem[BLD20]{bld}
Guillaume Barraquand and Pierre Le~Doussal.
\newblock Moderate deviations for diffusion in time dependent random media.
\newblock {\em Journal of Physics A: Mathematical and Theoretical}, 53(21):215002, 2020.

\bibitem[BR20]{mark}
Guillaume Barraquand and Mark Rychnovsky.
\newblock Large deviations for sticky {B}rownian motions.
\newblock {\em Electronic Journal of Probability}, 25, 2020.

\bibitem[BRAS06]{timo2}
M{\'{a}}rton Bal{\'{a}}zs, Firas Rassoul-Agha, and Timo Seppäläinen.
\newblock The random average process and random walk in a space-time random environment in one dimension.
\newblock {\em Communications in Mathematical Physics}, 266(2):499--545, may 2006.

\bibitem[BW22]{dom}
Dom Brockington and Jon Warren.
\newblock At the edge of a cloud of {B}rownian particles.
\newblock {\em arXiv preprint arXiv:2208.11952}, 2022.

\bibitem[CC22]{car}
Francesco Caravenna and Francesca Cottini.
\newblock {G}aussian limits for subcritical chaos.
\newblock {\em Electronic Journal of Probability}, 27:1--35, 2022.

\bibitem[CET21]{u6}
Giuseppe Cannizzaro, Dirk Erhard, and Fabio Toninelli.
\newblock Weak coupling limit of the anisotropic {KPZ} equation.
\newblock {\em arXiv preprint arXiv:2108.09046}, 2021.

\bibitem[CG17]{gu}
Ivan Corwin and Yu~Gu.
\newblock {Kardar--Parisi--Zhang equation and large deviations for random walks in weak random environments}.
\newblock {\em Journal of Statistical Physics}, 166:150--168, 2017.

\bibitem[Cor12]{Cor12}
Ivan Corwin.
\newblock The {K}ardar{--P}arisi{--Z}hang equation and universality class.
\newblock {\em Random Matrices: Theory Appl.}, 1(01):1130001, 2012.

\bibitem[CS20]{CS20}
Ivan Corwin and Hao Shen.
\newblock Some recent progress in singular stochastic partial differential equations.
\newblock {\em Bulletin of the American Mathematical Society}, 57(3):409--454, 2020.

\bibitem[CSZ16]{poly}
Francesco Caravenna, Rongfeng Sun, and Nikos Zygouras.
\newblock Polynomial chaos and scaling limits of disordered systems.
\newblock {\em Journal of the European Mathematical Society}, 19(1):1--65, 2016.

\bibitem[CSZ17]{marginal}
Francesco Caravenna, Rongfeng Sun, and Nikos Zygouras.
\newblock {Universality in marginally relevant disordered systems}.
\newblock {\em The Annals of Applied Probability}, 27(5):3050 -- 3112, 2017.

\bibitem[CSZ20]{u5}
Francesco Caravenna, Rongfeng Sun, and Nikos Zygouras.
\newblock The two-dimensional {KPZ} equation in the entire subcritical regime.
\newblock {\em Ann. Probab.}, 48(3), 2020.

\bibitem[CSZ23]{csz_2d}
Francesco Caravenna, Rongfeng Sun, and Nikos Zygouras.
\newblock The critical 2d stochastic heat flow, 2023.

\bibitem[CW17]{CW17}
Ajay Chandra and Hendrik Weber.
\newblock Stochastic {PDE}s, regularity structures, and interacting particle systems.
\newblock In {\em Annales de la facult{\'e} des sciences de Toulouse Math{\'e}matiques}, volume 26(4), pages 847--909, 2017.

\bibitem[Das22]{das}
Sayan Das.
\newblock Temporal increments of the {KPZ} equation with general initial data.
\newblock {\em arXiv preprint arXiv:2203.00666}, 2022.

\bibitem[DDP24]{DDP23}
Sayan Das, Hindy Drillick, and Shalin Parekh.
\newblock {KPZ} equation limit of sticky {B}rownian motion.
\newblock {\em Journal of Functional Analysis}, 287(10):110609, 2024.

\bibitem[DG22]{ew6}
Alexander Dunlap and Yu~Gu.
\newblock A quenched local limit theorem for stochastic flows.
\newblock {\em Journal of Functional Analysis}, 282(6):109372, 2022.

\bibitem[DGRZ20]{u2}
Alexander Dunlap, Yu~Gu, Lenya Ryzhik, and Ofer Zeitouni.
\newblock Fluctuations of the solutions to the {KPZ} equation in dimensions three and higher.
\newblock {\em Probability Theory and Related Fields}, 176(3-4):1217--1258, 2020.

\bibitem[DOV22]{dov}
Duncan Dauvergne, Janosch Ortmann, and Balint Virag.
\newblock The directed landscape.
\newblock {\em Acta Mathematica}, 229(2), 2022.

\bibitem[DT16]{dembo}
Amir Dembo and Li-Cheng Tsai.
\newblock {Weakly asymmetric non-simple exclusion process and the Kardar--Parisi--Zhang equation}.
\newblock {\em Communications in Mathematical Physics}, 341:219--261, 2016.

\bibitem[Flo14]{flo}
Gregorio R~Moreno Flores.
\newblock On the (strict) positivity of solutions of the stochastic heat equation.
\newblock {\em The Annals of Probability}, pages 1635--1643, 2014.

\bibitem[FS10]{FS10}
Patrik~L Ferrari and Herbert Spohn.
\newblock Random growth models.
\newblock {\em arXiv:1003.0881}, 2010.

\bibitem[GH04]{gaw}
Krzysztof Gawedzki and P{\'e}ter Horvai.
\newblock Sticky behavior of fluid particles in the compressible {K}raichnan model.
\newblock {\em Journal of statistical physics}, 116:1247--1300, 2004.

\bibitem[GIP15]{GIP15}
Massimiliano Gubinelli, Peter Imkeller, and Nicolas Perkowski.
\newblock Paracontrolled distributions and singular {PDE}s.
\newblock In {\em Forum of Mathematics, Pi}, volume~3. Cambridge University Press, 2015.

\bibitem[GJ14]{GJ14}
Patr{\'\i}cia Gon{\c{c}}alves and Milton Jara.
\newblock Nonlinear fluctuations of weakly asymmetric interacting particle systems.
\newblock {\em Arch. Ration. Mech. Anal.}, 212(2):597--644, 2014.

\bibitem[GP17]{GP17}
Massimiliano Gubinelli and Nicolas Perkowski.
\newblock {KPZ} reloaded.
\newblock {\em Commun. Math. Phys.}, 349(1):165--269, 2017.

\bibitem[GP18]{GP18}
Massimiliano Gubinelli and Nicolas Perkowski.
\newblock Energy solutions of {KPZ} are unique.
\newblock {\em J. Amer. Math. Soc.}, 31(2):427--471, 2018.

\bibitem[GRZ18]{u1}
Yu~Gu, Lenya Ryzhik, and Ofer Zeitouni.
\newblock {The Edwards--Wilkinson limit of the random heat equation in dimensions three and higher}.
\newblock {\em Communications in Mathematical Physics}, 363:351--388, 2018.

\bibitem[Gu20]{u4}
Yu~Gu.
\newblock Gaussian fluctuations from the {2D KPZ} equation.
\newblock {\em Stochastics and Partial Differential Equations: Analysis and Computations}, 8:150--185, 2020.

\bibitem[Hai13]{Hai13}
Martin Hairer.
\newblock Solving the {KPZ} equation.
\newblock {\em Annals of Mathematics}, pages 559--664, 2013.

\bibitem[Hai14]{Hai14}
Martin Hairer.
\newblock A theory of regularity structures.
\newblock {\em Invent Math.}, 198(2):269--504, 2014.

\bibitem[HCC23]{hass2023b}
Jacob~B Hass, Ivan Corwin, and Eric~I Corwin.
\newblock First passage time for many particle diffusion in space-time random environments.
\newblock {\em arXiv preprint arXiv:2308.01267}, 2023.

\bibitem[HCGCC23]{hass23}
Jacob~B Hass, Aileen~N Carroll-Godfrey, Ivan Corwin, and Eric~I Corwin.
\newblock Anomalous fluctuations of extremes in many-particle diffusion.
\newblock {\em Physical Review E}, 107(2):L022101, 2023.

\bibitem[HKLD23]{ldb}
Alexander~K Hartmann, Alexandre Krajenbrink, and Pierre Le~Doussal.
\newblock Probing the large deviations for the {B}eta random walk in random medium.
\newblock {\em arXiv preprint arXiv:2307.15041}, 2023.

\bibitem[HL15]{HL16}
Martin Hairer and Cyril Labb{\'e}.
\newblock A simple construction of the continuum parabolic {A}nderson model on $\mathbb{R}^2$.
\newblock {\em Electronic Communications in Probability}, 20:1--11, 2015.

\bibitem[HW09a]{hw09}
Chris Howitt and Jon Warren.
\newblock Consistent families of {B}rownian motions and stochastic flows of kernels.
\newblock {\em The Annals of Probability}, 37(4), jul 2009.

\bibitem[HW09b]{hw09b}
Chris Howitt and Jon Warren.
\newblock Dynamics for the {B}rownian web and the erosion flow.
\newblock {\em Stochastic Processes and their Applications}, 119(6):2028--2051, 2009.

\bibitem[JRAS19]{timo}
Mathew Joseph, Firas Rassoul-Agha, and Timo Sepp\"{a}l\"{a}inen.
\newblock Independent particles in a dynamical random environment.
\newblock In {\em Probability and analysis in interacting physical systems}, volume 283 of {\em Springer Proc. Math. Stat.}, pages 75--121. Springer, Cham, 2019.

\bibitem[KMT76]{KMT}
J.~Koml\'{o}s, P.~Major, and G.~Tusn\'{a}dy.
\newblock An approximation of partial sums of independent {RV}'s, and the sample {DF}. {II}.
\newblock {\em Z. Wahrscheinlichkeitstheorie und Verw. Gebiete}, 34(1):33--58, 1976.

\bibitem[KPZ86]{kpz}
Mehran Kardar, Giorgio Parisi, and Yi-Cheng Zhang.
\newblock Dynamic scaling of growing interfaces.
\newblock {\em Physical Review Letters}, 56(9):889, 1986.

\bibitem[KS88]{konno}
Nuri Konno and Tokuzo Shiga.
\newblock Stochastic partial differential equations for some measure-valued diffusions.
\newblock {\em Probability theory and related fields}, 79(2):201--225, 1988.

\bibitem[LD23]{lda}
Pierre Le~Doussal.
\newblock Dynamics at the edge for independent diffusing particles.
\newblock {\em arXiv preprint arXiv:2308.16709}, 2023.

\bibitem[LDT17]{ldt}
Pierre Le~Doussal and Thimoth{\'e}e Thiery.
\newblock {Diffusion in time-dependent random media and the Kardar-Parisi-Zhang equation}.
\newblock {\em Physical Review E}, 96(1):010102, 2017.

\bibitem[LJR04]{lejan}
Yves Le~Jan and Olivier Raimond.
\newblock Flows, coalescence and noise.
\newblock {\em Ann. Probab.}, 32(2):1247--1315, 2004.

\bibitem[MQR21]{MQR}
Konstantin Matetski, Jeremy Quastel, and Daniel Remenik.
\newblock The {KPZ} fixed point.
\newblock {\em Acta Mathematica}, 227(1):115--203, 2021.

\bibitem[MU18]{u3}
Jacques Magnen and J{\'e}r{\'e}mie Unterberger.
\newblock The scaling limit of the {KPZ} equation in space dimension 3 and higher.
\newblock {\em Journal of Statistical Physics}, 171:543--598, 2018.

\bibitem[Mue91]{mue91}
Carl Mueller.
\newblock On the support of solutions to the heat equation with noise.
\newblock {\em Stochastics: An International Journal of Probability and Stochastic Processes}, 37(4):225--245, 1991.

\bibitem[MW17]{WM}
Jean-Christophe Mourrat and Hendrik Weber.
\newblock Global well-posedness of the dynamic $\phi^{4}$ model in the plane.
\newblock {\em The Annals of Probability}, 45(4):2398--2476, 2017.

\bibitem[Par18]{Par19}
Shalin Parekh.
\newblock The {KPZ} limit of {ASEP} with boundary.
\newblock {\em Communications in Mathematical Physics}, 365(2):569--649, sep 2018.

\bibitem[QS15]{QS15}
Jeremy Quastel and Herbert Spohn.
\newblock The one-dimensional {KPZ} equation and its universality class.
\newblock {\em J. Stat. Phys.}, 160(4):965--984, 2015.

\bibitem[Qua11]{Qua11}
Jeremy Quastel.
\newblock Introduction to {KPZ}.
\newblock {\em Current developments in mathematics}, 2011(1), 2011.

\bibitem[SSS09]{sss0}
Emmanuel Schertzer, Rongfeng Sun, and Jan Swart.
\newblock {Special points of the Brownian net}.
\newblock {\em Electronic Journal of Probability}, 14(none):805 -- 864, 2009.

\bibitem[SSS14]{sss}
Emmanuel Schertzer, Rongfeng Sun, and Jan Swart.
\newblock Stochastic flows in the {B}rownian web and net.
\newblock {\em Mem. Amer. Math. Soc.}, 227(1065):vi+160, 2014.

\bibitem[SSS17]{sss2}
Emmanuel Schertzer, Rongfeng Sun, and Jan Swart.
\newblock {The Brownian web, the Brownian net, and their universality}.
\newblock {\em Advances in disordered systems, random processes and some applications}, pages 270--368, 2017.

\bibitem[TLD16]{ldt2}
Thimoth{\'e}e Thiery and Pierre Le~Doussal.
\newblock Exact solution for a random walk in a time-dependent 1d random environment: the point-to-point beta polymer.
\newblock {\em Journal of Physics A: Mathematical and Theoretical}, 50(4):045001, dec 2016.

\bibitem[Wal86]{Wal86}
John~B Walsh.
\newblock An introduction to stochastic partial differential equations.
\newblock In {\em {\'E}cole d'{\'E}t{\'e} de Probabilit{\'e}s de Saint Flour XIV-1984}, pages 265--439. Springer, 1986.

\bibitem[Yan22]{yang22}
Kevin Yang.
\newblock {KPZ equation from non-simple variations on open ASEP}.
\newblock {\em Probability Theory and Related Fields}, 183(1-2):415--545, 2022.

\bibitem[Yan23]{yang23}
Kevin Yang.
\newblock {Kardar--Parisi--Zhang equation from long-range exclusion processes}.
\newblock {\em Communications in Mathematical Physics}, pages 1--129, 2023.

\bibitem[Yu16]{yu}
Jinjiong Yu.
\newblock {Edwards-Wilkinson fluctuations in the Howitt-Warren flows}.
\newblock {\em Stochastic Processes and their Applications}, 126(3):948--982, 2016.

\end{thebibliography}

\end{document}